\documentclass[11pt]{amsart}

\usepackage{amsmath} %AMS-Latex
\usepackage{amssymb} %AMSsymbol
\usepackage{amscd} 
\usepackage{amsxtra}
\usepackage{epic,eepic}
\usepackage[margin=3cm]{geometry} 

\usepackage[hypertex]{hyperref} 
\font\rsfs=rsfs10
\newcommand{\marsfs}[1]{\mbox{\rsfs #1}}

\newcommand{\C}{\mathbb C}
\newcommand{\R}{\mathbb R}
\newcommand{\Z}{\mathbb Z}
\newcommand{\N}{\mathbb N} 
\newcommand{\Q}{\mathbb Q}
\newcommand{\T}{\mathbb T}

\newcommand{\LL}{\mathbb L}
\newcommand{\Proj}{\mathbb P}
\newcommand{\PROJ}{\operatorname{Proj}}
\newcommand{\K}{\mathbb K} 
\newcommand{\Nabla}{\boldsymbol{\nabla}} 

\newcommand{\ev}{\operatorname{ev}}
\newcommand{\age}{\operatorname{age}}
 
\newcommand{\Hom}{\operatorname{Hom}}
\newcommand{\End}{\operatorname{End}}
\newcommand{\Ext}{\operatorname{Ext}}
\newcommand{\Pic}{\operatorname{Pic}}
\newcommand{\Ker}{\operatorname{Ker}}

\newcommand{\Image}{\operatorname{Im}}
\newcommand{\Boxop}{\operatorname{Box}}

\newcommand{\id}{\operatorname{id}}
\newcommand{\Res}{\operatorname{Res}}
\newcommand{\rank}{\operatorname{rank}}
\newcommand{\unit}{\operatorname{\boldsymbol{1}}}

\newcommand{\Eff}{\operatorname{Eff}}
  
\newcommand{\Mir}{\operatorname{Mir}} 
\newcommand{\ch}{\operatorname{ch}} 
\newcommand{\tch}{\widetilde{\operatorname{ch}}} 
 
\newcommand{\Td}{\operatorname{Td}}

\newcommand{\pr}{\operatorname{pr}} 
\newcommand{\Vol}{\operatorname{Vol}} 
\newcommand{\Aut}{\operatorname{Aut}} 
\newcommand{\inv}{\operatorname{inv}}

\newcommand{\Spec}{\operatorname{Spec}} 
\newcommand{\NE}{\operatorname{NE}} 
\newcommand{\hNE}{\operatorname{\widehat{NE}}}
\newcommand{\Jump}{\operatorname{\boldsymbol{\Delta}}} 
\newcommand{\Osc}{\operatorname{Osc}} 
\newcommand{\PD}{\operatorname{PD}}
\newcommand{\VC}{\operatorname{VC}}

\newcommand{\bN}{\mathbf{N}} 
\newcommand{\hbN}{\widehat{\bN}}
\newcommand{\bM}{\mathbf{M}}
\newcommand{\hbM}{\widehat{\bM}} 
\newcommand{\bc}{\mathbf{c}}
\newcommand{\be}{\mathbf{e}}
\newcommand{\bJ}{\mathbf{J}}
\newcommand{\bI}{\mathbf{I}} 
\newcommand{\bbf}{\mathbf{f}}
\newcommand{\bm}{\mathbf{m}}
\newcommand{\bL}{\mathbf{L}} 

\newcommand{\bs}{{\boldsymbol{s}}} 
\newcommand{\bzero}{\boldsymbol{0}} 
\newcommand{\bF}{\boldsymbol{F}} 
\newcommand{\bvarsigma}{\boldsymbol{\varsigma}} 
\newcommand{\bD}{\boldsymbol{D}} 

\newcommand{\sfT}{\mathsf{T}}

\newcommand{\sfm}{\mathsf{m}}

\newcommand{\cO}{\mathcal{O}}
\newcommand{\cD}{\mathcal{D}}
\newcommand{\cL}{\mathcal{L}}

\newcommand{\cX}{\mathcal{X}}
\newcommand{\cY}{\mathcal{Y}} 
\newcommand{\cH}{\mathcal{H}}

\newcommand{\cM}{\mathcal{M}}
\newcommand{\tcM}{\widetilde{\cM}}
\newcommand{\cE}{\mathcal{E}}

\newcommand{\cZ}{\mathcal{Z}}

\newcommand{\cV}{\mathcal{V}}
  
\newcommand{\cI}{\mathcal{I}}
\newcommand{\cC}{\mathcal{C}}  
 
\newcommand{\cK}{\mathcal{K}}

\newcommand{\cMo}{\mathcal{M}^{\rm o}}
\newcommand{\cRz}{\mathcal{R}^{(0)}} 
\newcommand{\Sol}{\mathcal{S}}

\newcommand{\hb}{\hat{b}} 

\newcommand{\hrho}{\hat{\rho}} 
\newcommand{\hGamma}{{\widehat{\Gamma}}}
 
\newcommand{\hbeta}{\hat{\beta}} 
\newcommand{\hDelta}{\widehat{\Delta}} 
\newcommand{\hI}{\widehat{I}}
 
\newcommand{\hPi}{\widehat{\Pi}}

\newcommand{\frs}{\mathfrak{s}}

\newcommand{\frY}{\mathfrak{Y}}

\newcommand{\rsH}{\marsfs{H}} 
\newcommand{\rsF}{\marsfs{F}} 
\newcommand{\rsR}{\marsfs{R}} 

\newcommand{\chT}{\check{\T}}
\newcommand{\chY}{\check{Y}} 
\newcommand{\chcY}{\check{\cY}}
\newcommand{\chcX}{\check{\cX}} 
\newcommand{\chD}{\check{D}} 
\newcommand{\chfrY}{{\check{\frY}}}

\newcommand{\txi}{\tilde{\xi}} 
\newcommand{\tbxi}{\tilde{\boldsymbol{\xi}}}

\newcommand{\tvarsigma}{\tilde{\varsigma}}
\newcommand{\tbvarsigma}{\tilde{\bvarsigma}}
\newcommand{\tUpsilon}{\widetilde{\Upsilon}} 
\newcommand{\tbUpsilon}{\widetilde{\boldsymbol{\Upsilon}}}
\newcommand{\tE}{\widetilde{E}} 

\newcommand{\vW}{\vec{W}} 

\newcommand{\ov}{\overline}

\newcommand{\iu}{\pmb{\mathtt{i}}} 

\newcommand{\ceil}[1]{\lceil #1\rceil}

\newtheorem{theorem}{Theorem}[section]
\newtheorem{lemma}[theorem]{Lemma}
\newtheorem{proposition}[theorem]{Proposition}

\newtheorem{corollary}[theorem]{Corollary} 
\newtheorem{question}{Question} 

\theoremstyle{definition}
\newtheorem{definition}[theorem]{Definition}
\newtheorem{remark}[theorem]{Remark}
\newtheorem{example}[theorem]{Example}

\def\pair#1#2{\langle #1,#2\rangle}
\def\parfrac#1#2{\frac{\partial{#1}}{\partial #2}}
\def\corr#1{\left\langle #1 \right\rangle}

\title{Quantum Cohomology and Periods}
\author{Hiroshi Iritani}
\address{Department of Mathematics, Kyoto University, 
Kitashirakawa-Oiwake-cho, Sakyo-ku, Kyoto, 606-8502, Japan}
\email{iritani@math.kyoto-u.ac.jp}  
\thanks{This research is supported by Grant-in-Aid for 
Young Scientists (B) 22740042.}

\begin{document} 
\begin{abstract} 
In a previous paper \cite{Iritani:int}, 
the author introduced a $\Z$-structure in quantum cohomology 
defined by the $K$-theory and the Gamma class 
and showed that it is compatible with mirror symmetry 
for toric orbifolds. 
Applying the quantum Lefschetz principle to the previous results, 
we find an explicit relationship between solutions 
to the quantum differential equation for toric complete intersections  
and the periods (or oscillatory integrals) of their mirrors. 
We describe in detail the mirror isomorphism of variations of 
$\Z$-Hodge structure for a mirror pair of 
Calabi-Yau hypersurfaces (Batyrev's mirror). 
\end{abstract} 
\maketitle

\section{Introduction} 
Hodge theoretic mirror symmetry is concerned with 
the equivalence of Hodge structures from 
symplectic geometry (A-model or Gromov-Witten theory) of $Y$ 
and complex geometry (B-model or Kodaira-Spencer theory) 
of the mirror $\check{Y}$. 
In \cite{Iritani:int}, we introduced 
a $\Z$-structure in the A-model Hodge theory 
in terms of the $K$-group and the $\hGamma$-class of $Y$. 
When $Y$ is a weak Fano compact toric orbifold, 
we showed that this $\Z$-structure 
in the A-side is in fact mirror 
to the natural $\Z$-structure in the B-side. 
This was based on the mirror theorem \cite{CCIT:toric} 
for toric orbifolds which will be shown 
in joint work with Coates, Corti and Tseng 
and a calculation of oscillatory integrals on the B-side. 
In this paper we extend the previous results 
in \cite{Iritani:int} to the case of 
complete intersections in toric orbifolds. 

For simplicity, we explain the case where 
$Y$ is a Calabi-Yau manifold. 
The variation of Hodge structure on the A-side is 
given by the trivial holomorphic vector bundle 
$\rsH = H^*(Y) \times H^2(Y) \to H^2(Y)$ endowed 
with the flat Dubrovin connection 
\[
\nabla_V = d_V + V\circ_\tau, \qquad 
V \in H^2(Y)  
\]
where $V \circ_\tau$ is the quantum multiplication by $V$ 
at $\tau\in H^2(Y)$. 
%This is called the \emph{quantum $D$-module}. 
The Hodge filtration and the polarization form 
are given by 
$\rsF^{\,p} = H^{\le 2(\dim Y-p)}(Y)$ 
and $Q(\alpha,\beta) = (2\pi\iu)^{\dim Y} 
\int_Y((-1)^{\frac{\deg}{2}}\alpha)\cup \beta$ 
respectively. 
For $\cE \in K(Y)$, 
we have a unique flat section $\frs(\cE)$ of 
the Dubrovin connection satisfying 
\[
\frs(\cE) \sim (2\pi\iu)^{-\dim Y} e^{-\tau} \left( 
\hGamma_Y \cup (2\pi\iu)^{\frac{\deg}{2}} \ch(\cE) \right) 
\]
in the large radius limit, i.e.\ 
as $e^{\pair{\tau}{d}} \to 0$ for all nonzero 
effective classes $d\in H_2(Y;\Z)$. 
The Gamma class $\hGamma_Y$ here plays 
the role of a ``square root" of the Todd class 
(see \eqref{eq:half}) so that we have 
$Q(\frs(\cE_1),\frs(\cE_2)) = \chi(\cE_1,\cE_2)$ 
by Hirzebruch-Riemann-Roch. 
The \emph{$\hGamma$-integral structure} 
is defined to be the $\Z$-local system consisting of 
the flat sections $\frs(\cE)$, $\cE \in K(Y)$. 
We call the pairing 
\[
\Pi(\phi,\cE) := Q(\phi(\tau), \frs(\cE)(\tau)) 
%(2\pi\iu)^{\dim Y} \int_Y 
%((-1)^{\frac{\deg}{2}} \phi(\tau)) \cup 
%\frs(\cE)(\tau) 
\]
of any section $\phi(\tau)\in \rsH\,$ 
with the flat section $\frs(\cE)$  
the \emph{A-period} of $Y$. 
Our main theorem identifies the A-periods of $Y$ 
with the usual periods of the mirror $\chY$. 

Let $\cY$ be a quasi-smooth Calabi-Yau hypersurface 
in a weak Fano Gorenstein toric orbifold $\cX$. 
Here we allow $\cY$ to have orbifold singularities. 
Let $\Delta\subset \bN_\R$ be the fan polytope of $\cX$. 
The \emph{Batyrev mirror} of $\cY$ is the 
hypersurface $\chY_\alpha = \{W_\alpha(t) = 1\}$ 
in the algebraic torus 
$\chT = \Hom(\bN,\C^\times)\cong (\C^\times)^n$ 
defined by the Laurent polynomial 
$W_\alpha(t)= \sum_{b\in \Delta\cap \bN} \alpha_b t^b$ on $\chT$. 
The affine hypersurface $\chY_\alpha$ can be compactified to 
a Calabi-Yau orbifold $\chcY_\alpha$. 
\begin{theorem}[Theorems \ref{thm:Aperiod=Bperiod}, 
\ref{thm:intstr_Batyrevmirror}, 
\ref{thm:Opt}]
\label{thm:introd}
The A-period for $\cY$ associated to 
$\cE\in K(\cY)$ can be written as a period of $\chY_\alpha$ 
for some integral cycle $C_\cE$ 
if either $\cE$ is pulled-back from the 
ambient toric orbifold $\cX$ or $\cE = \cO_{\rm pt}$: 
\begin{align}
\label{eq:periodequality} 
\Pi(\Upsilon_v, \cE)(\varsigma(\alpha)) 
& = \int_{C_\cE} 
(-1)^{\age(v)} \age(v)! 
\Res\left( \frac{
\alpha^v t^v \frac{dt_1}{t_1} \wedge 
\cdots \wedge \frac{dt_n}{t_n}}{(W_\alpha(t)-1)^{\age(v)+1}} 
\right).  
\end{align} 
Here $\Upsilon_v$ is a section of $\rsH$ 
which (see \eqref{eq:Upsilon}) is asymptotically 
the same in the large radius limit 
as the unit class $\unit_v$ 
on the twisted sector associated to $v\in \Boxop$ 
and $\varsigma(\alpha)$ is the mirror map.  
\end{theorem}  
We calculate the left-hand side of \eqref{eq:periodequality} 
as explicit hypergeometric series (Theorem \ref{thm:mirrorthm_tw})
by applying the quantum Lefschetz principle 
\cite{Coates-Givental, CCIT:tw} 
to the mirror theorem \cite{CCIT:toric} for toric orbifolds. 
Theorem \ref{thm:introd} then follows from the Laplace 
transformation of the previous results in \cite{Iritani:int}. 
Similar results for toric complete intersections 
are given in Theorem \ref{thm:Aperiod=Bperiod}.  
We use Theorem \ref{thm:introd} to establish  
the mirror isomorphism between the 
\emph{ambient A-model VHS} of $\cY$ 
and the \emph{residual B-model VHS} of $\chcY_\alpha$ 
which preserves certain integral structures 
(Theorem \ref{thm:intstr_Batyrevmirror}). 

The present work is motivated by Givental's celebrated 
paper \cite{Givental:mirrorthm-toric} on mirror symmetry 
for toric complete intersections, where 
Givental remarked that each component of the $I$-function 
can be written as an oscillatory integral. 
In terms of a hypergeometric differential system, 
essentially the same integral structure has been 
identified in the work of Borisov-Horja \cite{Borisov-Horja:FM} 
and Hosono \cite{Hosono:central}. 
The $\hGamma$-structure was also proposed 
by Katzarkov-Kontsevich-Pantev \cite{KKP} independently. 
Our results give a partial affirmative answer 
to the conjecture of  
Hosono \cite[Conjecture 6.3]{Hosono:central}. 

%An important point in this paper is 
%that we are working with orbifolds. 
%It turns out that 
The concept of orbifold has been 
a rich source of ideas in mirror symmetry. 
For example, Batyrev's mirror may not admit a full 
crepant resolution for dimension bigger than 3. 
By the development of orbifold Gromov-Witten theory 
\cite{Chen-Ruan:new_coh_orb, Chen-Ruan:GW, AGV}, 
we can now work with partial resolutions with 
orbifold singularities. 
In this paper, we encounter a phenomenon of 
\emph{multi-generation}\footnote{This does \emph{not} 
mean that the quantum $D$-module is not cyclic.} 
of orbifold quantum $D$-modules. 
This phenomenon was first observed by Guest-Sakai \cite{Guest-Sakai} 
(in a different language) 
for a degree 3 Fano hypersurface in $\Proj(1,1,1,2)$. 
For an orbifold hypersurface, it can happen that 
the ambient part\footnote
{The ambient part is the subbundle of the quantum $D$-module 
with fiber $\iota^* H_{\rm orb}^*(\cX)\subset H_{\rm orb}(\cY)$ 
where $\iota \colon \cY \to \cX$ is the inclusion 
of the hypersurface $\cY$ into the ambient toric orbifold $\cX$.} 
of the small quantum $D$-module 
is not generated by the single unit class $\unit$ as an 
$\cO[z]\langle z \partial \rangle$-module\footnote{
On the other hand, when $z$ inverted, it is generated 
by $\unit$ as an 
$\cO[z,z^{-1}]\langle z\partial\rangle$-module 
under the assumption on the ambient toric orbifold 
in this paper.},  
but is generated by $\unit$ 
and the unit classes $\unit_v$ supported on twisted sectors. 
Here $\partial$ denotes the derivative in the 
$H^{\le 2}_{\rm orb}$-direction and $z$ is an additional 
variable in the quantum $D$-module 
(see Definition \ref{def:QDM}). 
For the A-model VHS of a Calabi-Yau hypersurface, 
this means that each Hodge filter $\rsF^{\, p}$ 
may not be generated by 
$\le (\dim \cY-p)$ times derivatives of the 
top filter $\rsF^{\, \dim \cY}$. 
% In general, the same phenomenon can occur 
% for toric orbifolds themselves \cite{CCIT:toric}. 
In fact, we will describe the quantum $D$-module of 
toric Calabi-Yau hypersurfaces in terms of the 
\emph{multi-GKZ system} 
(Theorem \ref{thm:CYhyp_multiGKZ}) --- a GKZ system 
defined by multiple generators.  
The same generalization of the GKZ system was  
proposed in a recent work 
by Borisov-Horja \cite{Borisov-Horja:better} 
who called it \emph{better behaved GKZ system}. 
% This is related to the point that we needed 
% to study the case $v\neq 0$ (twisted sector) 
% in Theorem \ref{thm:introd}. 
This multi-generation is a reason why 
we needed to show Theorem \ref{thm:introd}  
also for twisted sectors  $v\neq 0$. 

\vspace{5pt} 
\noindent
{\bf Acknowledgments} 
The author has learned a lot 
from various joint works with 
Alessandro Chiodo, Tom Coates, Alessio Corti, 
Sergey Galkin, Vasily Golyshev, 
Yongbin Ruan and Hsian-Hua Tseng. 
He would like to thank them all. 
He also would like to thank Yukiko Konishi and 
Satoshi Minabe for very helpful discussions concerning 
their work \cite{Konishi-Minabe:local}. 
He is grateful to Etienne Mann and Thierry Mignon 
for informing the author of their work \cite{Mann-Mignon} 
and to Martin Guest for very helpful comments on a draft 
version of this paper. 
 
\section{Preliminaries} 

\subsection{Orbifold Gromov-Witten Invariants} 
\label{subsec:orbGW}
Gromov-Witten theory for orbifolds has been developed 
by Chen-Ruan for symplectic orbifolds and 
by Abramovich-Graber-Vistoli for smooth Deligne-Mumford stacks. 
Here we fix notation for orbifold Gromov-Witten invariants. 
For the details of the subject, 
we refer the reader to the original 
articles \cite{Chen-Ruan:new_coh_orb, Chen-Ruan:GW, AGV}. 

Let $\cX$ be a proper smooth Deligne-Mumford stack 
over $\C$ and $X$ be its coarse moduli space. 
Set $n = \dim_\C \cX$. 
We assume that $X$ is projective. 
Let $\cI\cX$ be the inertia stack, which is the 
fiber product $\cX \times_{\cX \times \cX} \cX$ 
of the diagonal morphisms $\Delta\colon 
\cX \to \cX \times \cX$. A $\C$-valued 
point of $\cI\cX$ is a pair $(x,g)$ of a 
$\C$-valued point $x \in \cX$ and 
a stabilizer $g\in \Aut(x)$ at $x$.  
Let 
\[
\cI\cX = \bigsqcup_{v \in \sfT} \cX_v 
= \cX_0 \sqcup \bigsqcup_{v\in \sfT'} \cX_v, 
\quad \cX_0 = \cX. 
\]
be the decomposition of $\cI\cX$ into connected 
components. The index set $\sfT$ contains a special 
element $0\in \sfT$ corresponding to the 
trivial stabilizer $g=1$. 
We set $\sfT' = \sfT \setminus \{0\}$. 
Let $\age(v) \in \Q_{\ge 0}$ be the age (or degree shifting 
number) of the component $\cX_v$. 
The \emph{Chen-Ruan orbifold cohomology group}  
$H_{\rm orb}^*(\cX)$ is 
the $\Q$-graded vector space given by 
\[
H_{\rm orb}^p(\cX) := \bigoplus_{\{v\in \sfT | 
 p-2\age(v) \in 2 \Z\} } 
H^{p-2\age(v)}(\cX_v;\C), \quad p \in \Q.   
\] 
Throughout the paper, we ignore odd cohomology 
classes in Gromov-Witten theory 
i.e.\ elements in $H^{p-2\age(v)}(\cX_v)$ 
with $p-2\age(v)$ odd. 
($H_{\rm orb}^*(\cX)$ is sometimes denoted by 
$H_{\rm CR}^*(\cX)$ in the literature.)
We have an involution $\inv \colon \cI\cX \to \cI\cX$ 
given by $(x,g) \mapsto (x,g^{-1})$. 
This induces an involution 
$\inv^* \colon H_{\rm orb}^*(\cX) \to H_{\rm orb}^*(\cX)$. 
The orbifold Poincar\'{e} pairing 
$(\cdot,\cdot)_{\rm orb}\colon H^*_{\rm orb}(\cX) 
\otimes H^*_{\rm orb}(\cX) \to \C$ 
is defined by  
\[
(\alpha,\beta)_{\rm orb} := 
\int_{\cI\cX} \alpha \cup \inv^*\beta. 
\]
This is a nondegenerate symmetric bilinear form 
of degree $-2n$.  
Let $\cX_{0,l,d}$ denote the \emph{moduli stack of 
stable maps} of genus 0, $l$-pointed and 
degree $d \in H_2(X,\Z)$. 
(This is the same as the 
\emph{stack of twisted stable maps}  $\cK_{0,l}(\cX,d)$ in 
\cite{AGV}.) This is equipped with a 
virtual fundamental class $[\cX_{0,l,d}]^{\rm vir}
\in H_*(\cX_{0,l,d};\Q)$ and the evaluation maps 
\[
\ev_i \colon \cX_{0,l,d} \to \ov{\cI\cX}, \quad i=1,\dots, l   
\]
to the rigidified inertia stack\footnote{
The rigidified inertia stack $\ov{\cI\cX}$ is obtained from
$\cI\cX$ by taking the quotient of the automorphism group  
at $(x,g)\in \cI\cX$ by the cyclic group generated by 
$g$.} 
$\ov{\cI\cX}$ (see \cite{AGV}). 
Take $\alpha_1,\dots,\alpha_l \in H^*_{\rm orb}(\cX)$ 
and nonnegative integers $k_1,\dots,k_l$.  
The \emph{orbifold Gromov-Witten invariants} are defined by 
\[
\corr{\alpha_1 \psi^{k_1}, \alpha_2 \psi^{k_2}, \dots, 
\alpha_l \psi^{k_l}}_{0,l,d}  
:= \int_{[\cX_{0,l,d}]^{\rm vir}} 
\prod_{i=1}^l (\ev_i^*(\alpha_i) \cup \psi_i^{k_i}). 
\]
Because $\ov{\cI\cX}$ and $\cI\cX$ are the same as 
topological spaces, we can define the 
pull-back $\ev_i^*(\alpha_i)$ for $\alpha_i\in H^*_{\rm orb}(\cX)$. 
The class $\psi_i$ is the first Chern 
class of the $i$-th universal cotangent line bundle $\cL_i
\to \cX_{0,l,d}$ whose fiber at a stable map 
$f \colon \cC \to \cX$ is the cotangent space 
$T_{x_i}^*C$ at the $i$-th marked point 
of the \emph{coarse} domain curve $C$.  

\subsection{Twisted Invariants} 
Following \cite{Coates-Givental, Tseng:QRR, CCIT:tw}, 
we introduce the orbifold Gromov-Witten invariants 
\emph{twisted} by a vector bundle $\cV$ on $\cX$ and 
a characteristic class $\bc$ . 
We use these invariants to calculate the Gromov-Witten 
invariants of a complete intersection in $\cX$. 
Let $\bc(\cdot) = \exp(\sum_{k=0}^{\infty} s_k\ch_k(\cdot))$ 
be a universal invertible multiplicative characteristic class 
with parameters $\bs=(s_0,s_1,s_2,\dots)$. 
Let $\cI\cV$ be the vector bundle on $\cI\cX$ whose 
fiber at $(x,g)$ is the $g$-fixed subspace 
of $\cV_x$. 
In the twisted theory, the pairing $(\cdot,\cdot)_{\rm orb}$ 
is replaced with the following twisted 
Poincar\'{e} pairing: 
\[
(\alpha, \beta)_{\rm orb}^{\bc} = 
\int_{\cI\cX} \alpha \cup \inv^*(\beta) \cup \bc(\cI\cV).  
\] 
Using the universal family 
$u\colon \cC_{0,l,d} \to \cX$ over $\cX_{0,l,d}$, 
we define a $K$-group element 
$\cV_{0,l,d} \in K^0(\cX_{0,l,d})$ by 
$\cV_{0,l,d} = R\pi_* u^* \cV$. 
\[
\begin{CD}
\cC_{0,l,d} @>{u}>>  \cX \\ 
@V{\pi}VV @. \\ 
\cX_{0,l,d}  
\end{CD}
\]
Define the \emph{twisted Gromov-Witten invariants} by 
\begin{equation}
\label{eq:twistedinv} 
\corr{\alpha_1 \psi^{k_1}, \alpha_2 \psi^{k_2}, \dots, 
\alpha_l \psi^{k_l}}_{0,l,d}^{\bc}   
:= \int_{[\cX_{0,l,d}]^{\rm vir}} 
\bc(\cV_{0,l,d}) \cup 
\prod_{i=1}^l \ev_i^*(\alpha_i) \psi_i^{k_i}.  
\end{equation} 
Note that the twisted invariants equal the original ones 
when $\bc$ is trivial (i.e.\ $\bc\equiv 1$). 

% From Tseng's orbifold quantum Riemann-Roch \cite{Tseng:QRR}, 
% it follows that the twisted invariants satisfy 
% the String Equation (SE), the Dilaton Equation 
% (DE) and the Topological Recursion Relation (TRR) 
% listed e.g.\ in \cite[Section 1]{Pandharipande}. 
% (In the TRR, we need to use the twisted Poincar\'{e} pairing.) 
% This is because these equations 
% correspond to certain special geometric properties 
% of Givental's Lagrangian cone (see \cite{Givental:symplectic})  
% and the symplectic operator in 
% Tseng's quantum Riemann-Roch preserves such properties.  
% These equations show that the twisted Gromov-Witten 
% theory at genus zero gives rise to the 
% quantum cohomology and $D$-module. 

\subsection{Twisted Quantum Cohomology} 
%Here we introduce the twisted quantum cohomology.  
% following \cite{Iritani:genmir}. 
We can define both untwisted and twisted quantum cohomology, 
but we begin with the twisted version 
because the untwisted version is obtained from it 
by the specialization $\bc =1$. 
Let $\Eff_\cX \subset H_2(X;\Z)$ denote the semigroup 
generated by effective curves. The Novikov ring $\Lambda$ 
is defined to be the completion of the group ring 
$\C[\Eff_\cX]$. For a curve class $d\in \Eff_\cX$, 
let $Q^d$ be the corresponding element in $\Lambda$. 
Define $\Lambda_\bs$ to be the completion 
of $\C[\Eff_\cX][s_0,s_1,s_2,\dots]$ with respect to 
the additive valuation $v$ given by 
\[
v(Q^d)= \int_d \omega, \quad v(s_k) = k+1.  
\] 
where $\omega$ is a K\"{a}hler class of $\cX$. 
Let $\{\phi_1,\dots,\phi_N\}
\subset H_{\rm orb}^*(\cX)$ be a homogeneous 
$\C$-basis, $\{\tau^1,\dots, \tau^N\}$ be the dual co-ordinates 
on $H_{\rm orb}^*(\cX)$ and 
$\tau = \sum_{i=1}^N \tau^i \phi_i$ be a 
general point on $H_{\rm orb}^*(\cX)$. 
The \emph{twisted quantum product} 
$\bullet^\bc_\tau$ is defined by 
the formula: 
\begin{equation} 
\label{eq:twistedquantumprod} 
(\alpha\bullet^\bc_\tau \beta, \gamma)_{\rm orb}^{\bc}  
= \sum_{l\ge 0} \sum_{d\in \Eff_\cX} 
\corr{\alpha,\beta,\gamma,
\tau,\dots,\tau}_{0,l+3,d}^{\bc } \frac{Q^d}{l!} 
\end{equation} 
where $\alpha,\beta,\gamma \in H_{\rm orb}^*(\cX)$. 
This defines a unique element 
$\alpha\bullet^\bc_\tau \beta$ in $H_{\rm orb}^*(\cX)\otimes 
\Lambda_\bs[\![\tau]\!]$. Here $\Lambda_\bs[\![\tau]\!] 
:= \Lambda_\bs[\![\tau^1,\dots,\tau^N]\!]$. 
The product $\bullet^\bc_\tau$ 
is extended bilinearly over $\Lambda_\bs[\![\tau]\!]$ 
and defines a ring structure on $H_{\rm orb}^*(\cX)\otimes 
\Lambda_\bs[\![\tau]\!]$. 
We call the ring $(H_{\rm orb}^*(\cX) \otimes \Lambda_\bs[\![\tau]\!], 
\bullet^\bc_\tau)$ 
the \emph{twisted quantum cohomology}. 
%We set $\bullet_\tau := \bullet^\bc_\tau|_{\bc =0}$.  
For a topological ring $R$ with an additive valuation 
$v \colon R \to \R \cup \{\infty\}$, we 
define $R\{z,z^{-1}\}$ to be the space of 
all power series $\sum_{k\in \Z} a_k z^k$ with 
$a_k\in R$ such that $\lim_{|k| \to \infty} v(a_k) = \infty$. 
Let $R\{z\}$ (resp.\ $R\{z^{-1}\}$) 
denote the subspace of $R\{z,z^{-1}\}$ 
consisting of nonnegative (resp.\ nonpositive) power series in $z$. 
These are rings when $R$ is complete. 
We define the \emph{Dubrovin connection} 
$\Nabla^{\bc}_i \colon H^*_{\rm orb}(\cX) \otimes 
\Lambda_{\bs}\{z\}[\![\tau]\!] \to z^{-1} 
H^*_{\rm orb}(\cX)\otimes \Lambda_{\bs}\{z\}[\![\tau]\!]$ 
by 
\begin{equation*} 
%\label{eq:Dub}
\Nabla^{\bc}_i = \parfrac{}{\tau^i} 
+ \frac{1}{z} \phi_i\bullet^\bc_\tau. 
\end{equation*} 
The differential equation $\Nabla^{\bc}_i s(\tau,z) = 0$ 
for a cohomology-valued function $s$ is called 
the \emph{quantum differential equation}. 
Define $\bL^\bc(\tau,z) \in \End(H^*_{\rm orb}(\cX)) 
\otimes \Lambda_\bs\{z^{-1}\}[\![\tau]\!]$ by 
\begin{equation}
\label{eq:twistedfundsol}
(\bL^\bc(\tau,z) \alpha, \beta)_{\rm orb}^{\bc} 
= (\alpha,\beta)_{\rm orb}^{\bc} 
+ \sum_{\substack{(d,l) \neq (0,0) \\ 
d\in \Eff_\cX,\, l\ge 0}}   
\corr{\frac{\alpha}{-z-\psi}, \tau, \dots, \tau,\beta}^{\bc}_{0,l+2,d} 
\frac{Q^d}{l!}.  
\end{equation} 
Here $1/(-z-\psi)$ in the correlator should be 
expanded in the series $\sum_{k\ge 0} (-z)^{-k-1} \psi^k$. 
\begin{proposition}
The $\End(H^*_{\rm orb}(\cX))$-valued function 
$\bL^\bc(\tau,z)$ gives a fundamental solution to the quantum 
differential equation: It satisfies 
\[
\Nabla^{\bc}_i (\bL^\bc(\tau,z) \alpha) = 0, \quad 1\le i\le N, 
\quad \forall \alpha\in H^*_{\rm orb}(\cX)   
\]
and $\bL^\bc(\tau,z) = \id + O(Q,\tau)$. 
We also have 
\begin{equation} 
\label{eq:unitarityofbL}
(\bL^\bc(\tau,-z) \alpha, \bL^\bc(\tau,z) \beta)_{\rm orb}^{\bc} 
= (\alpha, \beta)_{\rm orb}^{\bc}.  
\end{equation} 
\end{proposition} 
\begin{proof} 
See \cite[Proposition 2.3]{Iritani:genmir} and \cite{Mann-Mignon}  
when $\cX$ is a smooth variety.  
In this proof, we will freely use the language of 
Givental's Lagrangian cone for which 
we refer the reader to \cite{Givental:symplectic, CCIT:tw}. 
From Tseng's orbifold Quantum Riemann-Roch (QRR) \cite{Tseng:QRR}, 
it follows that the twisted Gromov-Witten 
invariants \eqref{eq:twistedinv} satisfy 
the String Equation (SE), the Dilaton Equation 
(DE) and the Topological Recursion Relation (TRR) 
listed e.g.\ in \cite[Section 1]{Pandharipande}. 
(In the TRR, we need to use the twisted Poincar\'{e} pairing.) 
This is because these equations 
correspond to certain special geometric properties 
of Givental's Lagrangian cone (see \cite{Givental:symplectic})  
and the symplectic operator in 
Tseng's QRR preserves such properties.  
The differential equation for $\bL^\bc(\tau,z)$ has 
been proved for the untwisted theory for manifolds 
in \cite[Proposition 2]{Pandharipande} using TRR 
and the same proof applies to our case. 
It is easy to see that 
$\bL^\bc(\tau,z)^\dagger \beta$ is a tangent vector 
of Givental's Lagrangian cone for the twisted theory. 
Here $\bL^\bc(\tau,z)^\dagger$ denotes the adjoint 
of $\bL^\bc(\tau,z)$, i.e.\ 
$(\alpha, \bL^\bc(\tau,z)^\dagger\beta)_{\rm orb}^{\bc} 
= (\bL^\bc(\tau,z)\alpha, \beta)_{\rm orb}^{\bc}$. 
By the Lagrangian property of the cone, we know that 
$(\bL^\bc(\tau,-z)^\dagger \alpha, 
\bL^\bc(\tau,z)^\dagger \beta)_{\rm orb}^{\bc}$
contains only nonnegative powers in $z$. 
On the other hand $\bL^\bc(\tau,z)^\dagger \beta= \beta +O(z^{-1})$. 
Therefore we have 
$(\bL^\bc(\tau,-z)^\dagger \alpha, 
\bL^\bc(\tau,z)^\dagger \beta)_{\rm orb}^{\bc}
= (\alpha,\beta)_{\rm orb}^{\bc}$ and so 
$\bL^\bc(\tau,-z)^\dagger$ is inverse to $\bL^\bc(\tau,z)$. 
This proves \eqref{eq:unitarityofbL}. 
\end{proof} 
\begin{remark} 
The existence of a fundamental 
solution implies that the Dubrovin connection 
$\Nabla^{\bc}$ is flat, i.e.\ $[\Nabla^{\bc}_i,\Nabla^{\bc}_j]=0$. 
This in turn shows that the twisted 
quantum product $\bullet^\bc_\tau$ is associative. 
\end{remark} 
\begin{definition}[\cite{Givental:mirrorthm-toric, CCIT:tw}] 
We define the \emph{$J$-function} of the twisted theory 
by 
\begin{equation} 
\label{eq:twisted-bJ} 
\bJ^\bc(\tau,z) := \bL^\bc(\tau,z)^{-1} \unit 
= \bL^\bc(\tau,-z)^\dagger \unit.  
\end{equation} 
\end{definition}

\subsection{Equivariant Euler Twist} 
\label{subsec:equiv-eulertwist} 
We consider the case where $\bc$ is the 
$S^1$-equivariant Euler class $\be_\lambda$. 
Here $S^1$ acts on vector bundles by scaling the 
fibers and $\lambda\in H^2_{S^1}({\rm pt})$ 
denotes a generator. 
We have $\be_\lambda(\cE) = \sum_{i=0}^{r} 
\lambda^{i} c_{r - i}(\cE)$ for a rank $r$ vector 
bundle $\cE$.  
Then $\be_\lambda$ corresponds to the choice 
of parameters 
\[
s_0 = \log \lambda, \quad 
s_i = (-1)^{i-1} (i-1)! \lambda^{-i} \ (i\ge 1).  
\]
If $\cV_{0,l,d}$ is not represented by 
a vector bundle, the $\be_\lambda$-twisted invariants 
take values in $\C[\lambda,\lambda^{-1}]$. 
% Then the ground ring $\Lambda_\bs$ should be replaced with 
% the completion\footnote{One could write $\Lambda_\lambda
% = \C(\!(\lambda^{-1})\!)[\![\Eff_\cX]\!]$.} 
% $\Lambda_\lambda$ of $\C[\Eff_\cX][\lambda,\lambda^{-1}]$ 
% with respect to the additive valuation defined by 
% $v(Q^d) = \int_d \omega$ and $v(\lambda^{-1}) = 1$. 
% (Thus we allow power series in $\lambda$ infinite 
% in the \emph{negative} direction.) 
In this paper, we only consider 
the case where $\cV_{0,n,d}$ is 
a vector bundle and no negative powers of 
$\lambda$ appear. 
Then we can take the ground ring 
to be (instead of $\Lambda_\bs$) 
the completion $\Lambda_\lambda$
of $\C[\Eff][\lambda]$ with respect to 
the valuation $v(Q^d) = \int_d \omega$, $v(\lambda)=0$. 

We assume that $\cV$ is 
the sum $\cL_1 \oplus \cdots \oplus \cL_c$  
of line bundles such that $c_1(\cL_j)$ 
is nef and $\cL_j$ is a pull-back 
from the coarse moduli space $X$ for all $1\le j\le c$. 
% In this case the $(\be_\lambda,\cV)$-twisted invariants 
% are related to the Gromov-Witten invariants of a 
% complete intersection $\cY$ in $\cX$ with respect to 
% a regular section of $\cV$. 
% Let $H^2(\cX;\Z)$ denote the sheaf cohomology on the 
% topological stack $\cX$. This classifies 
% all the topological orbi-line bundles over $\cX$. 
% For $\xi\in H^2(\cX;\Z)$,  
% let $f_v(\xi)$ be the age of the corresponding 
% orbi-line bundle $\cL_\xi$ along the sector $\cX_v$. 
% This is the rational number in $[0,1)$ 
% such that the stabilizer along $\cX_v$ 
% acts on $\cL_\xi$ by $\exp(2\pi\iu f_v(\xi))$. 
% Let $\xi_i\in H^2(\cX;\Z)$ be the class of $\cL_i$.  
% Define the \emph{$\cV$-untwisted sector} 
% to be the following subspace of $H_{\rm orb}^*(\cX)$:
% \[
% H_\cV^{\rm untw} := \bigoplus_{v: f_v(\xi_1) = \cdots = f_v(\xi_c)=0} 
% H^*(\cX_v).  
% \]
Let $\cY$ be a quasi-smooth complete intersection in $\cX$ 
with respect to a regular section of $\cV$. 
Let $\iota \colon \cY \subset \cX$ denote the inclusion. 
The pull-back $\iota^* \colon H^*_{\rm orb}(\cX) 
\to H^*_{\rm orb}(\cY)$ and the push-forward 
$\iota_* \colon H^*_{\rm orb}(\cY) \to 
H^*_{\rm orb}(\cX)$ are defined by the 
inclusion $\cI\cY \subset \cI\cX$. 
% We write 
% $\bL^\lambda(\tau,z)$ and $\bJ^\lambda(\tau,z)$ 
% for the fundamental solution and the $J$-function 
% of the $(\be_\lambda,\cV)$-twisted theory of $\cX$.  
We also write $\bL_\cY(\tau,z)$, $\bJ_\cY(\tau,z)$ 
for the fundamental solution and the $J$-function 
of the \emph{untwisted} theory of $\cY$.  

\begin{proposition} 
\label{prop:Eulertw_Y} 
Under the above assumption, 
$\bL^{\be_\lambda}(\tau,z)$ and $\bJ^{\be_\lambda}(\tau,z)$ 
contain no negative powers in $\lambda$. 
So we can set $\bL^\be(\tau,z) := 
\bL^{\be_\lambda}(\tau,z)|_{\lambda =0}$, 
$\bJ^\be(\tau,z) := \bJ^{\be_\lambda}(\tau,z)|_{\lambda=0}$. 
Moreover, we have 
\[
\iota^* \bL^{\be}(\tau,z) \alpha 
= \bL_\cY(\iota^*\tau, z) \iota^* \alpha
\Bigr|_{H_2(Y;\Z) \to H_2(X;\Z)}. 
\]
% \[
% \lim_{\lambda \to 0} 
% (\bL^\lambda(\tau,z) \alpha, \beta)_{\rm orb}^{\bc} 
% = (\bL_\cY(\iota^*\tau,z) \iota^* \alpha, 
% \iota^* \beta)^{\cY}_{\rm orb} 
% \Bigr|_{H_2(Y;\Z) \to H_2(X;\Z)}.  
% \]
Here $\alpha,\beta\in H_{\rm orb}^*(\cX)$.  
The notation $H_2(Y;\Z) \to H_2(X;\Z)$ means 
to replace $Q^d$ with $Q^{\iota_*(d)}$ 
for $d\in H_2(Y;\Z)$.  
\end{proposition} 
\begin{proof}
% Let $\cX'_{0,l+2,d}$ be the union of components of 
% $\cX_{0,l+2,d}$ consisting of stable maps 
% $u$ such that for $2\le i\le l+1$, $\ev_i(u)$ belongs to $\cX_v$ 
% for some $v$ with $f_v(\xi_1) = \cdots = f_v(\xi_c)=0$. 
% Note that the twisted invariants contributing to 
% $\bL_\lambda(\tau,z)$, $\tau\in H^{\rm untw}_{\cV}$ 
% are integrals over $\cX'_{0,l+2,d}$. 
% Let $u\colon \cC \to \cX$ be a stable map in $\cX'_{0,l+2,d}$ 
% with marked points $x_1,\dots,x_{l+2}\in \cC$. 
% Let $v,w\in \sfT$ be such that $\ev_1(u) \in \cX_v$ 
% and $\ev_{l+2}(u) \in \cX_w$. 
% We claim that $H^1(\cC,u^*\cL_j)=0$ for all $1\le j\le c$ 
% and that the natural map 
% $H^0(\cC,u^*\cL_j) \to (u^*\cL_j)_{x_{l+2}}$ 
% is surjective if $f_w(\xi_j)=\age_{x_{l+2}}(u^*\cL_j)=0$. 
% Let $\pi \colon \cC \to \cC'$ be the map 
% forgetting the orbifold structure at $x_1,\dots,x_{l+2}$. 
% Then we have $\deg (\pi_* u^* \cL_i) = \pair{\xi_i}{d} 
% - f_v(\xi_i) - f_w(\xi_i)$. Since the orbifold 
% structure at each node of $\cC'$ is balanced, 
% we know that $\deg(\pi_*u^*\cL_i)$ is an integer. 
% Thus we have $\deg(\pi_*u^*\cL_i)\ge -1$. 
% Moreover $\deg(\pi_*u^*\cL_i)\ge 0$ if $f_w(\xi_i)=0$. 
The proof parallels the argument in \cite[Section 2.1]{Pandharipande}. 
By the assumption, 
for every stable map $u\colon \cC \to \cX$ in $\cX_{0,l+2,d}$, 
the convexity $H^1(\cC,u^*\cV) =0$ holds 
and the natural map $H^0(\cC, u^*\cV) \to (u^*\cV)_{x_{l+2}}$ 
is surjective. Here $x_{l+2}$ is the last marked point 
on $\cC$. Therefore $\cV_{0,l+2,d}$ is 
a vector bundle 
and we can define the subbundle $\cV'_{0,l+2,d}$ by 
the following exact sequence: 
\begin{equation} 
\label{eq:exactseq_V'}
\begin{CD}
0@>>> \cV'_{0,l+2,d} @>>> \cV_{0,l+2,d} 
@>>> \ev^*_{l+2} \cI\cV @>>>0.  
\end{CD} 
\end{equation} 
Here note that $\cI\cV$ defines a vector 
bundle on the rigidified inertia stack $\ov{\cI\cX}$ 
whose fiber at $(x,g)\in \ov{\cI\cX}$ is $\cV_x$.  
Using $\be_\lambda(\cV_{0,l+2,d}) = \be_\lambda(\cV'_{0,l+2,d}) 
\cup \ev_{l+2}^*\be_\lambda(\cI\cV)$, we find that 
$\bL^{\be_\lambda}(\tau,z) \alpha$ equals 
\begin{align*} 
\alpha + 
\sum_{\substack{(d,l)\neq (0,0)\\ 
d\in \Eff_\cX,\;l\ge 0}} \frac{Q^d}{l!} 
\inv^* {\ev_{l+2}}_* \left( 
\frac{\ev_1^* \alpha}{-z-\psi_1} 
\left(\prod_{j=2}^{l+1} \ev_j^*(\tau)\right) 
\be_\lambda(\cV'_{0,l+2,d})
\cap [\cX_{0,l+2,d}]^{\rm vir} \right). 
\end{align*} 
This shows that $\bL^{\be_\lambda}$ does not contain 
negative powers of $\lambda$. 
Since $\bL^{\be_\lambda}= \id + O(Q,\tau)$, 
$(\bL^{\be_\lambda})^{-1}$ and 
$\bJ^{\be_\lambda} = (\bL^{\be_\lambda})^{-1}\unit$ 
do not contain negative powers of $\lambda$ either. 
We denote by $\ev^\cX \colon \cX_{0,l+2,d} \to 
(\ov{\cI \cX})^{l+2}$ 
and $\ev^\cY \colon \cY_{0,l+2,d} 
\to (\ov{\cI\cY})^{l+2}$ 
the collection $(\ev_1,\dots, \ev_{l+2})$ of 
the evaluation maps. 
For the second statement, 
it suffices to show that 
\[
f^* \ev^\cX_* \left(
\psi_1^k \be(\cV_{0,l+2,d}') 
\cap [\cX_{0,l+2,d}]^{\rm vir} \right) 
= 
\sum_{d': \iota_*(d')=d} 
g_* \ev^\cY_* \left( 
\psi_1^k \cap [\cY_{0,l+2,d'}]^{\rm vir} \right) 
\]
where $f$ and $g$ are the inclusions: 
\[
\begin{CD}
(\ov{\cI \cY})^{l+1} \times \ov{\cI\cY} 
@>{g}>> (\ov{\cI \cX})^{l+1} \times \ov{\cI \cY} 
@>{f}>> (\ov{\cI \cX})^{l+1} \times \ov{\cI \cX}.  
\end{CD}
\] 
We consider the fiber diagram 
\[
\begin{CD} 
 \cZ @>{i}>> \cX_{0,l+2,d}  \\ 
@V{\ev^\cZ}VV  @VV{\ev^\cX}V \\
\ov{\cI\cX}^{l+1} \times \ov{\cI\cY} @>{f}>> (\ov{\cI\cX})^{l+2}  
\end{CD} 
\]
When $\cY$ is the zero locus of a regular section 
$s \in H^0(\cX,\cV)$, $\cZ$ is defined to be 
the zero locus of 
$\ev_{l+2}^*(s) \in H^0(\cX_{0,l+2,d},\ev_{l+2}^*\cI\cV)$. 
Using the \emph{refined Gysin map} $f^!$ 
in \cite{Fulton:intersection, Vistoli}, we have 
\[
f^* \ev_*^{\cX} \left(\psi_1^k \be(\cV_{0,l+2,d}') 
\cap[\cX_{0,l+2,d}]^{\rm vir} \right) 
= \ev_*^{\cZ} f^! \left(\psi_1^k \be(\cV_{0,l+2,d}') 
\cap[\cX_{0,l+2,d}]^{\rm vir} \right). 
\]
Let $j\colon \cY_{0,l+2,d} \to \cZ$ be the 
inclusion. It now suffices to show the equality 
of classes on $\cZ$: 
\[
\sum_{d':\iota_*(d') = d} 
j_*\left( 
\psi_1^k \cap [\cY_{0,l+2,d'}]^{\rm vir}\right) 
= f^! \left(
\psi_1^k \be(\cV'_{0,l+2,d}) \cap 
[\cX_{0,l+2,d}]^{\rm vir} \right). 
\]
Note that we only need to consider the case $k=0$ 
since $\psi_1^k$ factors out. 
By the functoriality \cite{Kim-Kresch-Pantev} 
of virtual classes we have 
\[
\sum_{d':\iota_*(d') = d} 
[\cY_{0,l+2,d'}]^{\rm vir}
= 0_\cX^![\cX_{0,l+2,d}]^{\rm vir}
\]
where $0_\cX\colon \cX_{0,l+2,d} \to \cV_{0,l+2,d}$ 
is the zero section (which is the bottom row 
of the diagram below). 
We can make the following fiber diagram: 
% \xymatrix{
% \cY_{0,l+2,d} \ar[d]_j \ar[r]^j 
% & \cZ \ar[d]_{\tilde{s}_\cZ} \ar[r]  
% & \cX_{0,l+2,d} \ar[dd]^{\tilde{s}} \\ 
% \cZ \ar[d] \ar[r]^{0_\cZ}  
% & \cV'_{0,l+2,d}|_\cZ \ar[d]^{h|_{\cZ}} \\
% \cX_{0,l+2,d} \ar[r]^{0_\cX'} 
% & \cV'_{0,l+2,d} \ar[r]^{h} & \cV_{0,l+2,d} }
\[
\begin{CD} 
\cY_{0,l+2,d} @>{j}>> \cZ @>{i}>> \cX_{0,l+2,d} \\ 
@V{j}VV  @V{\tilde{s}_\cZ}VV  @| \\
\cZ @>{0_\cZ}>> \cV'_{0,l+2,d}|_\cZ @. \cX_{0,l+2,d} \\
@V{i}VV @V{h|_\cZ}VV @V{\tilde{s}}VV \\
\cX_{0,l+2,d} @>{0_\cX'}>> \cV'_{0,l+2,d} @>{h}>> \cV_{0,l+2,d} 
\end{CD}
\]
where $\tilde{s}$ and $\tilde{s}_\cZ$ are 
the sections of $\cV_{0,l+2,d}$ and 
$\cV'_{0,l+2,d}|_\cZ$ induced from $s\in H^0(\cX,\cV)$, 
$0_\cX'$ and $0_\cZ$ are the zero sections 
and $h$ is the natural inclusion. 
We have $0_\cX = h\circ 0_\cX'$. 
Using the properties of the Gysin maps, we have 
\begin{align*} 
j_* 0_\cX^{!} [\cX_{0,l+2,d}]^{\rm vir} 
& = j_* 0'_\cX\!\!{}^{!}\, h^! [\cX_{0,l+2,d}]^{\rm vir} 
= 0_\cZ^* (\tilde{s}_{\cZ})_* h^! [\cX_{0,l+2,d}]^{\rm vir} \\  
& = \be(\cV_{0,l+2,d}') h^! [\cX_{0,l+2,d}]^{\rm vir} 
= f^!\left(
\be(\cV_{0,l+2,d}') \cap [\cX_{0,l+2,d}]^{\rm vir}\right).  
\end{align*} 
In the last step we used the exact sequence 
\eqref{eq:exactseq_V'}. The conclusion follows. 
\end{proof} 

Using 
$\phi_i \bullet_\tau^{\be_\lambda}
= -(z\partial_{\tau^i} \bL^{\be_\lambda}(\tau,z)) 
(\bL^{\be_\lambda}(\tau,z))^{-1}$, we obtain the 
following corollary. 
\begin{corollary}  
\label{cor:iota_quantumhom} 
Under the same assumption, 
the equivariant Euler twisted quantum 
product $\bullet_\tau^{\be_\lambda}$ 
has the non-equivariant limit $\bullet_\tau^\be$ 
and we have 
\[
\iota^*(\alpha \bullet_\tau^\be \beta) 
= (\iota^*\alpha) \bullet_{\iota^*\tau} (\iota^* \beta)
\Bigr|_{H_2(Y;\Z) \to H_2(\cX;\Z)} 
\] 
where $\alpha,\beta \in H^*_{\rm orb}(\cX)$ and 
$\bullet_{\iota^*\tau}$ in the right-hand side 
denotes the untwisted quantum product of $\cY$.   
\end{corollary} 

\subsection{The Specialization at $Q=1$} 
The divisor equation (\cite[Theorem 8.3.1]{AGV}) 
shows that the Novikov parameter $Q$ is actually 
redundant in the product $\bullet^\bc_\tau$ 
\eqref{eq:twistedquantumprod}.  
Writing 
\begin{equation} 
\label{eq:tau_decomp}
\tau = \tau_{0,2} + \tau', \quad 
\tau_{0,2} \in H^2(\cX_0), \quad 
\tau' \in \bigoplus_{p\neq 2}H^p(\cX) 
\oplus \bigoplus_{v\in \sfT'} H^*(\cX_v), 
\end{equation} 
we have 
\[
(\alpha\bullet^\bc_\tau \beta, \gamma)_{\rm orb}^{\bc} 
= \sum_{l\ge 0} \sum_{d\in \Eff_\cX} 
\corr{\alpha,\beta,\gamma,
\tau',\dots,\tau'}_{0,l+3,d}^{\bc}  
\frac{e^{\pair{\tau_{0,2}}{d}}Q^d}{l!}. 
\]
Therefore the parameter $Q$ plays the same 
role as $e^{\tau_{0,2}}$. 
We define 
\[
\circ^\bc_\tau := \bullet^\bc_\tau|_{Q=1}. 
\]
The new product $\circ_\tau^\bc$ is a formal power 
series in $\tau'$ and a formal \emph{Fourier} 
series in $\tau_{0,2}$. 
Similarly, by the divisor equation, 
the fundamental solution \eqref{eq:twistedfundsol} 
can be specialized to $Q=1$. 
Writing $L^\bc(\tau,z) := \bL^\bc(\tau,z)|_{Q=1}$, 
we have  
\begin{equation}
\label{eq:fundsol} 
(L^\bc(\tau,z) \alpha,\beta)_{\rm orb}^{\bc} 
= (e^{-\tau_{0,2}/z} \alpha, \beta)_{\rm orb}^{\bc}  
+ \sum_{\substack{(d,l)\neq (0,0) \\ d\in \Eff_\cX, l\ge 0}} 
\corr{\frac{e^{-\tau_{0,2}/z} \alpha}{-z-\psi}, 
\tau',\dots,\tau',\beta}_{0,l+2,d}^{\bc}  
\frac{e^{\pair{\tau_{0,2}}{d}}}{l!}.  
\end{equation} 
Here the action of $\tau_{0,2}$ on $H^*_{\rm orb}(\cX)$ 
is defined by 
$\tau_{0,2} \cdot \alpha = \pr^*(\tau_{0,2}) \cup \alpha$
where $\pr \colon \cI \cX \to \cX$ is the natural projection. 
The classical limit $Q=\tau=0$ corresponds, 
after the specialization $Q=1$, to the limit 
$\tau'= 0$ and $e^{\pair{\tau_{0,2}}{d}} \to 0$ 
for all nonzero $d\in \Eff_\cX$. 
This is called the \emph{large radius limit}. 

%We will work with 
%$\bullet_\tau$ when we regard the quantum 
%cohomology as a formal object 
%and with $\circ_\tau$ when we want to 
%regard it as an analytic object. 

\section{$\hGamma$-Integral Structure in Quantum Cohomology} 
\label{sec:Gammaint}
In this section we review the quantum $D$-module 
for stacks and its $\hGamma$-integral structure following 
\cite{Iritani:int}. See also \cite{Iritani:RIMS} 
for a review. 
\subsection{Untwisted Quantum $D$-Module with $Q=1$}
\label{subsec:untwQDM}
We denote by $\circ_\tau:=\circ_\tau^{\bc}|_{\bs=0}$ 
the quantum product of the untwisted theory of $\cX$ 
specialized to $Q=1$. 
In all the examples we treat in our paper,  
it turns out a posteriori that 
the quantum product $\circ_\tau$ is convergent in $\tau$. 
So henceforth we 
assume that $\circ_\tau$ is convergent over the 
region $U\subset H_{\rm orb}^*(\cX)$ 
containing the set 
\[
\left\{ \tau \in H^*_{\rm orb}(\cX) \;\big|\; 
\|\tau'\| \le e^{-M}, 
\ 
\Re(\pair{\tau_{0,2}}{d}) \le -M \   
\forall d \in \Eff_\cX\setminus \{0\} 
\right\}
\]
for some $M>0$. Here $\|\cdot \|$ is 
a certain norm on $H_{\rm orb}^*(\cX)$ 
and we used the decomposition \eqref{eq:tau_decomp}.  
The region $U$ is considered as a neighborhood 
of the large radius limit point. 

Let $(\tau,z)$ denote a general point on 
$U\times \C$ and $(-)\colon U\times \C \to U\times \C$ 
be the map sending $(\tau,z)$ to $(\tau,-z)$. 
In the untwisted theory we can extend the Dubrovin 
connection in the $z$-direction. 

\begin{definition}[{\cite[Definition 2.2]{Iritani:int}}]
\label{def:QDM} 
The \emph{quantum $D$-module} $QDM(\cX)$ 
is the triple $(F, \nabla, (\cdot,\cdot)_F)$ 
consisting of the trivial holomorphic vector 
bundle $F:= H^*_{\rm orb}(\cX) \times 
(U\times \C) \to (U\times \C)$, 
the meromorphic flat connection of $F$ 
\begin{equation*} 
\label{eq:Dubrovinconn} 
\nabla := d + \frac{1}{z} \sum_{i=1}^N 
(\phi_i \circ_\tau) d\tau^i + 
\left(- \frac{1}{z} (E\circ_\tau) + \frac{\deg}{2} \right) 
\frac{dz}{z}   
\end{equation*} 
and the pairing 
$(\cdot,\cdot)_F \colon (-)^* \cO(F) \otimes \cO(F) 
\to z^n \cO_{U\times \C}$  
defined by 
\[
(\alpha,\beta)_F := (2\pi\iu z)^n (\alpha,\beta)_{\rm orb} 
\qquad \text{for }\alpha \in F_{(\tau,-z)}, \ \beta \in F_{(\tau,z)}. 
\]
% induced by the orbifold Poincar\'{e} pairing 
% $F_{(\tau,-z)} \times F_{(\tau,z)}  
% = H^*_{\rm orb}(\cX) \times H^*_{\rm orb}(\cX) 
% \to \C$. 
Here $E \in \cO(F)$ is the \emph{Euler vector field}  
\[
E := c_1(T\cX) + \sum_{i} \left(1- \frac{1}{2} \deg \phi_i\right) 
\tau^i \phi_i 
\]
and $\deg$ denotes the degree as a class in 
$H_{\rm orb}^*(\cX)$. (In the definition of $\nabla$, 
$\frac{\deg}{2}$ should be understood as an element 
of $\End(H_{\rm orb}^*(\cX))$.)  
The connection $\nabla$ is called 
the (extended) \emph{Dubrovin connection}. 
It has poles of order $\le 2$ along $z=0$.  
The pairing $(\cdot,\cdot)_F$ 
is flat with respect to $\nabla$. 
When we refer to $QDM(\cX)$ as a $D$-module, 
we consider the action of the ring $\cO_U[z]\langle
z\partial_1,\dots,z \partial_N \rangle$ 
of differential operators on $\cO(F)$ given 
by $z\partial_i \mapsto z \nabla_i$. 
\end{definition}
\begin{remark} 
\label{rem:change} 
We work with the different conventions for $\nabla$ and $(\cdot,\cdot)_F$ 
from \cite{Iritani:int} to get a better match 
with the B-side. For the 
flat connection $\nabla^{\rm old}$ and the pairing 
$(\cdot,\cdot)^{\rm old}_F$ in \cite{Iritani:int}, 
we have $\nabla = \nabla^{\rm old} + \frac{n}{2} \frac{dz}{z}$ 
and $(\cdot,\cdot)_F = (2\pi\iu z)^n (\cdot,\cdot)^{\rm old}_F$, 
where $n= \dim_\C \cX$. 
In what follows, we will translate the contents 
in \cite{Iritani:int} in this new convention, 
but we will not remark the difference every time. 
\end{remark} 
\begin{remark}
\label{rem:gen_Hodge} 
The quantum $D$-module can be considered 
as a variation of \emph{generalized} Hodge structure. 
Generalizations of Hodge structure have been studied 
by many people and referred to in various ways:   
\emph{semi-infinite Hodge structure} \cite{Barannikov:qpI, 
Iritani:int},  
\emph{TERP structure} \cite{Hertling:tt*} 
and \emph{non-commutative Hodge structure} \cite{KKP} etc. 
\end{remark} 

The quantum $D$-module has a certain symmetry 
which we called the \emph{Galois action} 
in \cite{Iritani:int}. This comes from the 
divisor equation and the monodromy constraints   
for orbifold stable maps. 
Let $H^2(\cX;\Z)$ denote the sheaf cohomology on the 
topological stack $\cX$ which classifies topological 
orbifold line bundles. For $\xi\in H^2(\cX;\Z)$, 
let $\cL_\xi$ be the corresponding orbifold line bundle,  
$\xi_0 \in H^2(\cX;\Q)$ denote the image of $\xi$ 
and $f_v(\xi) \in [0,1)\cap \Q$ be the rational number 
such that the stabilizer along $\cX_v$ acts on 
fibers of $\cL_\xi$ by $\exp(2\pi\iu f_v(\xi))$. 
(The number $f_v(\xi)$ is called the 
\emph{age} of $\cL_\xi$ along $\cX_v$.) 
Define the map $G(\xi) \colon 
H^*_{\rm orb}(\cX) \to H^*_{\rm orb}(\cX)$ by 
\begin{equation} 
\label{eq:Galois} 
G(\xi) ( \tau_0 \oplus \bigoplus_{v\in \sfT'} \tau_v ) 
= (\tau_0 - 2\pi\iu \xi_0) \oplus 
\bigoplus_{v\in\sfT'} e^{2\pi\iu f_v(\xi)} \tau_v  
\end{equation} 
where $\tau_v \in H^*(\cX_v)$. 
Consider the following bundle isomorphism of $F$ 
\begin{align}
\label{eq:Galois_bundleisom}
\begin{split}  
G^F(\xi) \colon 
H_{\rm orb}(\cX) \times (U\times \C) & 
\longrightarrow 
H_{\rm orb}(\cX) \times (U\times \C), \\ 
(\alpha, (\tau, z)) 
& \longmapsto (dG(\xi) \alpha, (G(\xi) \tau, z))  
\end{split} 
\end{align} 
where $dG(\xi) \in \End(H_{\rm orb}^*(\cX))$ 
is  the differential of $G(\xi)$. 
\begin{proposition}[{\cite[Proposition 2.3]{Iritani:int}}] 
The bundle isomorphism $G^F(\xi)$ preserves the connection 
$\nabla$ and the pairing $(\cdot,\cdot)_F$. 
This defines the $H^2(\cX,\Z)$-action on $QDM(\cX)$
and $QDM(\cX)$ descends\footnote{We can assume that 
$U$ is invariant under the action of $H^2(\cX;\Z)$. } to the quotient 
$(U/H^2(\cX,\Z))\times \C$.  
\end{proposition} 

% Because the Dubrovin connection is flat, we have a 
% basis of solutions to the differential equation $\nabla s =0$. 
% This is called a fundamental solution. 
% We have a standard fundamental solution given by 
% descendant Gromov-Witten invariants. 
% For $\tau_0 \in H^*(\cX_0)$ and $\alpha \in H^*_{\rm orb}(\cX)$, 
% we define the action of $\tau_0$ by 
% $\tau_0 \cdot \alpha := \pr^*(\tau_0) \cup \alpha$,  
% where $\pr \colon \cI\cX \to \cX$ is the natural projection. 
% Let $\{\phi^k\}$ be the basis 
% dual to $\{\phi_i\}$ with respect to 
% $(\cdot,\cdot)_{\rm orb}$, i.e.\  
% $(\phi^k, \phi_i)_{\rm orb}= \delta^k_i$. 
% Define an $\End(H_{\rm orb}^*(\cX))$-valued 
% function $L(\tau,z)$ on $U\times \C^\times$ by 
% \begin{equation}
% \label{eq:fundsol} 
% L(\tau,z) \phi_i = e^{-\tau_{0,2}/z} \phi_i 
% - \sum_{\substack{(d,l)\neq (0,0) \\ d\in \Eff_\cX, 
% 1\le k\le N}} 
% \phi^k \corr{\phi_k, \tau',\dots,\tau',
% \frac{e^{-\tau_{0,2}/z} \phi_i}{z+\psi}}_{0,l+2,d}^\cX 
% e^{\pair{\tau_{0,2}}{d}}   
% \end{equation} 
% where $\tau = \tau_{0,2} + \tau'$ is the decomposition in 
% \eqref{eq:tau_decomp} and $1/(z+\psi)$ in the correlator 
% should be understood as the series 
% $\sum_{k=0}^\infty (-1)^k z^{-k-1} \psi^k$. 

The solution to the extended quantum 
differential equation $\nabla s=0$ is given by 
the fundamental solution $L(\tau,z) := L^{\bc}(\tau,z)|_{\bs =0}$ 
in \eqref{eq:fundsol} 
multiplied by $z^{-\frac{\deg}{2}}z^\rho$. 
\begin{proposition}[{\cite[Proposition 2.4]{Iritani:int}}]
\label{prop:fundsol}  
Set $\rho := c_1(T\cX)$ and define 
\[
z^{-\frac{\deg}{2}} z^\rho := \exp\left(- \frac{\deg}{2} \log z\right) 
\exp(\rho \log z).  
\]
Then $s_i(\tau,z) = 
L(\tau,z) z^{-\frac{\deg}{2}} z^\rho \phi_i$, $i=1,\dots,N$ form 
a basis of (multi-valued) $\nabla$-flat sections. 
Each $s_i$ is characterized by the asymptotic 
initial condition 
$s_i(\tau,z) \sim z^{-\frac{\deg}{2}} z^{\rho} e^{-\tau_{0,2}} \phi_i$ 
in the large radius limit. 
\end{proposition} 
Note that $L(\tau,z)$ is convergent on $U\times \C^*$ 
so far as the quantum product $\circ_\tau$ is analytic 
on $U$ since it is a solution to the quantum 
differential equation.

% \begin{remark} 
% The section $L(\tau,z) \phi_i$ is $\nabla$-flat 
% in the $\tau$-direction but not in the $z$-direction. 
% We have the following property for the pairing 
% \cite[Proposition 2.4]{Iritani:int} 
% \begin{equation} 
% \label{eq:unitarityofL} 
% (L(\tau,-z) \phi_i, L(\tau,z) \phi_j)_{\rm orb} = 
% (\phi_i,\phi_j)_{\rm orb}. 
% \end{equation} 
% The convergence of $L(\tau,z)$ follows from the 
% fact that it satisfies the differential 
% equation and the assumption that $\circ_\tau$ 
% is convergent. 
% \end{remark} 

\subsection{$\hGamma$-Integral Structure} 
Let $\Sol(\cX)$ denote the space of multi-valued 
flat sections for $\nabla$. By Proposition
\ref{prop:fundsol}, it is a $\C$-vector 
space spanned by $L(\tau,z)z^{-\frac{\deg}{2}} z^\rho \phi_i$, 
$1\le i\le N$. We will introduce a $\Z$-lattice 
$\Sol(\cX)_\Z$ 
in the space $\Sol(\cX)$ using the $K$-group. 
A similar rational structure was introduced also 
by Katzarkov-Kontsevich-Pantev \cite{KKP}. 
Define a pairing $(\cdot,\cdot)_\Sol \colon 
\Sol(\cX) \otimes \Sol(\cX) \to \C$ by 
\[
(s_1,s_2)_{\Sol} := (s_1(\tau, e^{\pi\iu}z), s_2(\tau,z))_{\rm orb}.  
\]
Here $s_1(\tau,e^{\pi\iu}z)$ denotes the analytic continuation 
of $s_1(\tau,z)$ along the path $[0,1]\ni 
\theta \mapsto e^{\pi\iu\theta} z$. Since $s_1, s_2$ 
are flat sections, the right-hand side of the above formula 
does not depend on $\tau$ and $z$. 
Note that $(\cdot,\cdot)_{\Sol}$ is neither symmetric 
nor anti-symmetric in general. 
It is symmetric (resp.\ anti-symmetric) when $\cX$ is an 
even (resp.\ odd) dimensional Calabi-Yau orbifold. 
The Galois action on $QDM(\cX)$ induces 
the following automorphism $G^\Sol(\xi)$ of 
$\Sol(X)$ for $\xi \in H^2(\cX;\Z)$: 
$(G^\Sol(\xi) s)(\tau,z):= dG(\xi) s(G(\xi)^{-1}\tau,z)$ 
for $s\in \Sol(\cX)$. 

Let $K(\cX)$ be the Grothendieck group of 
topological orbifold vector bundles on $\cX$. 
In the following, 
we could also use the Grothendieck group $K_{\rm alg}(\cX)$ 
of \emph{algebraic} vector bundles. 
Our integral structure depends only on the 
Chern character image of the $K$-group, 
so the algebraic $K$-group defines a subgroup 
of $\Sol(\cX)_\Z$. 
For an orbifold vector bundle $\cE$, take  
its pull-back $\pr^* \cE$ to $\cI\cX$ 
($\pr\colon \cI\cX \to \cX$ is the natural map) 
and consider the eigenbundle decomposition of 
$\pr^*\cE|_{\cX_v}$ with respect to the 
stabilizer action: 
\[
\pr^*\cE|_{\cX_v} = \bigoplus_{0\le f<1} 
(\pr^*\cE)_{v,f}
\]
where $(\pr^*\cE)_{v,f}$ is the piece on which 
the stabilizer of $\cX_v$ acts by $\exp(2\pi\iu f)$. 
The Chern character map $\tch\colon K(\cX) \to 
H^*(\cI\cX)$ is defined by 
\[
\tch(\cE) := \bigoplus_{v\in \sfT} \sum_{0\le f <1} 
e^{2\pi\iu f} \ch((\pr^*\cE)_{v,f}). 
\]
Let $\delta_{v,f,i}$, $i=1,\dots,l_{v,f}$ 
be the Chern roots of $(\pr^*\cE)_{v,f}$, where 
$l_{v,f} =\rank((\pr^*\cE)_{v,f})$. 
The $\hGamma$-class of $\cE$ is defined to be 
\begin{equation*}
%\label{eq:Gammaclass} 
\hGamma(\cE) := \bigoplus_{v\in \sfT} \prod_{0\le f<1} 
\prod_{i=1}^{l_{v,f}} \Gamma(1-f+\delta_{v,f,i}) 
\in H^*(\cI\cX). 
\end{equation*} 
Here the $\Gamma$-function in the right-hand side 
should be expanded in Taylor series at $1-f>0$.  
This is a multiplicative \emph{transcendental} 
characteristic class. 
We write $\hGamma_\cX := \hGamma(T\cX)$. 
For simplicity we assume that $\cX$ has no 
generic stabilizers, as this is true for 
our later examples. 

\begin{definition}[{\cite[Definition 2.9, Proposition 2.10, 
Remark 2.11]{Iritani:int}, \cite[Definition 3.2]{KKP}}]
\label{def:intstr} 
Define the \emph{$K$-group framing}\footnote
{The convention here is different from \cite{Iritani:int}.} 
$\frs \colon K(\cX) \to \Sol(\cX)$ 
of the space $\Sol(\cX)$ by 
\begin{align} 
\label{eq:Psi}
\begin{split} 
\frs (\cE)(\tau,z) &:= (2\pi \iu)^{-n} 
L(\tau,z) z^{-\frac{\deg}{2}} z^\rho \Psi(\cE) \\ 
\text{where } & \Psi(\cE) :=  %(2\pi)^{-\frac{n}{2}} 
\hGamma_\cX \cup (2\pi\iu)^{\frac{\deg_0}{2}} 
\inv^* \tch(\cE).  
\end{split} 
\end{align} 
Here $\deg_0$ denotes the degree \emph{without} the age shift, 
i.e.\ we define 
$(2\pi\iu)^{\frac{\deg_0}{2}}|_{H^{2k}(\cI\cX)} := (2\pi\iu)^k$ 
and $\hGamma_\cX \cup$ is the cup product in $H^*(\cI\cX)$. 
The \emph{$\hGamma$-integral structure} $\Sol(\cX)_\Z
\subset \Sol(\cX)$ is defined to be the image of $\frs$. 
This satisfies the following properties. 
\begin{enumerate}
\item $\Sol(\cX)_\Z$ is a lattice in $\Sol(\cX)$, i.e.\ 
$\Sol(\cX) = \Sol(\cX)_\Z\otimes_\Z \C$. 

\item We have $G^\Sol(\xi) (\frs(\cE)) = 
\frs(\cE \otimes \cL_\xi^\vee)$ for $\xi \in H^2(\cX;\Z)$.  
In particular the Galois action preserves the lattice 
$\Sol(\cX)_\Z$. 

\item The pairing $(\cdot,\cdot)_\Sol$ takes values in 
$\Z$ on $\Sol(\cX)_\Z$. For holomorphic vector bundles 
$\cE_1, \cE_2$, one has $(\frs(\cE_1),\frs(\cE_2))_\Sol 
= (-1)^n \chi(\cE_2,\cE_1) := \sum_{i=0}^n (-1)^{i+n} 
\dim \Ext^i(\cE_2,\cE_1)$. 
\end{enumerate} 
\end{definition} 
The last part (iii) of the properties follows from 
Kawasaki-Riemann-Roch \cite{Kawasaki:RR, Toen} 
and the fact that the $\hGamma$-class is roughly 
the half of the Todd class. In fact, for a smooth 
variety $X$, the $\hGamma$-class and the Todd class 
are related by 
\begin{equation}
\label{eq:half}
((-1)^{\frac{\deg_0}{2}}\hGamma_X)  
\cdot \hGamma_X \cdot e^{\pi\iu c_1(X)} 
= (2\pi\iu)^{\frac{\deg_0}{2}} \Td(TX)
\end{equation} 
thanks to the functional equality 
$\Gamma(1-z) \Gamma(1+z) = \pi z/\sin(\pi z)$. 
(For an orbifold the relationship is more complicated. 
See \cite[p.124]{Iritani:RIMS}.) 

% Using the $\hGamma$-integral structure, 
% we define the quantum cohomology central charge. 
% It is considered as an A-model analogue of periods 
% or oscillatory integrals. 
\begin{definition}
For $\cE\in K(\cX)$ and a section $\phi(\tau,z)
\in \cO(F)$ 
of the quantum $D$-module of $\cX$, 
we define the \emph{A-period} $\Pi(\phi,\cE)$ to be 
the multi-valued function on $U\times \C^\times$ 
\begin{equation} 
\label{eq:A-period}
\Pi(\phi,\cE)(\tau,z) := (\phi(\tau,-z),\frs(\cE)(\tau,z))_{F}.    
\end{equation} 
The special case $Z(\cE):=(2\pi\iu)^{-n}\Pi(\unit,\cE)$, 
$n=\dim_\C \cX$ is the  
\emph{quantum cohomology central charge} of $\cE$ 
defined in \cite{Iritani:int}. 
\end{definition}  
Under mirror symmetry the flat section $\frs(\cE)$ 
should correspond to a Gauss-Manin constant cycle $C_\cE$ 
and the above pairing $\Pi(\phi,\cE)$ 
to the integration of the de Rham form mirror to 
$\phi$ over $C_\cE$. 
The unit section $\unit$ should correspond to 
a holomorphic (oscillatory) volume form. 
Using $L(\tau,z)^\dagger = L(\tau,-z)^{-1}$ 
\eqref{eq:unitarityofbL}, 
we can rewrite the A-periods 
in terms of the inverse fundamental solution.  
\begin{equation} 
\label{eq:A-period2} 
\Pi(\phi,\cE)(\tau,z) = 
\left( L(\tau,-z)^{-1} \phi(\tau,-z), \;
z^{n-\frac{\deg}{2}} z^\rho \Psi(\cE)
\right)_{\rm orb}.   
\end{equation} 
In particular, $Z(\cE)$ is a component of the $J$-function: 
\begin{equation*} 
%\label{eq:centralcharge2} 
Z(\cE) = \frac{1}{(2\pi\iu)^n} 
\left( J(\tau,-z), \;
z^{n-\frac{\deg}{2}} z^\rho \Psi(\cE)
\right)_{\rm orb},   
\end{equation*}
where $J(\tau,z) = L(\tau,z)^{-1} \unit$ is the untwisted 
$J$-function of $\cX$ with $Q=1$. 

\section{Mirror Theorem for Toric Complete Intersections} 
In this section we state a Givental-style 
mirror theorem for complete intersections 
in toric orbifolds. By the mirror theorem 
we can calculate the $J$-function 
or the fundamental solution in terms of
explicit hypergeometric series. 

\subsection{Notation on Toric Orbifolds} 
\label{subsec:toricorb}
\emph{Toric orbifolds} or \emph{toric Deligne-Mumford stacks}  
were introduced by Borisov-Chen-Smith \cite{Borisov-Chen-Smith} 
in terms of a \emph{stacky fan}. 
Here we fix notation for toric orbifolds 
and state basic facts.  
We only consider compact weak Fano toric orbifolds 
without generic stabilizers. 
See \cite{Danilov:toric, Oda, Borisov-Chen-Smith} 
for the basics of toric varieties and stacks. 
A similar but more detailed account was given in 
\cite[Section 3.1]{Iritani:int} with a little 
different notation. 

Let $\bN\cong \Z^n$ be a free abelian group. 
Set $\bN_\R = \bN \otimes_\Z \R$. 
Let $\Delta\subset \bN_\R$ be an integral convex polytope 
containing the origin $0$ in its interior. 
We choose a stacky fan $(\Sigma, \beta)$ 
on $\bN$ \emph{adapted to} $\Delta$. 
It consists of the data 
\begin{itemize} 
\item 
a rational simplicial fan $\Sigma$ in the 
vector space $\bN_\R$; 
\item 
a homomorphism $\beta \colon \Z^m \to \bN$ 
such that $\{\R_{\ge 0} b_1,\dots, \R_{\ge 0} b_m\}$ 
is the set $\Sigma^{(1)}$ of one-dimensional cones 
of $\Sigma$, where $b_i = \beta(e_i)$ is the image 
of the standard basis $e_i \in \Z^m$  
\end{itemize} 
which are adapted to $\Delta$ in the sense 
that $\Delta$ is the convex hull of 
$b_1,b_2,\dots,b_m$ and that $b_1,\dots,b_m$ 
are on the boundary of $\Delta$. 
We call $\Delta$ the \emph{fan polytope}. 
These data give rise to a weak Fano (i.e.\ $c_1(\cX)$ is nef) 
toric orbifold $\cX$. The coarse moduli space 
$X$ of $\cX$ is the toric variety associated with the fan $\Sigma$. 
We furthermore assume that 
\begin{itemize} 
\item the fan $\Sigma$ admits a strictly convex 
piecewise linear function\footnote{A piecewise linear 
function is a continuous function on $\bN_\R$ which is 
linear on each cone of $\Sigma$. See \cite{Oda} 
for the (strict) convexity.}
$\varphi \colon N_\R \to \R$; 
\item the set $\Delta\cap \bN$ generate $\bN$ 
as a $\Z$-module. 
\end{itemize} 
The first condition means that the underlying 
toric variety $X$ is projective. 
The second condition\footnote{
This second assumption is satisfied if 
$\pi_1^{\rm orb}(\cX)$ is trivial; 
in particular if $\cX$ is a weighted projective space. 
We will state the toric mirror theorem without this 
assumption in \cite{CCIT:toric}. 
The results in this paper should be generalized 
also without this assumption, but we 
will stick to this case for simplicity.}
ensures that the quantum $D$-module of $\cX$ 
over the small parameter space $H^{\le 2}_{\rm orb}(\cX)$ 
is generated by the $I$-function  
(see \cite[Lemma 4.7]{Iritani:int}).  
Essentially the same assumption was 
made in \cite{Iritani:int} 
(see Remark 3.4 \emph{ibid}). 
We usually identify a cone $\sigma$ of $\Sigma$ 
with the subset $\{i \;|\; b_i \subset \sigma\}$ 
of $\{1,\dots,m\}$. 

\begin{remark} 
Borisov-Chen-Smith \cite{Borisov-Chen-Smith}
allowed $\bN$ to have torsion and the torsion part of $\bN$ 
equals the group of generic stabilizers of $\cX$.   
In this case the mirror of $\cX$ becomes 
disconnected \cite{Iritani:int}. 
We will restrict to the free $\bN$ to reduce  
technical complications. 
\end{remark}
Take a subset $\{b_{m+1},\dots,b_{m+s}\}$ of 
$(\bN \cap \Delta) \setminus \{b_1,\dots,b_m\}$ 
such that $b_1,\dots,b_{m+s}$ generate $\bN$ as 
an abelian group.  
These are called \emph{extended ray vectors}. 
They define an \emph{extended stacky fan}  
in the sense of Jiang \cite{Jiang:toricstackbundle}. 
%We can assume that $b_{m+s} = 0$ without loss of generality.  
Let $\hbeta\colon \Z^{m+s} \to \bN$ be the 
homomorphism sending the standard basis vectors 
$e_1,\dots,e_{m+s}$ to $b_1,\dots,b_{m+s}$. 
Then $\hbeta$ is surjective by the assumption. 
Define $\LL := \Ker \hbeta$. 
The \emph{(extended) fan sequence} is the exact sequence: 
\begin{align*} 
\begin{CD} 
0 @>>> \LL @>>> \Z^{m+s} @>{\hbeta}>> \bN @>>> 0   
\end{CD}
\end{align*} 
and the \emph{(extended) divisor sequence} 
is its dual: 
\[
\begin{CD} 
0 @>>> \bM @>{\hbeta^*}>> (\Z^{m+s})^* @>{D}>> \LL^* @>>> 0.   
\end{CD} 
\]
Here $\bM := \Hom(\bN,\Z)$. 
Let $D_i = D(e_i^*)\in \LL^*$ be the image of 
the standard basis $e_i^* \in (\Z^{m+s})^*$. 
The Picard group $\Pic(\cX)$ on the \emph{stack} $\cX$ 
is given by 
\[
\Pic(\cX) \cong H^2(\cX;\Z) \cong 
\LL^*\Big/\sum_{i=m+1}^{m+s} \Z D_i. 
\]
The image $\ov{D}_i$ of $D_i$ in $\Pic(\cX)$ is the 
class of a torus invariant divisor. 
We call $D_i$ the \emph{extended toric divisor class}. 
The anticanonical class is given by $\rho:= c_1(\cX) 
= -K_\cX = \sum_{i=1}^m \ov{D}_i$. 
The \emph{extended anticanonical class} is defined by 
$\hrho := \sum_{i=1}^{m+s} D_i$. 
Every element of $\Pic(\cX)$ is represented by  
an integral linear combination 
of toric divisors $\ov{D}_1,\dots,\ov{D}_m$. 
For an expression $\xi = \sum_{i=1}^m n_i \ov{D}_i$, 
define a piecewise linear function $\varphi_\xi\colon \bN_\R \to \R$ 
by $\varphi_\xi(b_i) = n_i$ for $1\le i\le m$.  
The function $\varphi_\xi$ is ambiguous 
up to an integral linear function in $\bM = \Hom(\bN,\Z)$. 
We have the following: 
\begin{itemize} 
\item $\xi$ is nef (resp.\ ample) 
$\Longleftrightarrow$ $\varphi_\xi$ is convex (resp.\ strictly convex); 
\item For $v\in \Boxop$, $\{\varphi_\xi(v)\}$ is 
the age $f_v(\xi)$ of the line bundle $\cL_\xi$ along $\cX_v$. 
%\item $\xi$ is a Cartier divisor on the coarse moduli space $X$ 
%
%$\Longleftrightarrow$ $\varphi_\xi$ has an integral slope 
%(in $\bM$) on each maximal cone. 
\end{itemize} 
%($H_2(\cX;\Q)$ is also identified with $\Ker(\beta\otimes \Q)$.)
% fits into the exact sequence 
% (the non-extended fan sequence):  
% \[
% \begin{CD} 
% 0 @>>> H_2(\cX;\Q) @>>> \Q^m @>{\beta\otimes \Q}>> \bN_\Q @>>> 0.  
% \end{CD} 
% \] 
Define the set $\Boxop$ by 
\[
\Boxop := \Bigl\{ v \in \bN \;\Big|\; \exists \sigma \in \Sigma, 
\ 0\le \exists c_i<1, \ v= \sum_{i\in \sigma} c_i b_i \Bigr\}.  
\] 
This parametrizes connected components of $\cI\cX$ 
\cite{Borisov-Chen-Smith}.  
For $v\in \Boxop$, let $\cX_v$ denote 
the corresponding component of $\cI\cX$ and 
$\unit_v\in H^0(\cX_v) \subset H^{2\age(v)}_{\rm orb}(\cX)$ 
denote the unit class supported on $\cX_v$. 
Here $\age(v)$ is given by 
$\age(v) = \sum_{i\in \sigma} c_i$ when 
$v$ is written as 
$v = \sum_{i \in \sigma} c_i b_i$ 
for some cone $\sigma\in \Sigma$ and 
$c_i\ge 0$. 
The extended divisors $D_{m+1},\dots,D_{m+s}$ correspond 
to the classes $\unit_{b_{m+1}}, \dots, 
\unit_{b_{m+s}}$ in $H^{\le 2}_{\rm orb}(\cX)$. 
%We have $\unit_{b_{m+s}} = \unit_0 = \unit$. 

Note that $H_2(\cX;\Q) \cong (\bigoplus_{i=m+1}^{m+s} \Q D_i)^\perp
\subset \LL_\Q := \LL\otimes \Q$.  
We see that $H_2(\cX;\Q)$ 
has a canonical complementary subspace in $\LL_\Q$. 
For $m+1\le j \le m+s$, 
$b_j$ is contained in a cone $\sigma$ 
of $\Sigma$ and we can write $b_j = \sum_{i\in \sigma} 
c_{ji} b_i$ for some $c_{ji}\ge 0$. 
Then $\delta_j :=e_j - \sum_{i\in \sigma} c_{ji} e_i
\in \Q^{m+s}$ belongs to $\LL_\Q$. 
We have 
\begin{equation} 
\label{eq:LQdecomp} 
\LL_\Q = H_2(\cX;\Q) \oplus \bigoplus_{j=m+1}^{m+s} \Q \delta_j.
\end{equation} 
The elements $\delta_{m+1},\dots,\delta_{m+s}$ are 
dual to $D_{m+1},\dots,D_{m+s}$ 
and regarded as orbifold \emph{homology} classes (of degree $\le 2$).  
Set $\NE_{\cX,\sigma} := 
\{ d \in H_2(\cX;\R) \;|\;
\forall i \in \{1,\dots,m\} \setminus \sigma, \ 
\pair{\ov{D}_i}{d} \ge 0\}$ for a cone $\sigma$.  
The \emph{Mori cone} $\NE_\cX\subset H_2(\cX;\R)$ 
is given by 
\[
\NE_\cX = \sum_{\sigma \in \Sigma} \NE_{\cX,\sigma}. 
\]
The \emph{extended Mori cone} 
$\hNE_\cX\subset \LL_\R := \LL\otimes_\Z \R$ 
is defined to be %(see \cite[Lemma 3.2]{Iritani:int} 
%for the dual statement): 
\[
\hNE_{\cX} := \NE_{\cX} + \sum_{m+1\le j \le m+s} 
\R_{\ge 0} \delta_j.  
\]
% Dually, we obtain a canonical splitting of the 
% projection $\hbL^*_\Q \to \LL^*_\Q$: 
% \[
% \hbL^*_\Q = \Bigl(
% \bigoplus_{j=m+1}^{m+s} \Q \delta_j\Bigr)^\perp 
% \oplus \bigoplus_{j=m+1}^{m+s} \Q \hD_j \cong 
% \LL^*_\Q \oplus \bigoplus_{j=m+1}^{m+s} \Q \hD_j. 
% \]
% Henceforth we regard $\LL^*_\Q = H^2(\cX;\Q)$ as a subspace 
% of $\hbL^*_\Q$ using this canonical splitting. 
% The K\"{a}hler cone ... 
For $v\in \Boxop$, we define 
$\K_v$ to be the subset of $\Q^{m}\times \Z^s
\subset \Q^{m+s}$ 
consisting of all $d \in \Q^m\times \Z^s$ 
such that $\sum_{i=1}^{m+s} d_i b_i +v =0$ 
and that $\{1\le i\le m\, |\, d_i\notin\Z\}$ 
is a cone of $\Sigma$. 
Let us write $v = \sum_{i\in \sigma} c_i b_i$ 
for some cone $\sigma$ and $c_i\in [0,1)$
and set $c_i = 0$ for $i\notin \sigma$. 
Then we have a relation $\sum_{i=1}^{m+s} (d_i + c_i)b_i =0$ 
for $d \in \K_v$. 
We denote by $d+v$ the element of $\LL_\Q$ defined 
by this relation. 
% \begin{align*} 
% \K_v &:= 
% \left\{ \nu \in \Q^{m}\times \Z^s  
% \; \Big|\; \sum_{i=1}^{m+s} \nu_i b_i + v = 0, \ 
% \exists \sigma \in \Sigma, \ 
% \forall i\in \{1,\dots,m \} \setminus \sigma, \ 
% \nu_i \in \Z \right\}, \\ 
% \K_{\rm eff} & := 
% \left\{ d\in \LL_\Q \;\big|\; 
% \exists \sigma \in \Sigma, \ 
% \forall i \in \{1,\dots,m+s\} \setminus \sigma,\ 
% \pair{D_i}{d} \in \Z_{\ge 0}
% \right\}. 
% \end{align*} 
The lattice $\LL$ acts on $\K_v$ by addition 
and $\K_0 \subset \LL_\Q$.  
We define the \emph{reduction function} 
$\{- \, \cdot \, \} \colon \K_v \to \Boxop$ by 
\[
\{-d\} := \sum_{i=1}^m \{ - d_i \} b_i 
\]
where $\{r\}$ denote the fractional part of $r$. 
Because $\sum_{i=1}^{m+s} d_i b_i + v=0$, 
we have $\{-d\} = \sum_{i=1}^{m+s} \ceil{d_i} b_i + v$ 
and so $\{-d\} \in \bN$.  
The reduction function in fact induces an 
isomorphism $\K_v /\LL \cong \Boxop$. 

\subsection{Mirror Theorem I: Toric Orbifolds}
\label{subsec:mirrortoric} 
Let $\cX$ be a toric orbifold as in the previous section. 
Define $\cM := \Spec\C[\LL] = \Hom(\LL,\C^\times)$. 
For $d\in \LL$, let $q^d$ denote the corresponding 
element in $\C[\LL]$. This is a 
function $q^d \colon \cM \to \C^\times$. 
The space $\cM$ has a partial (possibly singular) 
compactification $\ov\cM := \Spec\C[\LL \cap \hNE_\cX]$.  
It has a special point (large radius limit point) 
$\bzero$ defined by $q^d=0$ for all nonzero $d\in \LL\cap \hNE_\cX$. 
We choose a $\Z$-basis $p_1,\dots,p_{r+s}$ of $\LL^*$ 
(here $r:= m-n$) such that each $p_a$ is \emph{extended nef}   
i.e.\ $p_a$ is semi-positive on $\hNE_\cX$ 
and $p_{r+1},\dots,p_{r+s} \in \sum_{j=m+1}^{m+s} \Q_{\ge 0} D_j$. 
Then we have the corresponding co-ordinates 
$q_1,\dots,q_{r+s}$ on $\cM$ such that 
$q^d = q_1^{\pair{p_1}{d}} \cdots q_{r+s}^{\pair{p_{r+s}}{d}}$. 
These co-ordinates $(q_1,\dots,q_{r+s})$ give a 
desingularization $\C^{r+s} \to \ov\cM$ 
such that $\bzero$ corresponds to the origin of $\C^{r+s}$. 
For $d\in \LL_\Q$, $q^d$ defines a possibly multi-valued 
function on $\cM$. 
Let $\ov{p}_a\in H^2(\cX;\Z)$ denote the image of 
$p_a \in \LL^*$. 
We write $\ov{p} \log q := \sum_{a=1}^{r} \ov{p}_a \log q_a$. 
This is an $H^2(\cX;\C)$-valued (multi-valued) function on $\cM$. 
% Define a multi-valued map $\ov{p}\log q \colon 
% \cM \to H^2(\cX;\C)$ by 
% $\ov{p} \log q = \sum_{a=1}^{r} \ov{p}_a \log q_a$. 
% This does not depend on the choice of a basis $\{p_a\}$. 
% Over the real locus $\cM_\R = \Hom(\LL,\R_{>0})$, 
% we shall take the branches of $\log q_a$, $q^d$ so that 
% $\log q_a \in \R$ and $q^d \in \R_{>0}$. 

\begin{definition}[\cite{CCIT:toric}; 
See also {\cite[Section 4.1]{Iritani:int}}] 
Take $v\in \Boxop$. 
Define an $H_{\rm orb}^*(\cX)$-valued (multi-valued) 
function $I^v(q,z)$ on an open subset of 
$\cM \times \C^\times$ by  
\[
I^v(q,z) = e^{\ov{p} \log q/z} 
\sum_{d\in \K_v} q^{d +v} \prod_{i=1}^{m+s} 
\frac{\prod_{k>d_i, \{k\} = \{d_i\}} (\ov{D}_i + k z)} 
{\prod_{k>0, \{k\} = \{d_i\} } (\ov{D}_i + k z) } 
\unit_{\{-d\}}. 
\]
Here all but finite terms in the infinite product cancel 
and $\ov{D}_j=0$ for $m+1 \le j\le m+s$. 
The terms with $d+v\notin \hNE_\cX$ automatically 
vanish and 
$I^v(q,z)$ is convergent in a neighborhood of $\bzero$.  
Apart from the prefactor $e^{\ov{p}\log q/z}$, 
it is homogeneous of degree $2\age(v)$ 
with respect to the grading of $H_{\rm orb}^*(\cX)$, 
$\deg(q^d) := 2 \pair{\hrho}{d}$ and $\deg z :=2$.  
The series $I(q,z) := I^0(q,z)$ 
is called the \emph{$I$-function}. 
We have the asymptotics  
\begin{align*} 
I^v(q,z) &= e^{\ov{p}\log q/z} (\unit_v +  O(q)) \\  
I(q,z) & = \unit + \frac{\tau(q)}{z}  + O(z^{-2}) 
\end{align*} 
where $O(q)$ denotes a function vanishing at $\bzero$ 
and $\tau(q)$ is a multi-valued map  
with values in $H^{\le 2}_{\rm orb}(\cX)$, 
called the \emph{mirror map}. 
The map $\tau(q)$ induces a single-valued 
map
\begin{equation} 
\label{eq:X_mirrormap} 
\tau(q)\colon \{(q_1,\dots,q_r)\;|\; 0<|q_a|<\epsilon \} 
\to H^{\le 2}_{\rm orb}(\cX;\C)/H^2(\cX;\Z) 
\end{equation} 
for some $\epsilon>0$.  
Here $H^2(\cX;\Z)$ acts on 
$H^{\le 2}_{\rm orb}(\cX)$ by the Galois action 
$\xi\mapsto G(\xi)$. 
\end{definition} 

The following will be shown in joint work 
with Coates, Corti and Tseng \cite{CCIT:toric} 
(see \cite{CCLT:wp} for the case of weighted projective spaces):  
\begin{theorem}[\cite{CCIT:toric}] 
\label{thm:toricmirrorthm}
Let $\cX$ be a toric orbifold in Section 
\ref{subsec:toricorb} and 
$J(\tau,z)$ be the untwisted $J$-function 
of $\cX$ with $Q=1$.  
Then we have $I(q,z) = J(\tau(q),z)$.
\end{theorem} 

The function $I^v(q,z)$ can be obtained 
from $I(q,z) = I^0(q,z)$ by differentiation. 
Writing $D_i = \sum_{a=1}^{r+s} \sfm_{ia} p_a$, 
we define the ($z$-decorated) logarithmic vector field 
$\bD_i$ on $\cM$ by $\bD_i := z \sum_{a=1}^{r+s} 
\sfm_{ia} q_a (\partial/\partial q_a)$. 
Taking $\delta \in \K_0$ such that $v=\{-\delta\}$ 
and $\ceil{\delta_i}\ge 0$ for all $i$, 
we can easily see that (see also \cite[Lemma 4.7]{Iritani:int}) 
\begin{equation} 
\label{eq:Iv_der_I} 
I^v(q,z) = q^{-\delta } 
\left( \prod_{i=1}^{m+s} 
\prod_{\nu=0}^{\ceil{\delta_i}-1} 
(\bD_i - \nu z ) \right)  I(q,z). 
\end{equation} 
In the terminology of Givental's Lagrangian cone 
\cite{Coates-Givental, CCIT:tw}, 
$I^v(q,-z)$ is in the tangent space to the 
cone at $-z I(q,-z)$. 
Therefore, $I^v$ appears as a 
column vector of the inverse fundamental 
solution. 
\begin{corollary}[{\cite[Eqn (65)]{Iritani:int}}] 
\label{cor:Iv_forX} 
Let $L(\tau,z)$ denote the 
fundamental solution \eqref{eq:fundsol} 
of the untwisted $(\bs=0)$ theory of $\cX$. 
There exists an $H^*_{\rm orb}(\cX)$-valued 
function $\theta_v(q,z)\in H_{\rm orb}^*(\cX) 
\otimes \cO_{\tcM}[z]$ defined 
on a finite cover $\tcM$ of $\cM$ and in 
a neighborhood of $\bzero$ 
such that 
\[
I^v (q,z) = L(\tau(q),z)^{-1} \theta_v(q,z), 
\quad 
\theta_v(q,z) = \unit_v + O(q).    
\] 
Also $\theta_v(q,z)$ is homogeneous of degree $2 \age(v)$ 
and $\theta_0(q,z) =\unit$. 
\end{corollary} 
\begin{proof} 
We differentiate $I(q,z) = J(\tau(q),z) = 
L(\tau(q),z)^{-1} \unit$ by the differential 
operator appearing in \eqref{eq:Iv_der_I}.  
Here notice that $z \partial_a \circ 
L(\tau(q),z)^{-1} = L(\tau(q),z)^{-1} 
\circ (z \partial_a + (\partial_a \tau) \circ_{\tau(q)})$ 
for $\partial_a = q_a (\partial/\partial q_a)$.  
\end{proof} 

\subsection{Mirror Theorem II: Toric Complete Intersection} 
\label{subsec:mirrorthm_comp} 
As before, let $\cV$ be the sum of line bundles 
$\cL_1 \oplus \cL_2 \oplus\cdots \oplus \cL_c$ over a toric 
orbifold $\cX$ and $\cY\subset \cX$ 
be a quasi-smooth complete intersection 
with respect to a regular section of $\cV$. 
Let $\iota \colon \cY \to \cX$ be the inclusion. 
% Denote by $\iota^*\colon H^*_{\rm orb}(\cX) \to H^*_{\rm orb}(\cY)$ 
% and $\iota_* \colon H_{\rm orb}(\cY) \to H_{\rm orb}^*(\cX)$ 
% the pull-back and the push-forward defined by 
% $\cI\cY \subset \cI\cX$. 
Let $\xi_i$ be the class of $\cL_i$ in 
$\Pic(\cX) \cong H^2(\cX;\Z)$. 
% We define the \emph{superage} of $v\in \Boxop$ 
% by $\sage(v) := \age(v) - \sum_{i=1}^c f_v(\xi_i)$. 
% If $\cX_v$ intersects with $\cY$, 
% $\sage(v)$ is the age of the sector $\cY \cap \cX_v$. 
% Note that 
% \[
% \text{$\cX_v$ intersects with $\cY$ } 
% \Longleftrightarrow  \ 
% \sharp\{i\,|\, f_v(\xi_i) \in \Z\} \le \dim \cX_v 
% \]
% and in this case $\sage(v)\ge 0$. 
We assume that 
\begin{itemize} 
\item The classes $\xi_1,\dots,\xi_c$ 
and $c_1(\cY) = c_1(\cX) -\sum_{i=1}^c \xi_i$ are nef. 

\item The line bundles $\cL_1,\dots,\cL_n$ are pulled back 
from the coarse moduli space $X$, i.e.\ 
$\xi_i \in H^2(X,\Z)$. 

% \item For each $v\in \Boxop$, one of the following holds: 
% (1) $\sage(v) > 1$ \emph{or} (2) $\sharp\{i\,|\, f_v(\xi_i) \in \Z\} 
% > \dim \cX_v$ \emph{or} (3) $f_v(\xi_i) = 0$ for all $1\le i\le c$. 
%
% \item For $m+1 \le j\le m+s$ and $1\le i\le c$, 
% $f_{b_j}(\xi_i) =0$. 
\end{itemize} 
% Here $f_v(\xi_i)$ denotes the age of the line bundle 
% $\cL_i$ along the sector $\cX_v$. 
Let $\varphi_i\colon \bN_\R \to \R$ be the 
piecewise linear function corresponding 
to $\xi_i$ (see Section \ref{subsec:toricorb}). 
By the second assumption, we have 
$\{\varphi_i(v)\}=f_v(\xi_i) = 0$ for all $1\le i\le c$
and $v\in \Boxop$.  
Define a lift $\txi_i\in \LL^*$ of $\xi_i$ 
by $\txi_i := \sum_{j=1}^{m+s} \varphi_i(b_j) D_j$. 
The lift $\txi_i$ does not depend on the choice of 
$\varphi_i$. 
Then $\txi_i$ is extended nef (semi-positive on 
$\hNE_\cX$) since $\pair{\txi_i}{\delta_j} = 0$. 
Set $\hrho_\cY := \hrho - \sum_{i=1}^c \txi_i$. 
This is also extended nef.  

% \begin{lemma}
% Let $\iota^* \colon H^*_{\rm orb}(\cX) \to H^*_{\rm orb}(\cY)$ 
% denote the pull-back by the inclusion $I\cY \subset \cI\cX$. 
% If $d\in \K_{\rm eff}$ and 
% $\pair{\txi_i}{d}<0$ then $\iota^* \unit_{v(d)} =0$. 
% \end{lemma} 

\begin{definition} 
\label{def:Ifunction_Y}
Let us write $I^v(q,z)= e^{\ov{p}\log q/z} \sum_{d\in \K_v} 
q^{d+v} \Box_d \unit_{\{-d\}}$.  
For $v\in \Boxop$, 
we define an $H_{\rm orb}^*(\cX)$-valued function 
$I_\cV^v(q,z)$ by 
\[
I_\cV^v(q,z) = 
e^{\ov{p}\log q/z} 
\sum_{d\in \K_{v}}q^{d+v} 
\prod_{i=1}^c 
\prod_{k=1}^{\pair{\txi_i}{d+v}} 
(\xi_i + k z) 
\cup \Box_d \unit_{\{-d\}}.  
\]
Note that $\Box_d \unit_{\{-d\}} =0$ 
for $d+v \notin \hNE_\cX$ and 
$\pair{\txi_i}{d+v}\ge 0$ otherwise. 
Also it is easy to see that $\pair{\txi_i}{d+v}$ 
is an integer. 
Under the above assumption, $I_\cV^v(q,z)$ is convergent 
near $\bzero$. Apart from the prefactor 
$e^{\ov{p}\log q/z}$, it is homogeneous of degree 
$2\age(v)$ %=2\sage(v)$  
with respect to 
%the superage grading 
%$\deg_\cY \unit_v := 2 \sage(v)$, 
the grading of $H_{\rm orb}^*(\cX)$, 
$\deg_\cY q^d := 2\pair{\hrho_\cY}{d}$ 
and $\deg z:=2$. 
We set $I_\cV(q,z) := I^0_{\cV}(q,z)$. 
We have the asymptotics: 
\begin{align}
\label{eq:Y_asymptotic} 
\begin{split} 
I_\cV^v(q,z) &= e^{\ov{p}\log q/z}( \unit_v + O(q))  \\  
I_\cV(q,z) & = F(q) \unit + \frac{G(q)}{z}  + O(z^{-2})   
\end{split} 
\end{align} 
where $F(q)$ is a power series of the form 
$1 + \sum_{d\neq 0} c_d q^d$, 
$c_d \in \Q$ with integral exponents 
$d\in \LL \cap \hNE_\cX$ 
and $G(q)$ is an $H^{\le 2}_{\rm orb}(\cX)$-valued map.  
The \emph{mirror map} 
\begin{equation} 
\label{eq:Y_mirrormap}
\tvarsigma(q) := \frac{G(q)}{F(q)} 
\end{equation} 
defines a single valued map from a neighborhood of 
$\bzero$ to $H^{\le 2}_{\rm orb}(\cX;\C)/H^2(\cX;\Z)$. 
%We also write $\varsigma(q) = \iota^* \tvarsigma(q)$. 
\end{definition} 

\begin{theorem} 
\label{thm:mirrorthm_tw}
Let $L^{\be}(\tau,z)$, $J^{\be}(\tau,z)$ 
be the fundamental solution 
and the $J$-function of the $(\be,\cV)$-twisted 
theory of $\cX$. 
For $v\in \Boxop$, there exists an 
$H^*_{\rm orb}(\cX)$-valued function 
$\tUpsilon_v(q,z) \in H^*_{\rm orb}(\cX)
\otimes \cO_{\tcM}[z]$ defined on 
a finite cover $\tcM$ of $\cM$ and in a neighborhood of 
$\bzero$ such that  
\begin{equation} 
\label{eq:Upsilon}
I_\cV^v(q,z) 
= L^\be(\tvarsigma(q),z)^{-1} \tUpsilon_v(q,z), 
\quad 
\tUpsilon_v(q,z) = \unit_v + O(q).   
\end{equation} 
Also $\tUpsilon_v(q,z)$ is homogeneous of degree 
$2 \age(v)$ for the grading $\deg_\cY(q^d) = 
2 \pair{\hrho_\cY}{d}$. 
We find that $\tUpsilon_0 = F(q) \unit$ 
by comparing the asymptotics in $z$.  
Therefore, 
\[
I_\cV(q,z) = F(q) 
J^\be(\tvarsigma(q),z).  
\]
\end{theorem} 
\begin{proof} 
When the mirror map $\tau(q)$ for $\cX$ is ``linear", 
the last statement follows from the quantum Lefschetz 
theorem \cite[Corollary 5.1]{CCIT:tw} applied to 
the previous theorem \ref{thm:toricmirrorthm}. 
First we see how to modify the proof 
of quantum Lefschetz in \cite{CCIT:tw} 
to calculate a convenient slice ($I$-function) 
of the twisted Lagrangian cone. 
Let $\marsfs{L}_{\bs}$ denote the $(\bc,\cV)$-twisted 
Lagrangian cone \cite[Section 3]{CCIT:tw} 
of $\cX$. 
% For $d\in \LL_\Q$, we write 
% $d = \ov{d} + \sum_{j=m+1}^{m+s} k_j \delta_j$, 
% $\ov{d}\in H_2(\cX;\Q)$ under the 
% decomposition \eqref{eq:LQdecomp}.  
% and   
% introduce co-ordinates $u_1,\dots,u_r$, 
% $v_1,\dots,v_s$ such that 
% $q^d = \exp(u_1 \pair{\ov{p}_1}{\ov{d}}  + \cdots + 
% u_r \pair{\ov{p}_r}{\ov{d}}) 
% v_1^{k_1} \cdots v_s^{k_s}$. 
% Then we have $\ov{p}\log q = \sum_{a=1}^r u_a \ov{p}_a$. 
Define 
\[
\bI_{\bs}(q,z) = e^{\ov{p} \log q/z} 
\sum_{d\in \K_0}  q^d Q^{\ov{d}} 
\prod_{i=1}^c \prod_{k=1}^{\pair{\txi_i}{d}} 
\exp\left( \bs (\xi_i + k z) 
\right) 
\Box_d \unit_{\{-d\}}.  
\] 
Here $\ov{d}\in H_2(\cX;\Q)$ is the $H_2(\cX;\Q)$-component 
of $d\in \LL_\Q$ under the decomposition \eqref{eq:LQdecomp} 
and $\bs(x) = \sum_{k\ge 0} s_k x^k/k!$. 
We claim 
that $-z \bI_\bs(q,-z)$ is on the cone 
$\marsfs{L}_{\bs}$. 
Here we regard $\bI_\bs$ as a 
$\Lambda_\bs[\![\log q_1,\dots,\log q_r, 
q_{r+1}^{1/e},\dots,q_{r+s}^{1/e}]\!]$-valued point on 
Givental's loop space $\cH$ for $e\in \N$ 
such that $e \K_0 \subset \LL$.  
(See the definition of $\cH$ and $\marsfs{L}_\bs$ as  
formal schemes in \cite[Appendix B]{CCIT:tw}.)  
At $\bs=0$, $-z\bI_0(q,-z)$ is on the untwisted cone 
$\marsfs{L}_0$ by 
Theorem \ref{thm:toricmirrorthm}. 
Write $\txi_i = \sum_{a=1}^{r+s} v_{ia} p_a$ and 
define the logarithmic vector field $\tbxi_i := 
z \sum_{a=1}^{r+s} v_{ia} q_a (\partial/\partial q_a)$. 
Then the same argument as the last paragraph of 
the proof of Theorem 4.8 in \cite{CCIT:tw} shows 
that 
\[
\bbf_\bs(q) = 
\exp\left(
-\sum_{i=1}^c G_0(-\tbxi_i,z)\right) (-z\bI_{0}(q,-z)) 
\]
is on the untwisted cone $\marsfs{L}_0$, where 
$G_y(x,z)$ is a formal power series depending on 
$\bs$ defined in \cite{CCIT:tw}. 
Applying Tseng's symplectic operator $\Delta^{\rm tw}$ 
(\cite{Tseng:QRR}; we use the convention  
in \cite[Theorem 4.1]{CCIT:tw}),  
we get an element 
$\Delta^{\rm tw} (\bbf_\bs(q))$ on 
$\marsfs{L}_\bs$. 
Using the property of the function 
$G_y(x,z)$ and $\tbxi_i (e^{\ov{p}\log q/z} q^d)  
= (\xi_i + z \pair{\txi_i}{d} ) e^{\ov{p}\log q/z} q^d$ 
(see Eqns (12), (13) in \cite{CCIT:tw} and 
the discussion following them), 
we find that this equals  
$-z \bI_\bs(q,-z)$. 
This proves the claim. 
% The claim implies that there exists 
% an element $\bg_\bs(q,z) \in H^*_{\rm orb}(\cX) 
% \otimes R\{z\}$ and a mirror map 
% $\bh_\bs(q) \in H^*_{\rm orb}(\cX) \otimes R$ 
% such that 
% $\bI_\bs(q,z) = 
% \bL^\bs(\bh_\bs(q),z)^{-1}\bg_\bs(q,z)$.   
Taking $\bc = \be_\lambda$, we obtain 
a vector $\bI_\lambda$ 
on the $(\be_\lambda,\cV)$-twisted Lagrangian cone: 
\[
\bI_\lambda(q,z) := \bI_\bs(q,z)|_{\bc = \be_\lambda} = 
e^{\ov{p}\log q/z} 
\sum_{d\in \K_0} 
q^{d} Q^{\ov{d}} 
\prod_{i=1}^c \prod_{k=1}^{\pair{\txi_i}{d}}
(\lambda + \xi_i+kz) \Box_{d} \unit_{\{-d\}}. 
\]
By the discussion as in \cite[Section 5.2]{CCIT:tw}, 
we know that $\bI_\lambda$ and $\bJ^{\be_\lambda}$ 
are related as 
$\bI_\lambda(q,z) = \bF(q)
\bJ^{\be_\lambda}(\tbvarsigma(q;\lambda),z)$ where 
$\bF(q)$, $\bvarsigma(q;\lambda)$ 
are determined by the $z$-asymptotics of $\bI_\lambda$ 
in the same way as \eqref{eq:Y_asymptotic} 
and \eqref{eq:Y_mirrormap}. 
(Here $F(q) = \bF(q)|_{Q=1}$, 
$\tvarsigma(q) = \tbvarsigma(q;0)|_{Q=1}$.) 
Now we differentiate $\bI_\lambda$ by the differential 
operator appearing in \eqref{eq:Iv_der_I}. We find  
\begin{equation} 
\label{eq:der_Ilambda}
q^{-\delta } 
\left( \prod_{i=1}^{m+s} 
\prod_{\nu=0}^{\ceil{\delta_i}-1} 
(\bD_i - \nu z ) \right) 
\bI_\lambda= 
e^{\ov{p}\log q/z} 
\sum_{d\in \K_v} q^{d+ v} Q^{\ov{d+v}} 
\prod_{i=1}^c \prod_{k=1}^{\pair{\txi_i}{d+\ceil{\delta}}}
(\lambda + \xi_i+kz) 
\Box_{d} \unit_{\{-d\}} 
\end{equation} 
where $d+\ceil{\delta} = (d_i + \ceil{\delta_i})_{i=1}^{m+s}$ 
is an element of $\K_0$. 
%This is close to $I^v_\cV$, but is still different. 
Applying the infinite-rank differential operator 
$\prod_{i=1}^c \prod_{k=1}^{\pair{\txi_i}{\delta}} 
(\lambda + \tbxi_i + kz)^{-1}$ 
to the above element 
(here note that $\pair{\txi_i}{\delta} 
\in \Z_{\ge 0}$ since $\delta\in \hNE_\cX \cap \K_0$), 
we obtain 
\begin{equation*} 
%\label{eq:I^v_lambda}
\bI^v_\lambda(q,z) := e^{\ov{p}\log q/z} 
\sum_{d\in \K_v} 
q^{d+v} Q^{\ov{d+v}} 
\prod_{i=1}^c \prod_{k=1}^{\pair{\txi_i}{d+v}}
(\lambda+ \xi_i+kz) 
\Box_{d} \unit_{\{-d\}}.  
\end{equation*} 
Here we expand $(\lambda + \tbxi_i + kz)^{-1}$ as 
$\sum_{k=0}^\infty \lambda^{-k-1} (-\tbxi_i -kz)^k$. 
Because 
\[
\bI_\lambda = \bF(q) \bJ^{\be_\lambda}
(\tbvarsigma(q;\lambda),z) 
= \bL^{\be_\lambda}(\tbvarsigma(q;\lambda),z)^{-1} 
\bF(q) \unit 
\]
and $\bI^v_\lambda$ is obtained from 
$\bI_\lambda$ by differentiation, 
$\bI^v_\lambda = \bL^{\be_\lambda}
(\tbvarsigma(q;\lambda),z)^{-1} \tbUpsilon_v$ 
for an $H^*_{\rm orb}(\cX)$-valued function 
$\tbUpsilon_v(q,z;\lambda)$ 
which is regular at $z=0$ 
(see the proof of Corollary \ref{cor:Iv_forX}).  
Here $\tbUpsilon_v$ is defined over the ring 
$\C[z](\!(\lambda^{-1})\!)[\![\Eff_\cX]\!][\![\log q_1,
\dots, \log q_r, q_{r+1}^{1/e},\dots,q_{r+s}^{1/e}]\!]$. 
But $\bI^v_\lambda$, $\bL^{\be_\lambda}$, 
$\tbvarsigma(q;\lambda)$ do not contain 
negative powers of $\lambda$, so it follows 
that $\tbUpsilon_v$ is also regular at $\lambda=0$. 
Now the conclusion follows by setting $\lambda =0$, 
$Q=1$. 
\end{proof} 
% \begin{remark} 
% As discussed in \cite{CCLT:wp}, by Reid-Tai criterion, 
% the condition $\sage(v) >1$ implies that 
% the singularity along the sector $\cY \cap \cX_v$ 
% is well-formed and terminal. 
% But our assumption does not quite follow from 
% the condition that $\cY$ itself is terminal 
% since a sector $\cX_v$ may not intersect with $\cY$. 

% % \begin{table}[htbp]
% \begin{tabular}{l l l l l} 
% $\cX $ & $\cV$ & singularity of $\cY$ & $c_1(\cY)$ & assumption \\ 
% \hline  
% $\Proj(1,1,1,3,3)$ & $\cO(9)$ & canonical & CY & yes \\ 
% $\Proj(1,1,1,1,3)$ & $\cO(7)$ & canonical & CY & no \\ 
% $\Proj(1,1,1,1,1,3)$ & $\cO(7)$ & terminal &  Fano & yes \\ 
% $\Proj(1,1,1,1,1,3)$ & $\cO(2) \oplus \cO(6)$ & smooth 
% & CY & no
% \end{tabular} 
% %\end{table}
% \end{remark} 

% \begin{example} 
% Our assumption is not satisfied for a degree 9 
% Calabi-Yau hypersurface in $X_9 \subset 
% \Proj(1,1,1,1,1,4)$ (with a canonical singularity)
% but is satisfied for a Fano 
% hypersurface in $X_7 \subset \Proj(1,1,1,1,1,3)$ 
% (with a terminal singularity). 
% Our assumption is satisfied for 
% Calabi-Yau $X_{9} \subset \Proj(1,1,1,3,3)$.  
% A Calabi-Yau $(2,6)$-complete intersection 
% in $\Proj(1,1,1,1,1,3)$ (no singularities) 
% does not satisfy our assumption. 
%\end{example} 
\begin{remark}
\label{rem:multigen}
Recall that $I^v$ was obtained from $I$ by 
differentiation (see \eqref{eq:Iv_der_I}). 
In the twisted case, in general, 
$I^v_\cV$ cannot be written in the form 
$I^v_\cV = P_v I_\cV$ 
for some differential operator 
$P_v \in \cO_{\tcM}[z]\langle z\partial_1,\dots, z\partial_{r+s} 
\rangle$. 
This means that the twisted quantum $D$-module over 
$H^{\le 2}_{\rm orb}(\cX)$ 
may not be generated by the unit section $\unit$ 
as an $\cO[z]\langle z \partial \rangle$-module,  
where $\partial$ denotes the derivative in 
the $H^{\le 2}_{\rm orb}(\cX)$ direction. 
Differentiating $I_\cV$ by the same differential 
operator as in \eqref{eq:Iv_der_I}, we obtain  
(cf.\ \eqref{eq:der_Ilambda}) 
\begin{equation} 
\label{eq:der_IcV}
q^{-\delta } 
\left( \prod_{i=1}^{m+s} 
\prod_{\nu=0}^{\ceil{\delta_i}-1} 
(\bD_i - \nu z ) \right) 
I_\cV = 
e^{\ov{p}\log q/z} 
\left( \prod_{i=1}^c 
\prod_{k=1}^{\pair{\txi_i}{\delta}} 
(\xi_i + k z ) \unit_v + O(q) \right). 
\end{equation} 
This equals $I_\cV^v$ if $\pair{\txi_i}{\delta}=0$ for all $i$. 
The equalities $\pair{\txi_i}{\delta}=0$ $(\forall i)$ 
can be achieved (for some $\delta$) 
if there exists a cone $\sigma$ 
in $\Sigma$ such that $v\in \sigma$ and 
$v$ is in the monoid generated by 
$\{b_1,\dots,b_{m+s}\} \cap \sigma$. 
On the other hand, if we invert the variable 
$z$, i.e.\ restrict the $D$-module to the complement 
of $z=0$, we can see from \eqref{eq:der_IcV} that 
the twisted quantum $D$-module is still generated 
by $\unit$. 
Such a non-generation phenomenon first appeared 
in the work of Guest-Sakai \cite{Guest-Sakai}. 
\end{remark} 

We remark that one can calculate $L^\be$ and $\tUpsilon_v$ 
from the functions $I^v_\cV$ using the \emph{Birkhoff factorization} 
in the theory of loop groups, 
as observed by Coates-Givental \cite{Coates-Givental} 
and Guest \cite{Guest:D-mod}. 
Using the fact that $H^*(\cX_v)$ is generated by $\unit_v$ 
as a $\C[\ov{p}_1,\dots,\ov{p}_r]$-module, we can 
find differential operators $P_{v,i}(z\partial) \in 
\C[z\partial_1,\dots,z\partial_r]$, $i=1,\dots,l_v$  
(where $\partial_a = q_a (\partial/\partial q_a)$) 
such that $\phi_{v,i} = P_{v,i}(\ov{p}_1,\dots,\ov{p}_r) 
\unit_v$, $v\in \Boxop$, $i=1,\dots,l_v$ 
form a basis of $H^*_{\rm orb}(\cX)$. 
Then by the asymptotics \eqref{eq:Y_asymptotic} 
and the previous theorem, we have 
\[
P_{v,i}(z\partial) I_\cV^v(q,z) = e^{\ov{p}\log q/z} 
(\phi_{v,i} + O(q)) = 
L^\be(\tvarsigma(q),z)^{-1} \tUpsilon_{v,i}(q,z). 
\] 
Here $\tUpsilon_{v,i}(q,z)=
P_{v,i}(z \tvarsigma^*\nabla^{\be}) \tUpsilon_v(q,z)
= \phi_{v,i}+O(q)$ 
is an $H^*_{\rm orb}(\cX)$-valued 
function regular at $z=0$. 
We consider the matrix formed by the column vectors 
$P_v^i(z\partial) I_{\cV}^v$ 
and regard it as an element of the loop group 
with loop parameter $z$.  
Then the above equation shows that 
$(L^\be)^{-1}$ and $(\tUpsilon_{v,i})_{v,i}$ are 
obtained from it by the Birkhoff factorization \cite{Pressley-Segal}.  
Here we use the fact that $L^\be = \id + O(z^{-1})$ and 
$\tUpsilon_{v,i}(q,z)$ is regular at $z=0$. 
This also gives a proof that 
$\tUpsilon_{v}(q,z)$ and $L^\be(\varsigma(q),z)$ 
are analytic near $q=\bzero$. 

Using Proposition \ref{prop:Eulertw_Y}, we get 
the following corollary of Theorem \ref{thm:mirrorthm_tw}. 
\begin{corollary} 
\label{cor:mirrorthm_Y}
Let $L_\cY(\tau,z)$, $J_\cY(\tau,z)$ denote 
the fundamental solution and the $J$-function 
of the untwisted theory of $\cY$. 
Set $\Upsilon_v(q,z) = \iota^* \tUpsilon_v(q,z)$ and 
$\varsigma(q) = \iota^* \tvarsigma(q)$. 
Then we have 
$\iota^*I_\cV^v(q,z) = 
L_\cY(\varsigma(q),z)^{-1} \Upsilon_v(q,z)$ and 
$\iota^* I_\cV(q,z)= F(q) J_\cY(\varsigma(q),z)$. 
\end{corollary} 

\section{Equality of Periods: A-periods $=$ B-periods} 
In this section we show that the A-periods of $\cX$ 
equal ordinary periods 
(or oscillatory integral) of the mirror. 
The key point --- the hypersurface $J$-function is a Laplace
transform of the ambient one and 
the same for mirror oscillatory integrals --- 
had been observed in Givental's paper  
\cite{Givental:mirrorthm-toric} on toric mirror theorem 
and in the Coates-Givental proof \cite{Coates-Givental} 
of quantum Lefschetz.

\subsection{Laplace Transform of A-Periods} 
Let $\cX$ be a toric orbifold 
in Section \ref{subsec:toricorb} and 
$\cY$ be a complete intersection in $\cX$ 
with respect to $\cV = \cL_1\oplus \cdots \oplus\cL_c$ 
in Section \ref{subsec:mirrorthm_comp}. 
Here we show that the Laplace transforms 
of the A-periods of $\cX$ 
give precisely those of $\cY$. 
% We start with a formula of the Laplace transform  
% of $I_\cX(q,z)$ or the quantum cohomology 
% central charge \eqref{eq:centralcharge} of $\cX$. 
%We will restrict the domain of definition 
%of $I_\cX$ to the real locus $\cM_\R = \Hom(\LL,\R_{>0})$. 
% We choose branch cuts of $q^d$ (with $d\in \LL_\Q$)  
% and $\log q_a$ on $\cM_\R$ 
% so that $q^d, \log q_a \in \R_{>0}$. 
% Under these branch cuts, $I(q,z)$, $\tau(q)$ are 
% single-valued on $\cM_\R$. 
We choose a lift $\txi_j\in \LL^*$ of $\xi_j$ 
for $1\le j\le c$ as in Section \ref{subsec:mirrorthm_comp}. 
Then $\txi_j$ defines a one-parameter subgroup 
$\C^\times \ni r \mapsto r^{\txi_i}$ 
of $\cM = \LL^*\otimes \C^\times$. 
In co-ordinates, 
$r^{\txi_j}= (r^{v_{j1}},\dots, r^{v_{j,r+s}})$ 
when we set $\txi_j = \sum_{a=1}^{r+s} v_{ja} p_a$. 
By the formula \eqref{eq:A-period2} 
and Corollary \ref{cor:Iv_forX}, 
the A-period $\Pi(\theta_v,\cE)$ of $\cX$ 
is given by  
\begin{equation} 
\label{eq:Aperiod_Iv} 
\Pi(\theta_v,\cE)(\tau(q),z)  =  
\left( I^v(q,-z), \; z^{n-\frac{\deg}{2}} z^{\rho}
\Psi(\cE) \right)_{\rm orb}.  
\end{equation} 
Here $\theta_v$ is the section\footnote
{More precisely, $\theta_v$ is a section 
of $\tau^*QDM(\cX)$. The A-period $\Pi(\theta_v,\cE)$ 
should be understood as the pairing 
of $\theta_v(q,-z)$ and $(\tau^* \frs(\cE))(q,z)$ 
in the pulled-back quantum $D$-module.} 
of the quantum $D$-module $QDM(\cX)$ 
of $\cX$ in Corollary \ref{cor:Iv_forX}. 
Define the (partial) Laplace transform 
$\hPi(\phi,\cE)$ by
\[
\hPi(\phi,\cE)(q,s,z) = 
(\prod_{j=1}^c s_j)
\int_0^\infty \dots \int_0^\infty 
\Pi(\phi,\cE)\left(
\tau (
{\textstyle\prod_{j=1}^c} (zr_j)^{\txi_j} \cdot q
), z\right) 
e^{-\sum_{j=1}^c r_j s_j} 
dr_1 \cdots dr_c.  
\]
where $s=(s_1,\dots, s_c)\in (\R_{>0})^c$. 
Note that $\Pi(\phi,\cE)(\tau(q),z)$, 
$\hPi(\phi,\cE)(q,s,z)$ are 
multi-valued. 
We can regard them as a single-valued function 
in $(\log q_1,\dots,\log q_{r+s},\log z)$. 
 
\begin{proposition}
\label{prop:Laplacetrans} 
For $\cE\in K(\cX)$ and $v\in \Boxop$, we have 
\begin{align*}
%\label{eq:Laplacetrans} 
\hPi(\theta_v,\cE)(q,s,z) = 
\left (
I_\cV^v\left(q'_s,
%\textstyle\prod_{j=1}^c (e^{\pi\iu}/s_j)^{\txi_j}\cdot q,
-z\right ), \; 
z^{n-\frac{\deg}{2}} 
z^{\rho_\cY}
\Psi_\cV(\cE)
\right )_{\rm orb} 
\end{align*}
where we set $\Psi_\cV(\cE) := e^{\pi\iu c_1(\cV)} \hGamma(\cV^\vee) 
\cup \Psi(\cE)$, $\rho_\cY :=c_1(\cY) = \rho -\sum_{j=1}^c \xi_j$ and 
\[ q'_s = \prod_{j=1}^c 
(e^{\pi\iu}s_j^{-1})^{\txi_j} \cdot q. 
\quad \text{i.e.\ } 
\log q_{s,a}' = \log q_a 
+ \sum_{j=1}^c (\pi \iu - \log s_j)v_{ja}, \ \log s_j>0. 
\] 
% where we set 
% \[
% \Psi_\cV(\cE) := \left(
% \bigoplus_{v\in \Boxop} 
% \prod_{j=1}^c  
% e^{\pi\iu \xi_j} 
% \Gamma(1 - \xi_j) \unit_{v}
% \right) \cup \Psi(\cE). 
% \]
% and $\deg_\cY$ is the superage grading 
% in Section \ref{subsec:mirrorthm_comp}.  
Note that the right-hand side also gives the analytic 
continuation of $\hPi(\theta_v,\cE)$ in $s$.
\end{proposition} 
\begin{proof} 
First we calculate the Laplace transform of $I^v(q,-z)$. 
Writing $I^v(q,-z) = e^{-\ov{p}\log q/z} \sum_{d\in\K_v} 
q^{d+v} \ov{\Box}_d \unit_{\{-d\}}$, we have 
\begin{align*} 
& \int_0^\infty \cdots \int_0^\infty   
I^v\left(\textstyle \prod (r_j z)^{\txi_j} \cdot q, -z\right) 
e^{-\sum_{j=1}^c r_j s_j} dr_1\cdots dr_c  \\ 
& 
= \sum_{d\in \K_v, d+v\in \hNE_\cX} 
e^{-\ov{p} \log q/z} q^{d+v}  \ov{\Box}_d
\prod_{j=1}^c  \int_0^\infty 
e^{-\xi_j \log(r_j z)/z} (r_j z)^{\pair{\txi_j}{d+v}} 
e^{-r_j s_j} d r_j  \unit_{\{-d\}} \\ 
& 
= \sum_{d \in \K_v, d+v\in \hNE_\cX} 
e^{-\ov{p}\log q/z} q^{d+v} \ov{\Box}_d 
\prod_{j=1}^c \left(z/s_j\right)
^{\pair{\txi_j}{d+v}-\frac{\xi_j}{z}}s_j^{-1} 
\Gamma\left (1+ 
\pair{\txi_j}{d+v} - \tfrac{\xi_j}{z} \right)  
\unit_{\{-d\}}.  
\end{align*} 
Using $\Gamma(1+x) = x \Gamma(x)$ and 
$\pair{\txi_j}{d+v} \in \Z_{\ge 0}$, 
we find that this is $(s_1\cdots s_c)^{-1}$ times 
\begin{align*} 
\prod_{j=1}^c 
e^{(\pi\iu -\log z) \xi_j/z} 
\Gamma\left(1- \xi_j/z \right) 
\cdot  
I_\cV^v\left(
\textstyle \prod_{j=1}^c (e^{\pi\iu}/s_j)^{\txi_j} 
\cdot q, -z\right).  
\end{align*} 
The conclusion now easily follows 
from this and \eqref{eq:Aperiod_Iv}. 
\end{proof} 
 
\begin{remark} 
The integral in $r_j$ yielding the factor 
$\Gamma_j = \Gamma(1+ \pair{\txi_j}{d+v} -\xi_j/z)$ in 
the above calculation 
should be understood as a vector-valued integration. 
The exchange of sum and integral  
is justified by the estimates  
$\|\Box_d\| \le C_1 C_2^{|d|}/\pair{\hrho}{d}!$ 
and $\|\Box_d \prod_{j=1}^c \Gamma_j\|
\le C_1 C_2^{|d|}/\pair{\hrho_\cY}{d}!$ 
for some $C_1,C_2>0$. 
Here $|d| = \sum_{a=1}^{r+s}\pair{p_a}{d}$ and $\|\cdot \|$ 
is the operator norm with respect to some norm 
on $H_{\rm orb}^*(\cX)$. 
Note that we need the fact that $\hrho, \hrho_\cY$ 
are extended nef. 
\end{remark} 

\begin{remark} 
We can view the map $\Psi_\cV(\cE)$ as defining 
a dual\footnote{Because the Euler twisted theory has 
a degenerate pairing, we have to distinguish 
the $\hGamma$-integral structure from its dual: 
the integral structure itself should be given by 
$\Psi^\cV(\cE) = \hGamma(\cV)^{-1} \Psi(\cE)$.} 
integral structure of the Euler twisted theory. 
\end{remark} 

Using the mirror theorem (Corollary \ref{cor:mirrorthm_Y}) 
we obtain the following corollary. 
\begin{corollary} 
\label{cor:Laplace} 
For an algebraic vector bundle $\cE$ on $\cY$, 
we have  
\[
(2\pi\iu)^c 
\Pi_\cY(\Upsilon_v,\cE)\left( 
\varsigma \left(\textstyle
q'_s \right),z
\right) 
= 
\hPi(\theta_v,\iota_* \cE)(q,s,z), 
\]
where $\Pi_\cY(\Upsilon_v, \cE)$ denotes the A-period 
\eqref{eq:A-period} for $\cY$ and 
$\Upsilon_v$ is the section of the quantum $D$-module 
of $\cY$ appearing in Corollary \ref{cor:mirrorthm_Y} 
and \eqref{eq:Upsilon}.  
\end{corollary} 
\begin{proof} 
By Toen's Grothendieck-Riemann-Roch \cite{Toen} 
we have 
\begin{equation} 
\label{eq:Toen_pushforward} 
\tch(\iota_* \cE) =\iota_* \left(
\prod_{j=1}^c \frac{1-e^{-\xi_j}}{\xi_j} 
\cdot\tch(\cE) 
\right). 
\end{equation} 
Using this, $\iota^* \hGamma_\cX = 
\hGamma_\cY \cup \iota^* \hGamma(\cV)$ 
and $\Gamma(1-x) \Gamma(1+x) = \pi x/\sin (\pi x)$, we find  
\[
\Psi_\cV(\iota_* \cE) 
%(2\pi\iu)^c 
%\iota_*\left(\hGamma_\cY \cup 
%(2\pi\iu)^{\frac{\deg_0}{2}} \inv^* \tch(\iota^* V)) \right) 
= (2\pi\iu)^c \iota_* \Psi_{\cY}(\cE). 
\] 
Here $\Psi_\cY$ denotes the map $\Psi$ \eqref{eq:Psi} 
for $\cY$. 
The conclusion follows from the previous 
proposition and Corollary \ref{cor:mirrorthm_Y}. 
\end{proof}

\subsection{Mirror Construction} 
\label{subsec:mirrorconst} 
In toric geometry, various mirror constructions have been 
found by Batyrev \cite{Batyrev:dual}, 
Batyrev-Borisov \cite{Batyrev-Borisov:comp}, 
Givental \cite{Givental:ICM} and  
Hori-Vafa \cite{Hori-Vafa:mirrorsym}. 
Following Givental and Hori-Vafa, 
we shall construct mirrors for 
nef complete intersections in toric orbifolds
as Landau-Ginzburg models. 

Let $\cX$ be a toric orbifold in Section \ref{subsec:toricorb}. 
A \emph{nef partition} is 
a partition $\{1,\dots,m\} = I_0 \sqcup I_1 \sqcup 
\cdots \sqcup I_c$ such that 
$\xi_j := \sum_{i\in I_j} \ov{D}_i$ is nef for all 
$0\le j\le c$ and that $\xi_1,\dots,\xi_c$ are pulled 
back from the coarse moduli space $X$, i.e.\  
they are Cartier divisors on $X$. 
This is a special case of the situation 
in Section \ref{subsec:mirrorthm_comp}. 
In the case of the original nef partition due to 
Batyrev-Borisov \cite{Batyrev-Borisov:comp}, 
$I_0$ is assumed to be empty. 
We need not to assume that 
$\xi_0$ is Cartier on $X$. 
As before, we assume the existence of 
a quasi-smooth complete intersection 
$\cY \subset \cX$ with respect to 
$\cV = \cL_1\oplus \cdots \oplus \cL_c$ where $\cL_i$ 
is the line bundle in the class $\xi_i \in \Pic(\cX)$.  
Let $\varphi_j \colon \bN_\R \to \R$ be the piecewise 
linear function defined by $\varphi_j(b_i) = 1$ 
for $i\in I_j$ and $\varphi_j(b_i) = 0$ for $i \in 
\{1,\dots,m\} \setminus I_j$. 
By the assumption $\varphi_j(b_i)$ is $0$ or $1$ 
for all $1\le i\le m+s$ and $j\ge 1$.  
We set $\hI_j := \{1\le i\le m+s\,|\, \varphi_j(b_i)=1\}$ 
for $j\ge 1$. 
The sets $\hI_1,\dots,\hI_c$ are mutually disjoint. 
Set $\hI_0 := \{1,\dots,m+s\} \setminus 
\bigcup_{j=1}^c \hI_j$. 

Consider the torus $\chT := \Hom(\bN,\C^\times) 
= \bM \otimes_\Z \C^\times$. 
Let $t$ denote a point on $\chT$. 
Each element $b \in \bN$ defines a function 
$t^b \colon \chT \to \C^\times$. 
Take $\alpha=(\alpha_1,\dots,\alpha_{m+s})$ 
in $(\C^\times)^{m+s}$ 
and define the Laurent polynomials 
$W^{(0)}_\alpha(t),\dots, W^{(c)}_\alpha(t)$ on $\chT$ as   
$
W_\alpha^{(j)}(t) = \sum_{i\in \hI_j} \alpha_{i} t^{b_i}$. 
A mirror of $\cY$ is given by the complete intersection 
in $\chT$ 
\[
\chY_{\alpha} = \{t\in \chT\;|\; 
W_\alpha^{(1)}(t) = \cdots = W_\alpha^{(c)}(t) =1\} 
\]
endowed with a holomorphic function 
$W_\alpha^{(0)} \colon \chY_\alpha \to \C$. 
The pair $(\chY_\alpha, W_\alpha^{(0)})$ is called 
the \emph{Landau-Ginzburg model}. 
We assume that $\chY_\alpha$ is a non-empty 
smooth complete intersection for generic $\alpha$. 
The translation of the torus $\chT$ 
induces the $\chT$-action on the parameter space: 
$\alpha \mapsto t\cdot \alpha := 
(t^{b_1} \alpha_1,\dots,t^{b_{m+s}} \alpha_{m+s})$. 
Then 
$(\chY_\alpha, W_\alpha^{(0)}) \cong 
(\chY_{t\cdot \alpha}, W_{t\cdot \alpha}^{(0)})$. 
Therefore the parameter space of the mirror family 
descends to $\cM$ (in Section \ref{subsec:mirrortoric}) 
via the exact sequence 
(the divisor sequence tensored with $\C^\times$): 
\begin{align*} 
\begin{CD}
1 @>>> \chT = \bM \otimes_\Z \C^\times 
@>>> (\C^\times)^{m+s} @>>> \cM =\LL^*\otimes_{\Z} \C^\times 
@>>> 1.  
\end{CD} 
\end{align*} 
%We write $[\alpha]$ the image of $\alpha$ in $\cM$. 
In \cite{Iritani:int} we considered the mirror of 
a toric orbifold $\cX$ itself. In this case 
$I_0=\{1,\dots,m\}$ and the mirror is 
the family of functions 
$\sum_{j=1}^{m+s} \alpha_j t^{b_j} \colon \chT \to \C$. 

\begin{remark} 
\label{rem:compactification} 
Batyrev and Borisov  
\cite{Batyrev-Borisov:comp} dealt with the 
case where $I_0$ is empty. 
In this case $\cY$ is Calabi-Yau. 
They considered a Calabi-Yau 
compactification of $\chY_\alpha$ 
in a toric variety $\widehat{\Proj}_\nabla$. 
Here $\widehat{\Proj}_\nabla$ is a 
crepant partial resolution of the toric variety $\Proj_\nabla$ 
associated with the polytope 
$\nabla = \nabla_1 + \cdots + \nabla_c\subset \bN_\R$,   
where $\nabla_i$ is the convex hull of 
$\{b_j\;|\; j\in I_i\}\cup \{0\}$. 
It would be interesting 
to find a partial compactification 
of $(\chY_\alpha, W^{(0)}_\alpha)$ 
with good topological properties. 
\end{remark} 
\begin{remark} 
We hope that the existence of a quasi-smooth 
complete intersection $\cY$ and 
that of a smooth complete intersection 
$\chY_\alpha$ are related. 
In the Batyrev-Borisov construction \cite{Batyrev-Borisov:comp}, 
it was shown that a general complete intersection 
$\cY$ is quasi-smooth if and only if 
the compactification of a 
general $\chY_\alpha$ is quasi-smooth. 
\end{remark} 

\subsection{A-Periods and B-Periods} 
\label{subsec:periodcal} 
Take $\C^\times$-co-ordinates $(t_1,\dots,t_n)$ on 
$\chT$ associated with a basis of $\bN$. 
Define a holomorphic volume form $\Omega_\alpha$ 
on $\chY_\alpha$ by 
\[
\Omega_\alpha = 
\frac{\frac{dt_1}{t_1} \wedge 
\cdots \wedge \frac{dt_n}{t_n}}
{dW^{(1)}_\alpha \wedge \cdots \wedge dW^{(c)}_\alpha}.  
% = \Res_{\chY_\alpha} \left[ 
% \frac{\frac{dt_1}{t_1}\wedge \cdots \wedge \frac{dt_n}{t_n}}
% {W^{(1)}_\alpha-1, \dots, W^{(c)}_\alpha-1} 
% \right].  
\]
We shall consider the following oscillatory integral 
(B-periods): 
\begin{equation} 
\label{eq:oscint}
\int_{\Gamma(\alpha)} \phi(t) 
e^{-W_\alpha^{(0)}(t)/z} \Omega_\alpha   
\end{equation} 
for a Laurent polynomial $\phi\colon \chT \to \C$ 
and a possibly noncompact cycle $\Gamma(\alpha) 
\subset \chY_\alpha$ such that 
$\Re(W_\alpha^{(0)}(t)/z) \to \infty$ 
in the end of $\Gamma(\alpha)$.  
More generally, for 
$\vec{k}=(k_1,\dots,k_c)\in (\Z_{\ge 0})^c$, 
we introduce the residue symbol 
\[
\Osc(\phi,\vec{k};\alpha) = 
\left(\textstyle\prod_{j=1}^c z^{k_j} k_j!\right)
\Res_{\chY_\alpha} \left(
\frac{\phi(t) e^{-W_\alpha^{(0)}(t)/z}\, 
\frac{dt_1}{t_1} \wedge \cdots \wedge \frac{dt_n}{t_n}
} 
{\prod_{j=1}^c (W^{(j)}_\alpha(t) -1)^{k_j+1}} 
\right)
\]
and define the ``oscillatory" residue integral 
\begin{equation}
\label{eq:oscresint}
\int_{\Gamma(\alpha)} \Osc(\phi,\vec{k};\alpha) 
= \frac{\prod_{j=1}^c z^{k_j} k_j!}{(2\pi\iu)^c} 
\int_{T(\Gamma(\alpha))}
\frac{\phi(t) e^{-W_\alpha^{(0)}(t)/z} \,
\frac{dt_1}{t_1} \wedge \cdots \wedge \frac{dt_n}{t_n}} 
{\prod_{j=1}^c (W_\alpha^{(j)}(t) -1)^{k_j+1}}. 
\end{equation} 
Here $T(\Gamma(\alpha))\subset \chT$ 
is a cycle given as follows: 
Take a small tubular neighbourhood $N$ 
of $\chY_\alpha$ in $\chT$. Then  
$N\setminus \bigcup_{j=1}^c (W_\alpha^{(j)})^{-1}(1)$ 
has a deformation retraction to an  
$(S^1)^c$-bundle $T \to \chY_\alpha$.  
We take $T(\Gamma(\alpha))$ to be the total space of 
$T|_{\Gamma(\alpha)}$. 
Note that \eqref{eq:oscresint} equals \eqref{eq:oscint} 
when $\vec{k}=0$.  

In this section we consider the integral over 
the \emph{real locus}. 
Set\footnote
{Note that this does not depend on the choice of 
co-ordinates $t_i$ on $\chT$.} 
$\chT_\R := \Hom(\bN,\R_{>0}) = \{(t_1,\dots, t_n)\,|\, 
t_i >0 \, (\forall i) \}$. 
When $\alpha\in (\R_{>0})^{m+s}$, 
we can define the real cycle $\Gamma_\R(\alpha)$ 
in $\chY_\alpha$ by  
$\Gamma_\R(\alpha) := \chY_\alpha \cap \chT_\R$.   
Similarly we define $\cM_\R := \Hom(\LL, \R_{>0}) 
\subset \cM$. 
% Using the inquality $\beta_1 x_1 + \cdots + \beta_k x_k 
% \ge x_1^{\beta_1} \cdots x_k^{\beta_k}$ for 
% $\beta_i\ge 0$, $\sum_{i=1}^k \beta_i = 1$, $x_i>0$, 
For $\alpha \in (\R_{>0})^{m+s}$, 
we have the estimate\footnote
{Use the fact that $0$ is in the interior of $\Delta$  
and the inequality
$\beta_1 x_1 + \cdots + \beta_k x_k 
\ge x_1^{\beta_1} \cdots x_k^{\beta_k}$
for $x_i>0$, $\beta_i\ge 0$, 
$\sum_{i=1}^k \beta_i = 1$.}: 
\[
\sum_{j=0}^c W^{(j)}_\alpha (t)
= \sum_{i=1}^{m+s} \alpha_i t^{b_i} \ge 
\epsilon(\alpha) \max_{1\le i\le n}
\left\{ t_i^{1/N}, t_i^{-1/N} \right\}
\qquad \forall  t \in \chT_{\R}    
\]
for some $\epsilon(\alpha) >0$ and $N\in \N$. 
Restricting this to $\Gamma_\R(\alpha) = \chT_\R \cap \chY_\alpha$, 
we get $W^{(0)}_\alpha(t)+c\ge \epsilon(\alpha) \max_{1\le i\le n}
\{ t_i^{1/N}, t_i^{-1/N} \}$. 
% Hence $W^{(0)}_\alpha(t)+c$ also satisfies 
% the same estimate on $\Gamma_\R(\alpha)
% \subset \chY_\alpha$. 
Consider the integrals 
\eqref{eq:oscint}, \eqref{eq:oscresint} 
with $\Re(z)>0$, $\Gamma(\alpha)= \Gamma_\R(\alpha)$ 
and $\alpha\in (\R_{>0})^{m+c}$.  
Take $\Proj^n$ as a compactification of $\chT$.  
Then the convergence of the integral \eqref{eq:oscint}
is ensured by the exponential factor 
$e^{-W^{(0)}_\alpha(t)/z}$ because $\phi(t) \Omega_\alpha$ 
grows at most polynomially near the infinity $\Proj^n 
\setminus \chT$.   
By taking $T(\Gamma_\R(\alpha))$ 
to be a semi-algebraic cycle   
(as in \cite[Appendix]{Pham:GM}) 
which is sufficiently close to $\Gamma_\R(\alpha)$, 
we have the convergence of \eqref{eq:oscresint} 
similarly. 

For $v\in \Boxop$, we set 
$\alpha^v := \prod_{j\in \sigma} \alpha_j^{c_j}$ when 
$v = \sum_{j\in \sigma} c_j b_j$ for some 
cone $\sigma$ and $c_j\in [0,1)$. 
The following is the first main theorem of this paper.  
\begin{theorem} 
\label{thm:Aperiod=Bperiod} 
Let $\cY$ be a toric complete intersection in 
Section \ref{subsec:mirrorconst}. 
The A-periods \eqref{eq:A-period} of the structure sheaf $\cO_\cY$ 
equal the oscillatory residue integrals 
over $\Gamma_\R(\alpha)$.  
\[
\Pi_\cY(\Upsilon_v,\cO_\cY)(\varsigma(q),z) = 
\int_{\Gamma_\R(\alpha)} 
\Osc(\alpha^v t^v, \vec{\varphi}(v);\alpha). 
\]
Here $\vec{\varphi}(v) = (\varphi_1(v),\dots, \varphi_c(v))$, 
$v\in \Boxop$,  $\alpha \in (\R_{>0})^{m+s}$,  
$q = [\alpha]\in \cM_\R$, $z>0$ and 
the functions $\varsigma(q)$, 
$\Upsilon_v(q)$ are as in Corollary \ref{cor:mirrorthm_Y} 
(see also \eqref{eq:Y_mirrormap}, \eqref{eq:Upsilon}). 
In particular, the quantum cohomology central charge 
$Z_\cY(\cO_\cY)=(2\pi\iu)^{-\dim \cY}
\Pi_\cY(\unit,\cO_\cY)$ is given by 
\[
Z_\cY(\cO_\cY)(\varsigma(q),z) 
= \frac{1}{(2\pi\iu)^{\dim \cY}F(q)} 
\int_{\Gamma_\R(\alpha)} 
e^{-W^{(0)}_\alpha(t)/z} \Omega_\alpha. 
\]
Moreover, the A-period $\Pi_\cY(\Upsilon_v, \iota^*\cE)$ for 
$\cE\in K(\cX)$ is in the $\Z$-span of  
the monodromy transforms of 
$\Pi_\cY(\Upsilon_v, \cO_\cY)$ 
with respect to the monodromy around $q=\bzero$. 
\end{theorem} 
\begin{remark} 
% (1) The function $F(q)$ in the right-hand side 
% can be viewed as a normalization factor 
% of the volume form $\Omega_\alpha$. 
% In the case of Calabi-Yau hypersurfaces, 
% we will see in Theorem \ref{thm:} 
% that $F(q)$ itself is an integral period 
% of $\Omega_\alpha$. 
%
The A-periods $\Pi_\cY(\Upsilon_\gamma, \iota^*\cE)$ 
should be written as an oscillatory integral 
over an integral cycle $\Gamma_\cE$ 
which is monodromy-generated by $\Gamma_\R$, 
but we do not know its explicit representative. 
% However we do not try to justify the existence of 
% parallel transports of $\Gamma_\R$ 
% (related to the question raised 
% in Remark \ref{rem:compactification}). 
\end{remark}

\begin{remark} 
By Corollary \ref{cor:mirrorthm_Y}, 
the A-period $\Pi_\cY(\Upsilon_v, \cE)$ for 
$\cE\in K(\cY)$ can be expressed 
in terms of the explicit hypergeometric function $I_\cV^v$: 
\[
\Pi_\cY(\Upsilon_v,\cE)(\varsigma(q),z) = 
\left(\iota^* I_{\cV}^v(q,-z), 
z^{\dim \cY-\frac{\deg}{2}} z^{\rho_\cY} 
\Psi_\cY(\cE)\right)_{\rm orb}. 
\]
Hence theorem \ref{thm:Aperiod=Bperiod} gives equalities of 
oscillatory integrals and hypergeometric series. 
\end{remark} 
\begin{proof}[Proof of Theorem \ref{thm:Aperiod=Bperiod}]
The case $\cY = \cX$ was proved in \cite[Theorem 4.14]{Iritani:int}. 
In this case we have the following: 
\[
\Pi(\unit,\cO_\cX)(\tau(q),z) =  
\int_{\chT_\R} e^{-\sum_{j=0}^c 
W^{(j)}_\alpha(t)/z } 
\frac{dt_1}{t_1}\wedge \cdots \wedge \frac{dt_n}{t_n},  
\]
where $q=[\alpha]\in \cM_\R$, $\alpha \in (\R_{>0})^{m+s}$.   
By \eqref{eq:Iv_der_I} and \eqref{eq:Aperiod_Iv}, 
we know that $\Pi(\theta_v,\cO_\cX)$ can be obtained 
from $\Pi(\unit,\cO_\cX)$ by differentiation. 
Using the fact that the vector field $\bD_i$ there 
is lifted to $z\alpha_i (\partial/\partial \alpha_i)$ 
on the $\alpha$-space, we calculate 
\begin{equation*} 
%\label{eq:toric_osci}
\Pi(\theta_v,\cO_\cX)(\tau(q),z) 
= 
\int_{\chT_\R} 
\alpha^v t^v 
e^{-\sum_{j=0}^c W^{(j)}_\alpha(t)/z} 
\frac{dt_1}{t_1}\wedge \cdots \wedge 
\frac{dt_n}{t_n}.  
\end{equation*} 
By the assumption that $\chY_\alpha$ is smooth 
for generic $\alpha$, we can see that the map 
\[
\vW_\alpha:= (W^{(1)}_\alpha ,\dots ,W_\alpha^{(c)}) \colon 
\chT_\R \longrightarrow (\R_{>0})^c
\]
is generically submersive 
for generic $\alpha\in (\R_{>0})^{m+s}$. 
Hence the above oscillatory integral can be rewritten as  
\begin{align*} 
\Pi(\theta_v,\cO_\cX)(\tau(q),z) &= 
\int_{u \in (\R_{>0})^c}
\prod_{j=1}^c du_j e^{-\sum_{j=1}^c u_j/z} 
P_v(\alpha, u), \\ 
\text{where } \ 
P_v(\alpha, u) &:=  
\int_{\chT_\R \cap \{\vW_\alpha(t) = u\} } 
\alpha^v t^v e^{-W^{(0)}_\alpha(t)/z} 
\frac{\frac{dt_1}{t_1}\wedge \cdots \wedge \frac{dt_n}{t_n}}
{dW_\alpha^{(1)}\wedge \cdots \wedge dW_\alpha^{(r)}}.   
\end{align*} 
Here we set $P_v(\alpha,u)=0$ if $u$ is not in the image of 
$\vW_\alpha$. 
We take the partial Laplace transform of the both-hand sides. 
Under the divisor map $D \colon \Z^{m+s} \to \LL^*$, 
$\txi_j$ can be lifted to the sum $\sum_{i\in \hI_j} e_i \in \Z^{m+s}$. 
Therefore the $c$-dimensional 
flow $q\mapsto \prod_{j=1}^c (zr_j)^{\txi_j}\cdot q$ on $\cM$ 
can be lifted to the flow on the $\alpha$-space 
scaling $W^{(j)}_\alpha$ by $z r_j$ for $1\le j\le c$
and leaving $W^{(0)}_\alpha$ invariant. 
Hence $\hPi(\theta_v,\cO_\cX)(q,s,z)$ 
with $s\in (\R_{>0})^c$ equals 
\begin{align*} 
\left ( \prod_{k=1}^c  
\int_0^\infty
 s_k dr_k \right ) & 
\int_{u\in (\R_{>0})^c}
\left( \prod_{j=1}^c 
du_j 
e^{- (u_j+s_j) r_j} (z r_j)^{\varphi_j(v)}
\right) 
P_v(\alpha,u) \\ 
& = 
\int_{u\in (\R_{>0})^c} 
\left( \prod_{j=1}^c 
\frac{du_j s_j z^{\varphi_j(v)} \varphi_j(v)!}
{(u_j+s_j)^{1+\varphi_j(v)}} \right) 
P_v(\alpha,u). 
\end{align*}
We can change the order of 
the integration by Fubini since the integrand 
can be viewed as a non-negative measure. 
Because $\hPi(\theta_v,\cO_\cX)$ is well-defined 
by Proposition \ref{prop:Laplacetrans}, the integral 
in the right-hand side also converges for $s_j>0$. 
If $s_j \in \C \setminus \R_{\le 0}$, 
there exists a constant $C(s_j)>0$ such that 
\[
\left |\frac{1}{(u_j+s_j)^{1+\varphi_j(v)}} \right | 
\le   \frac{C(s_j)}{(u_j+1)^{1+\varphi_j(v)}} \quad 
\text{for all } u_j>0.      
\] 
Therefore the above integral can be analytically 
continued in $s$ and makes sense for 
$s_j\in \C \setminus \R_{<0}$. 
The jump of the values across the branch cut $\R_{<0}$ 
can be calculated by the Cauchy integral formula. 
For a function $f(s_j)$ on $\C\setminus \R_{<0}$ 
we set $(\Jump_jf)(-s_j) := 
\lim_{\epsilon \to +0} 
(f(-s_j-\iu \epsilon) - f(-s_j + \iu \epsilon))$ 
for $s_j>0$. 
Then we have  
\[
\left(\Jump_1 \cdots \Jump_c 
\hPi(\theta_v,\cO_\cX)\right) (q, -s_1,\dots, -s_c, z) 
= \prod_{j=1}^c (-2\pi\iu s_j)
\prod_{j=1}^c 
\left(z\parfrac{}{s_j}\right)^{\varphi_j(v)}
P_v(\alpha,s). 
\]
On the other hand, we can calculate 
the left-hand side using Proposition 
\ref{prop:Laplacetrans} when $q$ is sufficiently 
close to $\bzero$ and $s_j \ge 1$.  
It is 
% and $I_\cV(e^{2\pi\iu \txi_j} q, -z) 
% = \bigoplus_{v\in \Boxop} 
% e^{-2\pi\iu (\xi_j/z+f_v(\txi_j)} \unit_v \cup 
% I_\cV(q,-z)$, 
\begin{align*} 
& 
\left( 
\textstyle\prod_{j=1}^c (e^{-2\pi\iu \xi_j/z} -1 )\cup 
I^v_\cV\left(\textstyle\prod_{j=1}^c s_j^{-\txi_j} \cdot q,  -z\right), 
\;  
z^{n-\frac{\deg}{2}} z^{\rho_\cY} \Psi_\cV(\cO_\cX)
\right)_{\rm orb} \\
& \qquad = (-1)^c \hPi(\theta_v, \iota_* \cO_\cY)(q, e^{\pi\iu} s, z) 
\qquad \qquad \qquad \ \   
\text{by \eqref{eq:Toen_pushforward} and Proposition 
\ref{prop:Laplacetrans}} \\ 
& \qquad = (-2\pi\iu)^c \Pi_\cY(\Upsilon_v, \cO_\cY) 
\left(
\varsigma\left(\textstyle \prod_{j=1}^c s_j^{-\txi_j} \cdot q\right), 
z\right) 
\quad \text{by Corollary \ref{cor:Laplace}}. 
\end{align*} 
We arrive at the formula in the theorem by 
differentiating 
\[
P_v(\alpha,s) = \int_{T(\chT_\R \cap \{ \vW_\alpha(t) = s\})}
\frac{
\alpha^v t^v e^{-W_\alpha^{(0)}(t)/z} 
\frac{dt_1}{t_1}\wedge \cdots \wedge \frac{dt_n}{t_n}}
{(W_\alpha^{(1)}(t)-s_1) \cdots (W_\alpha^{(c)}(t)-s_c)}
\]
in $s_1,\dots,s_c$ and setting $s_1 = \cdots = s_c = 1$. 
The last statement follows from the fact 
that $K(\cX)$ is generated by line bundles 
\cite{Borisov-Horja:K} 
and the monodromy formula 
\begin{equation*} 
%\label{eq:monodromy_transform}
\Pi_\cY(\phi,\cE)(\varsigma(e^{2\pi\iu \txi} \cdot q),z) = 
\Pi_\cY(\phi,\iota^*(\cL_\xi^\vee) \otimes \cE)(\varsigma(q),z)
\end{equation*} 
where $\cE\in K(\cY)$, $\txi\in \LL^*$ and 
$\xi\in \Pic(\cX)$ is its image. 
(It follows from Definition \ref{def:intstr}, (ii) and 
$\varsigma(e^{2\pi\iu\txi}\cdot q) = G(\iota^*\xi)^{-1} 
\varsigma(q)$.) 
\end{proof} 

\section{Toric Calabi-Yau Hypersurfaces} 
In this section we restrict our attention 
to a Calabi-Yau hypersurface $\cY$ 
in a Gorenstein weak Fano toric orbifold $\cX$. 
Based on the period calculation 
in Theorem \ref{thm:Aperiod=Bperiod}, 
we study mirror symmetry of $\cY$ as an isomorphism of  
variations of Hodge structure 
and compare the integral structures on 
the both sides. We set $n:= \dim_\C \cX$ as before.

\subsection{Batyrev Mirror} 
Batyrev \cite{Batyrev:dual} constructed mirror 
pairs of Calabi-Yau hypersurfaces based on the 
duality of reflexive polytopes. 
This is the case where the ambient toric 
orbifold $\cX$ is Gorenstein (i.e.\ $K_\cX$ is pulled back 
from the coarse moduli space) 
and we take the nef partition 
$I_0= \emptyset$, $I_1 = \{1,\dots,m\}$ 
in Section \ref{subsec:mirrorconst}. 
We refer the reader to \cite{Batyrev:dual}, 
\cite[Section 4]{Cox-Katz} for details. 
The fan polytope $\Delta$ is said to be \emph{reflexive} 
if the integral distance between 0 and 
all hyperplanes generated by codimension 
one faces equal 1. 
If $\Delta$ is reflexive, the dual polytope 
$\Delta^*=\{y\in \bM_\R\,|\, \pair{x}{y} \ge -1\, 
(\forall x\in \Delta)\}$ have integral points 
as its vertices  and $(\Delta^*)^* = \Delta$. 
A weak Fano toric orbifold is Gorenstein 
if and only if its fan polytope is reflexive. 
% Furthermore, we assume that 
% \begin{itemize}
% \item[($\star$)] 
% $k\Delta \cap \bN = 
% (\Delta\cap \bN) + \cdots + (\Delta\cap \bN)$
% ($k$ fold sum) for all $k\in \N$. 
% \end{itemize} 
% This implies that $\Delta\cap \bN$ generates $\bN$ over $\Z$. 
% The condition ($\star$) is, however, more restrictive. 
% For example $\cX= \Proj(1,1,1,2,5)$ 
% does not satisfy ($\star$). 
% It ensures that the ambient quantum $D$-module of $\cY$ 
% (introduced later) is generated by $\unit$. 
% It is satisfied if $\cX$ is a smooth toric 
% manifold or admits a smooth toric crepant resolution. 
A general anticanonical section $\cY$ in a weak Fano 
Gorenstein toric orbifold $\cX$ 
is a quasi-smooth Calabi-Yau orbifold. 
As in Section \ref{subsec:mirrorconst}, 
the (uncompactified) mirror $\chY_\alpha$ of $\cY$ is 
given as a hypersurface in the torus $\chT$: 
\[
\chY_\alpha = W_\alpha^{-1}(1) \subset \chT, \quad 
W_\alpha(t) := W^{(1)}_\alpha(t) = \sum_{i=1}^{m+s} \alpha_i t^{b_i}.  
\]
Take $\C^\times$ co-ordinates $(t_1,\dots,t_n)$ on 
$\chT$ as before. 
Introduce another variable $t_0$ and 
define the subring $S_\Delta$ of $\C[t_0,t_1^\pm,\dots,t_n^\pm]$ 
by 
\[
S_\Delta = \bigoplus_{k\ge 0} S_\Delta^k, \quad 
S_\Delta^k = \bigoplus_{b\in  k \Delta \cap \bN}
\C t^b t_0^k.  
\]
It is graded by the degree of $t_0$. 
The toric variety $\Proj_\Delta :=\PROJ S_\Delta$ 
is a compactification of $\chT$. 
The variety $\Proj_\Delta$ is associated with 
the normal fan of $\Delta$ and its fan polytope 
is $\Delta^*$. 
The closure $\ov{\chY}_\alpha$ of 
$\chY_\alpha$ in $\Proj_\Delta$ is an anticanonical 
section of $\Proj_\Delta$. 
We say that $\chY_\alpha$ is \emph{$\Delta$-regular} 
if the intersection of $\ov{\chY}_\alpha$ 
and every torus orbit $O$ in 
$\Proj_\Delta$ is a smooth subvariety of 
codimension 1 in $O$. 
% tsugi ha problematic 
% This is equivalent to the 
% condition that for each face $\Delta'$ of $\Delta$ 
% and the Laurent polynomial 
% $W_{\alpha,\Delta'}(t) := \sum_{b_i \in \Delta'} \alpha_i t^{b_i}$, 
% the equations  
% \[
% W_{\alpha, \Delta'}(t) = t_1 \parfrac{}{t_1} W_{\alpha, \Delta'}(t) 
% = \cdots = t_n \parfrac{}{t_n} W_{\alpha,\Delta'}(t) = 0 
% \]
% have no solutions on $\chT$. 
% Here the case $\Delta' = \Delta$ is included. 
The set of parameters $\alpha$ for which $\chY_\alpha$ 
is $\Delta$-regular is a non-empty Zariski open subset 
$(\C^\times)^{m+s}_{\rm reg}$ of $(\C^\times)^{m+s}$.  
%\cite[Proposition 3.1.3]{Batyrev:dual}.  
This is invariant under the action of $\chT$ and 
descends to a Zariski open subset $\cM_{\rm reg}$ in $\cM$. 
Let $\chY_\alpha$ be $\Delta$-regular. 
A resolution of $\ov{\chY}_\alpha$ 
by a Calabi-Yau orbifold is constructed as follows.  
Choose a projective crepant resolution 
$\chcX \to \Proj_\Delta$ by a toric orbifold $\chcX$. 
This amounts to choosing a triangulation of the fan polytope 
$\Delta^*$. The fan polytope of $\chcX$ is 
still $\Delta^*$. 
Then the pull-back $\chcY_\alpha\subset \chcX$ 
of $\ov{\chY}_\alpha$ 
is a quasi-smooth Calabi-Yau hypersurface 
which gives a crepant resolution of $\ov{\chY}_\alpha$. 
\begin{remark} 
Batyrev \cite{Batyrev:dual} showed that one can 
choose $\chcX$ with only terminal singularities 
(MPCP desingularization). 
In this case $\chcY_\alpha$ becomes also terminal. 
In particular, we can take $\chcY_\alpha$ to be 
smooth in dimension 3 because terminal Gorenstein orbifolds 
in dimension 3 are all smooth. 
From a viewpoint of orbifold mirror symmetry, we do not 
need to restrict our attention to terminal 
partial resolutions. 
\end{remark} 

\subsection{A-Model VHS of a Calabi-Yau Hypersurface}
By the Gorenstein condition, the orbifold cohomology 
of $\cY$ is graded by even integers\footnote{
Recall that we ignore odd cohomology classes 
on the A-side.}. 
Since $\cY$ is Calabi-Yau, 
the quantum $D$-module $QDM(\cY)$ of $\cY$ 
in Definition \ref{def:QDM} restricted to 
$H^2_{\rm orb}(\cY)$ and $z=1$ reduces to a 
\emph{variation of Hodge structure} (VHS)
in the classical sense 
(cf.\ Remark \ref{rem:gen_Hodge}). 
See \cite[Section 8.5]{Cox-Katz}. 
We furthermore restrict our attention 
to the ``ambient part". 
Set 
\[
H_{\rm amb}^*(\cY) := \Image(\iota^* \colon 
H^*_{\rm orb}(\cX) \to H^*_{\rm orb}(\cY)). 
\]
By Corollary \ref{cor:iota_quantumhom}, 
$H_{\rm amb}^*(\cY)$ is closed under quantum product. 
For the convergence domain $U\subset H_{\rm orb}^*(\cY)$ 
in Section \ref{subsec:untwQDM}, we set 
$U' = H_{\rm amb}^2(\cY) \cap U$.  
Take a $\C$-basis $\eta_1,\dots,\eta_\ell$ in $H_{\rm amb}^2(\cY)$. 
Let $\tau^1,\dots,\tau^\ell$  
be the corresponding co-ordinates on $H_{\rm amb}^2(\cY)$ 
and $\tau = \sum_{i=1}^\ell \tau^i \eta_i$ be a 
general point on $H_{\rm amb}^2(\cY)$. 

\begin{definition}
The \emph{ambient A-model VHS} 
of $\cY$ is the tuple $(\rsH_{\rm A}, \nabla^{\rm A}, 
\rsF^{\, \bullet}_{\rm A}, Q_{\rm A})$ 
consisting of the locally free sheaf 
$\rsH_{\rm A} = H_{\rm amb}^*(\cY) \otimes \cO_{U'}$ over $U'$,  
the flat Dubrovin connection 
$\nabla^{\rm A} \colon \rsH_{\rm A} \to \rsH_{\rm A} 
\otimes \Omega_{U'}^1$ 
\[
\nabla^{\rm A} = d + \sum_{i=1}^\ell (\eta_i \circ_\tau) d \tau^i,  
\] 
the decreasing Hodge filtration 
$\rsF^{\, p}_{\rm A} = H_{\rm amb}^{\le 2(n-1-p)}(\cY) 
\otimes \cO_{U'}$ on $\rsH_{\rm A}$ and 
the $\nabla^{\rm A}$-flat 
$(-1)^{n-1}$-symmetric pairing 
$Q_{\rm A} \colon \rsH_{\rm A} \otimes \rsH_{\rm A} \to \cO_{U'}$
\begin{equation*} 
%\label{eq:Q_A}
Q_{\rm A}(\alpha, \beta) = (2\pi\iu)^{n-1} 
((-1)^{\frac{\deg}{2}} \alpha, \beta)_{\rm orb}.    
\end{equation*} 
The \emph{Galois action} of $\iota^*H^2(\cX;\Z)$ on the A-model 
VHS is defined similarly to \eqref{eq:Galois}, 
\eqref{eq:Galois_bundleisom}. 
\end{definition} 
The A-model VHS satisfies 
Griffiths transversality and 
the Riemann bilinear relation: 
\begin{equation} 
\label{eq:Griffiths_bilineareq} 
\nabla^{\rm A} \rsF^{\,p}_{\rm A} \subset \rsF^{\,p-1}_{\rm A}
\otimes \Omega_{U'}^1, 
\quad 
Q_{\rm A}(\rsF_{\rm A}^{\,p},\rsF_{\rm A}^{\,n-p}) =0. 
\end{equation} 
By a result of Mavlyutov \cite[Theorem 5.1]{Mavlyutov:chiral}, 
we have the decomposition $
H^*_{\rm orb}(\cY) = H^*_{\rm amb}(\cY) \oplus \Ker (\iota_*)$. 
Since the two summands are orthogonal to each other 
with respect to the orbifold Poincar\'{e} pairing, 
we know that the polarization form $Q_{\rm A}$ is non-degenerate 
on the ambient part $\rsH_{\rm A}$. 

The $\hGamma$-integral structure in 
Definition \ref{def:intstr} induces 
an integral structure on the above A-model VHS. 
Let $L_\cY(\tau) := L_\cY(\tau,z=1)$ denote the 
fundamental solution of the quantum differential 
equation (with $Q=1$) of $\cY$. 
By Proposition \ref{prop:Eulertw_Y}, 
we know that $L_\cY(\tau)$ with $\tau\in U'$ 
maps a class 
in $H^*_{\rm amb}(\cY)$ to a flat section of $\rsH_{\rm A}$. 
Therefore, if $\cE \in K(\cY)$ satisfies 
$\tch(\cE) \in H^*_{\rm amb}(\cY)$, 
we have a flat section $\frs(\cE)(\tau) 
= \frs(\cE)(\tau,z=1)$ of $\rsH_{\rm A}$ 
in the same way as \eqref{eq:Psi}: 
\[
\frs(\cE)(\tau) = (2\pi\iu)^{-(n-1)} 
L_\cY(\tau) \Psi_\cY(\cE), \quad 
\Psi_\cY(\cE) = \hGamma_\cY 
\cup (2\pi\iu)^{\frac{\deg_0}{2}} 
\inv^* \tch(\cE).   
\]
\begin{definition} 
Set $H_{\rm A} := \Ker\nabla^{\rm A} \subset \rsH_{\rm A}$.  
Define the local subsystem 
$H_{{\rm A},\Z}^{\rm amb}$ of $H_{\rm A}$ as  
\begin{align*} 
H_{{\rm A},\Z}^{\rm amb} 
&:= \{ \frs(\iota^*\cE) \,|\, \cE \in K(\cX)\}. 
\end{align*} 
It is a $\Z$-lattice of $H_{\rm A}$ 
and preserved under the Galois action 
by $\iota^*H^2(\cX;\Z)$. 
For $\frs(\cE_1), \frs(\cE_2) \in H_{{\rm A},\Z}^{\rm amb}$, 
we have  
\begin{equation*}
%\label{eq:Q_A_Euler}
Q_{\rm A}(\frs(\cE_1), \frs(\cE_2)) =  \chi_\cY (\cE_1, \cE_2). 
\end{equation*}
We call $H_{{\rm A},\Z}^{\rm amb}$ 
the \emph{ambient $\hGamma$-integral structure} 
on the ambient A-model VHS of $\cY$. 
\end{definition} 
\begin{remark} 
The above integral structure also induces rational 
and real structures $H_{{\rm A},\Q}$, $H_{{\rm A},\R}$ 
on the A-model VHS. 
With respect to the real involution $\kappa$ 
defined by $H_{{\rm A},\R}$, we hope to have the 
Hodge decomposition and the Riemann bilinear inequality:  
\begin{equation} 
\label{eq:Hodgedecomp_bilinearineq} 
\rsH_{\rm A} = 
\rsF^{\, p}_{\rm A} \oplus \kappa(\rsF^{\, n-p}_{\rm A}), \quad 
\iu^{p-q} Q_{\rm A}(\phi,\kappa(\phi)) > 0 
\end{equation} 
for $\phi\in \rsH^{\ p,q}_{\rm A} = \rsF_{\rm A}^{\, p} \cap 
\kappa(\rsF^{\, q}_{\rm A}) \setminus \{0\}$, $q= n-1-p$. 
From a result in \cite{Iritani:tt*}, it follows that 
these properties hold in a neighbourhood of the large 
radius limit\footnote{We showed 
the Riemann bilinear inequality for the $(p,p)$ part 
(or algebraic part) of quantum cohomology VHS in \cite{Iritani:tt*}. 
Note that the ambient part $H^*_{\rm amb}(\cY)$ 
is contained in the $(p,p)$ part.}. 
In fact, by mirror symmetry, we will see that 
these properties hold globally.  
\end{remark} 

\subsection{B-Model VHS} 
As we saw in Section \ref{subsec:mirrorconst}, 
the parameter space of the mirror $\chY_\alpha$ 
(or its compactification $\chcY_\alpha$) 
descends to $\cM = \LL^* \otimes \C^\times$. 
We are interested in the VHS associated with the 
family of $\Delta$-regular Calabi-Yau hypersurfaces:
\[
\pr_2 \colon 
\chfrY \overset{\rm def}{=} 
\left\{(t, \alpha) \in 
\chcX \times (\C^\times)^{m+s}_{\rm reg}  
\; \big |\; t\in \chcY_\alpha\right \}\big /\chT 
\longrightarrow \cM_{\rm reg}. 
\]
In this paper, we restrict our attention to the VHS 
on the untwisted middle cohomology $H^{n-1}(\chcY_\alpha)$. 
Furthermore, we only consider classes obtained as residues.  
Let $\Res \colon H^n(\chcX\setminus \chcY_\alpha) 
\to H^{n-1}(\chcY_\alpha)$ be the Poincar\'{e} residue 
map: 
\[
\int_{\gamma} \Res(\omega) = 
\frac{1}{2\pi\iu} 
\int_{T(\gamma)} \omega, \quad \omega \in 
H^n(\chcX\setminus \chcY_\alpha),  
\]
where $T(\gamma)\subset \chcX\setminus \chcY_\alpha$ 
is a tube of an $(n-1)$-cycle $\gamma \subset 
\chcY_\alpha$ (as appeared in \eqref{eq:oscresint}). 
Define the residue part of $H^{n-1}(\chcY_\alpha)$ by 
\[
H_{\rm res}^{n-1}(\chcY_\alpha) := \Image( \Res \colon 
H^0(\chcX, \Omega^n_{\chcX}(*\chcY_\alpha)) 
\to H^{n-1}(\chcY_\alpha)).   
\]
Here $H^0(\chcX, \Omega^n_{\chcX}(*\chcY_\alpha))$ 
is the space of algebraic $n$-forms 
with arbitrary poles along $\chcY_\alpha$. 
% An explicit description of this space is given in 
% \cite[Theorem 9.7]{Batyrev-Cox}. 
Let $\chD_1,\dots,\chD_N$ be the toric divisors 
of $\chcX$. 
We claim that we have the orthogonal decomposition 
with respect to the intersection pairing: 
\begin{equation} 
\label{eq:decomp_Mav}
H^{n-1}(\chcY_\alpha) = H_{\rm res}^{n-1}(\chcY_\alpha) 
\oplus \left( 
\sum_{i=1}^N f_{i*} H^{n-3}(\chcY_\alpha \cap \chD_i) 
\right) 
\end{equation} 
where $f_i \colon \chD_i \cap \chcY_\alpha 
\hookrightarrow \chcY_\alpha$ 
is the inclusion. 
To see this, we use the Gysin 
exact sequence (see \cite[Eqn (7)]{Mavlyutov:semiample}): 
\begin{equation}
\label{eq:Gysinseq}
\begin{CD}
\bigoplus_{i=1}^N H^{n-3}(\chcY_\alpha \cap \chD_i) 
@>{\bigoplus f_{i*}}>> H^{n-1}(\chcY_\alpha) 
@>>> W_{n-1}(H^{n-1}(\chY_\alpha)) @>>> 0.   
\end{CD} 
\end{equation} 
Here $W_\bullet(H^{n-1}(\chY_\alpha))$ 
denotes the weight filtration 
of Deligne's mixed Hodge structure. 
In the proof of \cite[Theorem 4.4]{Mavlyutov:semiample}, 
it is shown that the composition 
$H^0(\chcX, \Omega^n_{\chcX}(*\chcY_\alpha)) \to 
H^{n-1}(\chcY_\alpha) \to W_{n-1}(H^{n-1}(\chY_\alpha))$ 
is surjective. 
Hence $H_{\rm res}^{n-1}(\chcY_\alpha)$ and 
$\sum_{i=1}^N f_{i*}H^{n-3}(\chcY_\alpha\cap \chD_i)$ generate 
$H^{n-1}(\chcY_\alpha)$. On the other hand, 
$H^{n-1}_{\rm res}(\chcY_\alpha)$ and 
$f_{i*}H^{n-3}(\chcY_\alpha \cap \chD_i)$ are orthogonal 
to each other 
because $\Res(\omega)$ for a holomorphic $n$-form $\omega$ 
on $\chcX\setminus \chcY_\alpha$ 
vanishes on $\chcY_\alpha \cap \chD_i$.  
This proves the claim. 
The decomposition \eqref{eq:decomp_Mav} 
gives a topological characterization 
of $H_{\rm res}^{n-1}(\chcY_\alpha)$: 
It consists of degree $n-1$ classes 
on $\chcY_\alpha$ which vanish on 
the toric boundaries $\chcY_\alpha \cap \chD_i$. 
Therefore the subspace $H_{\rm res}^{n-1}(\chcY_\alpha)$ 
is defined over $\Q$ and is 
preserved by the Gauss-Manin connection. 
We can also see that each class 
$\alpha \in H_{\rm res}^{n-1}(\chcY)$ is primitive, i.e.\ 
$\alpha \cup H =0$ for an ample hyperplane class $H$. 
By \eqref{eq:decomp_Mav} and \eqref{eq:Gysinseq}, 
we have the identification 
\[
H_{\rm res}^{n-1}(\chcY_\alpha) \cong W_{n-1}(H^{n-1}(\chY_\alpha)). 
\]
Since $W_{n-1}(H^{n-1}(\chY_\alpha))$ 
is the lowest weight component, this identification 
induces a $\Q$-Hodge structure of weight $n-1$ 
on $H^{n-1}_{\rm res}(\chcY_\alpha)$. 
\begin{definition}  
The \emph{residual B-model VHS} of 
the family $\pr_2 \colon \chfrY \to \cM_{\rm reg}$ 
is a tuple 
$(\rsH_{\, \rm B}, \nabla^{\rm B}, H_{\rm B,\Q}, 
\rsF^{\,\bullet}_{\rm B}, 
Q_{\rm B})$ where $\rsH_{\,\rm B}$ is the locally free subsheaf 
of $(R^{n-1}{\pr_2}_* \C_\chfrY)\otimes \cO_{\cM_{\rm reg}}$
over $\cM_{\rm reg}$ 
whose fiber at $[\alpha]$ is the residue 
part $H_{\rm res}^{n-1}(\chcY_\alpha)$, 
$\nabla^{\rm B}$ is the Gauss-Manin connection, 
$H_{\rm B,\Q} 
\subset \Ker\nabla^{\rm B}$ is the 
rational structure explained above,  
$\rsF^{\, p}_{{\rm B}, [\alpha]} = 
\bigoplus_{j\ge p} 
H^{j,n-1-j}_{\rm res}(\chcY_\alpha)$ 
is the standard Hodge filtration and 
$Q_{\rm B}$ is the intersection form: 
\[
Q_{\rm B}(\omega_1,\omega_2) = (-1)^{(n-1)(n-2)/2}
\int_{\chcY_\alpha} \omega_1 \cup \omega_2. 
\]
\end{definition} 

The residual B-model VHS 
satisfies the usual properties of 
a variation of \emph{polarized} Hodge structure 
as given in \eqref{eq:Griffiths_bilineareq} and 
\eqref{eq:Hodgedecomp_bilinearineq}. 
% Note that $Q_{\rm B}$ is non-degenerate by 
% the orthogonal decomposition \eqref{eq:decomp_Mav}. 

Consider the relative homology group 
$H_*(\chT,\chY_\alpha)$. 
The Morse-theoretic argument in \cite[Section 3.3.1]{Iritani:int} 
(see also \cite{Danilov-Khovanskii}) shows that 
\begin{equation} 
\label{eq:relhom} 
H_k(\chT,\chY_\alpha;\Z) 
\cong H_k(\chT, \{\Re(W_\alpha(t))\gg 0\} ;\Z) 
\cong 
\begin{cases} 
0 & k \neq n \\ 
\Z^{\Vol(\Delta)} & k=n 
\end{cases} 
\end{equation} 
where $\Vol(\Delta)$ is the normalized volume 
of $\Delta$ such that the volume of the standard $n$-simplex 
is 1. The group $H_n(\chT,\{\Re(W_\alpha(t)) \gg 0\})$ 
is generated by \emph{Lefschetz thimbles} emanating 
from critical points of $W_\alpha(t)$;  
$\Vol(\Delta)$ is the number of critical points
of $W_\alpha(t)$ (with multiplicities).  
By the relative homology exact sequence, we have 
\begin{equation*} 
\label{eq:relhom_exactseq}
\begin{CD} 
0 @>>> H_n(\chT) @>>> H_n(\chT,\chY_\alpha) 
@>{\partial}>> H_{n-1}(\chY_\alpha) @>>> H_{n-1}(\chT) @>>>0. 
\end{CD}
\end{equation*} 
% \[
% \begin{CD}
% 0 @>>> H^{n-1}(\chT) @>>> H^{n-1}(\chY_\alpha) 
% @>{\delta}>> H^n(\chT, \chY_\alpha) 
% @>>> H^n(\chT) @>>> 0 
% \end{CD} 
% \]
% Here we used the fact that $H^{n-1}(\chT) \to H^{n-1}(\chY_\alpha)$ 
% is injective \cite{Danilov-Khovanskii} and 
% $H^n(\chY_\alpha)=0$. 
The image of $H_n(\chT,\chY_\alpha;\Z)$ under 
$\partial$ consists of the \emph{vanishing cycles} 
of $W_\alpha(t)$. 
\begin{lemma}
\label{lem:VC}
The image of the composition 
\[ 
\begin{CD} 
H_n(\chT,\chY_\alpha) @>{\partial}>> H_{n-1}(\chY_\alpha) 
@>>> H_{n-1}(\chcY_\alpha) 
@>{\PD}>> H^{n-1}(\chcY_\alpha) 
\end{CD} 
\]
is $H_{\rm res}^{n-1}(\chcY_\alpha)$. 
Here $\PD$ is the Poincar\'{e} duality isomorphism (defined over $\Q$).   
We denote by $\VC \colon H_n(\chT,\chY_\alpha) 
\to H_{\rm res}^{n-1}(\chY_\alpha)$ the 
resulting surjection. 
\end{lemma} 
\begin{proof} 
It is clear that an image of the above map vanishes 
on the toric boundaries $\chcY_\alpha \cap \chD_i$. 
Thus the image is contained in $H_{\rm res}^{n-1}(\chcY_\alpha)$. 
The dual of this map is given by 
\begin{equation} 
\label{eq:dualincl}
H_{\rm res}^{n-1}(\chcY_\alpha)^\vee 
\cong H_{\rm res}^{n-1}(\chcY_\alpha) 
\cong W_{n-1}H^{n-1}(\chY_\alpha) \hookrightarrow 
H^{n-1}(\chY_\alpha) \overset{\delta}{\longrightarrow} 
H^n(\chT,\chY_\alpha) 
\end{equation} 
where the first isomorphism is via the intersection 
pairing. The map $\delta$ is dual to $\partial$ 
and its kernel is $H^{n-1}(\chT)$. 
Because the intersection of $H^{n-1}(\chT)$ 
and $W^{n-1}H^{n-1}(\chY_\alpha)$ is zero for 
the weight reason, the above dual map is injective. 
\end{proof} 

\begin{definition} 
The \emph{vanishing cycle integral structure} 
$H_{\rm B, \Z}^{\rm vc} \subset H_{\rm B,\Q}$
on the residual B-model VHS is defined to 
be the image of $H_n(\chT,\chY_\alpha;\Z)$ 
under the vanishing cycle map $\VC 
\colon H_n(\chT,\chY_\alpha) \to 
H_{\rm res}^{n-1}(\chcY_\alpha)$. 
\end{definition} 

\begin{remark} 
The mixed Hodge structure of affine hypersurfaces 
in the algebraic tori was studied 
by Danilov-Khovanskii \cite{Danilov-Khovanskii} and  
Batyrev \cite{Batyrev:mixed}. 
The Hodge structure of toric hypersurfaces 
has been studied by Batyrev-Cox \cite{Batyrev-Cox} 
and Mavlyutov \cite{Mavlyutov:semiample, Mavlyutov:chiral}.
\end{remark} 

\subsection{Mirror Isomorphism 
with Integral Structures} 
In this section, as we did in Section 
\ref{subsec:toricorb}, we assume that 
$\bN$ is generated by $\Delta\cap \bN$ 
as a $\Z$-module.  
Also we choose \emph{non-zero} vectors 
$\{b_{m+1},\dots, b_{m+s}\} 
\subset \Delta\cap \bN \setminus \{b_1,\dots,b_m\}$ 
such that $b_1,\dots,b_{m+s}$ generate $\bN$ over $\Z$. 
Because $\Delta$ is reflexive, every $b_i$ has to be 
on the boundary of $\Delta$. 
Then the lift $\txi_1$ of $\xi_1= 
\ov{D}_1+\cdots + \ov{D}_m$ defined 
in Section \ref{subsec:mirrorthm_comp} 
equals $\sum_{i=1}^{m+s} D_i$ 
and so $\hrho_\cY = \hrho - \txi_1=0$. 
Thus the degrees of the variables $q_1,\dots,q_{r+s}$ are zero. 
By the homogeneity of the $I$-function 
$I_\cV$, the mirror map $\varsigma(q) 
= \iota^* \tvarsigma(q)$ for $\cY$ 
(see \eqref{eq:Y_mirrormap}) takes values 
in $H^2_{\rm amb}(\cY)$. 

We briefly review  
the mirror isomorphism of $D$-modules 
for a toric orbifold $\cX$ in \cite{Iritani:int}. 
We can associate the \emph{B-model $D$-module} 
$(\cRz, \nabla, (\cdot,\cdot)_{\cRz})$ 
to the Landau-Ginzburg mirror $(\chT, W_\alpha(t))$ 
of $\cX$. It is a meromorphic flat connection 
over $\cMo \times \C$ with poles along $z=0$ 
such that the underlying $\Z$-local system at $([\alpha],z) 
\in \cMo \times \C^\times$ is given by 
the lattice $H^n(\chT, \{\Re(W_\alpha(t)/z)\ll 0\};\Z)$. 
Here $\cMo$ is a Zariski open subset of $\cM$ containing 
$\cM_{\rm reg}$. 
The oscillatory form 
$\phi(t) e^{W_\alpha(t)/z} \frac{dt_1}{t_1}\wedge \cdots 
\wedge \frac{dt_n}{t_n}$ 
appearing in Section \ref{subsec:periodcal} 
gives a section of the B-model $D$-module $\cRz$. 
We have the mirror isomorphism\footnote
{Because we changed our convention of $QDM(\cX)$ 
from \cite{Iritani:int} (see Remark \ref{rem:change}), 
we also need to modify the definition of the B-model 
$D$-module accordingly. The necessary modification 
makes the B-model $D$-module more natural. 
We defined the integration pairing of a section 
$[\phi(t) e^{W(t)/z} \frac{dt_1}{t_1} \cdots 
\frac{dt_n}{t_n}]$ of $\cRz$ and a Lefschetz thimble
with the additional factor of $(-2\pi z)^{-n/2}$ 
in \cite[Eqn.\ (53)]{Iritani:int}.  
Then the integral structure, the flat connection 
and the pairing on $\cRz$ were introduced \emph{through}  
this integration pairing. 
We just need to remove the factor $(-2\pi z)^{-n/2}$ there 
to redefine these ingredients.}
\cite[Proposition 4.8]{Iritani:int} in a neighbourhood 
of $q=\bzero$ 
\[
\Mir_\cX \colon (\tau\times \id)^*(QDM(\cX)/H^2(\cX;\Z)) \cong 
(\cRz, \nabla, (\cdot,\cdot)_{\cRz})  
\]
sending the unit section $\unit$ to 
$[e^{W_\alpha(t)/z} \frac{dt_1}{t_1} \wedge  
\cdots \wedge \frac{dt_n}{t_n}]$.    
Here $\tau$ is the mirror map in \eqref{eq:X_mirrormap} 
and $H^2(\cX;\Z)$ acts by Galois action. 
This induces an inclusion of $\Z$-lattices: 
\[
\Mir_\cX^\Z \colon K(\cX) \hookrightarrow 
H_n(\chT, \{\Re(W_\alpha(t)/z) \gg 0\} ;\Z )  
\]
such that for $\Gamma_\cE = \Mir_\cX^\Z(\cE)$, 
\[
\bigl(\phi(q,-z), \frs(\cE)(\tau(q),z)
\bigr)_F = 
\int_{\Gamma_\cE} \Mir_\cX( \phi(q,-z) )
\]
for any section $\phi(q,z)$ of 
$(\tau\times \id)^*QDM(\cX)$. 
(It is known by Hua \cite{Hua:K-toric} that 
$K(\cX)$ is a free $\Z$-module and thus 
$K(\cX) \cong \Sol(\cX)$.) 
For $z>0$ and $\alpha \in (\R_{>0})^{m+s}$, 
$\Mir_\cX^\Z$ sends the structure sheaf $\cO_\cX$  
to the real Lefschetz thimble $\chT_\R$. 
By \cite[Theorem 4.11]{Iritani:int}, 
the map $\Mir_\cX^\Z$ gives 
an isomorphism of lattices under the assumption  
\begin{equation*} 
% \label{eq:assump_Kpairing} 
K(\cX) \to \Hom(K(\cX),\Z), \ \cE \mapsto \chi(\cdot, \cE), \ 
\text{is surjective.}
\end{equation*} 
This holds true in our case because 
$\cX$ does not have generic stabilizers and 
we can apply the result of Kawamata \cite{Kawamata:toric} 
that the derived category $D^b(\cX)$ of coherent sheaves 
has a full exceptional collection $\{\cE\}_{i=1}^N$. 
(In this case, $\{\cE_i\}_{i=1}^N$ forms a $\Z$-basis of $K(\cX)$ 
and the Gram matrix $\chi(\cE_i,\cE_j)$ is an upper-triangular 
matrix with diagonal entries all equal to one.) 
Therefore $\Mir_\cX^\Z$ is an isomorphism. 
% Here we used the Poincar\'{e} duality\footnote
% {This is defined by the intersection pairing 
% $H_n(\chT,\{\Re(W_\alpha(t)/z) \ll 0\};\Z) 
% \times H_n(\chT,\{\Re(W_\alpha(t)/z) \gg 0\};\Z) 
% \to \Z$ which is perfect.}  
% \[
% \PD\colon 
% H_n(\chT, \{\Re(W_\alpha(t)/z)\gg 0\};\Z)
% \cong H^n(\chT,\{\Re(W_\alpha(t)/z)\ll 0\};\Z)
% \]
% to identify the relative homology with the relative 
% cohomology. 
If $q = [\alpha]$ is sufficiently 
close to $\bzero$, all the critical values of 
$W_\alpha(t)$ are contained in the ball 
$\{u\in \C\,|\,|u|\le 1/2\}$. 
Then we have the canonical identification for $z>0$:  
\begin{equation}
\label{eq:relativehomologies}
H_n(\chT, \{\Re(W_\alpha(t)/z)\gg 0\};\Z)
\cong 
H_n(\chT, \{\Re(W_\alpha(t))\ge 1\};\Z) 
\cong H_n(\chT,\chY_\alpha;\Z).  
\end{equation} 
The following is the second main theorem of the paper. 
\begin{theorem}
\label{thm:intstr_Batyrevmirror} 
Let $(\cY, \chcY_\alpha)$ be a Batyrev's mirror pair 
of Calabi-Yau hypersurfaces. 
The ambient A-model VHS of $\cY$ and 
the residual B-model VHS of $\chcY_\alpha$ are 
isomorphic: 
\[
\Mir_\cY \colon 
\varsigma^*\left(
(\rsH_{\rm A}, \nabla^{\rm A}, 
\rsF^{\, \bullet}_{\rm A}, Q_{\rm A})/\iota^*H^2(\cX;\Z)
\right) 
\cong 
(\rsH_{\, \rm B}, \nabla^{\rm B}, 
\rsF^{\, \bullet}_{\rm B}, Q_{\rm B})  
\]
in a neighbourhood of $q = \bzero$. 
We have the following correspondence  
under $\Mir_\cY$: 
\[
\Mir_\cY \colon \Upsilon_v(q,z=1) 
\longmapsto 
(-1)^{\age(v)} \age(v)! \Res\left (
\frac{\alpha^v t^v \frac{dt_1}{t_1}
\wedge \cdots \wedge \frac{dt_n}{t_n}}
{(W_\alpha(t)-1)^{\age(v)+1}}
\right), 
\]
where $v\in \Boxop$, $q = [\alpha]$ and 
$\Upsilon_v(q,1)$ is the section of 
$\varsigma^*\rsH_{\rm A}$ in Corollary \ref{cor:mirrorthm_Y} 
(see \eqref{eq:Upsilon}). 
In particular, $F(q) \unit$ 
corresponds to the holomorphic volume form 
$\Omega_\alpha= \frac{dt_1}{t_1}
\wedge \cdots \wedge \frac{dt_n}{t_n}/dW_\alpha$ 
on $\chcY_\alpha$. 
The mirror isomorphism $\Mir_\cY$ induces an 
isomorphism 
\[
\Mir_\cY^\Z \colon 
H^{\rm amb}_{{\rm A},\Z} \cong 
H^{\rm vc}_{{\rm B},\Z}
\] of $\Z$-local systems. 
Moreover, when $z>0$ and $q=[\alpha]$ is sufficiently 
close to $\bzero$, we have the commutative diagram: 
\begin{equation} 
\label{eq:intstr_XandY} 
\begin{CD} K(\cX) @>{\Mir_\cX^\Z}>{\cong}> 
H_n(\chT, \{\Re(W_\alpha(t)/z) \gg 0\};\Z)\\ 
@V{\iota^*}VV @VV{\epsilon_{n-1} \VC}V \\
H^{\rm amb}_{{\rm A},\Z} @>{\Mir_\cY^\Z}>{\cong}> 
H^{\rm vc}_{{\rm B},\Z} 
\end{CD}
\end{equation} 
where $\epsilon_{n} = (-1)^{n(n-1)/2}$. 
Here we used the identification \eqref{eq:relativehomologies}
to define the right vertical arrow $\VC$ 
(see Lemma \ref{lem:VC}). 
The left vertical arrow $\iota^*$ sends 
$\cE$ to $\frs(\iota^*\cE)(\varsigma(q))$ 
for $\cE\in K(\cX)$. 
% The two horizontal arrows $\Mir_\cX^\Z$, $\Mir_\cY^\Z$ 
% are isomorphisms under the assumption \eqref{eq:assump_Kpairing}. 
\end{theorem} 
\begin{proof} 
Consider the map 
$\rsR \colon S_\Delta^+ \to H^0(\chT, 
\Omega^n_{\chT}(*\chY_\alpha))$ \cite{Batyrev:mixed} 
defined by 
\[
\rsR(t^b t_0^k ) = (-1)^{k-1} (k-1)! 
\frac{t^b }{(W_\alpha(t)-1)^{k}}
\frac{dt_1}{t_1} \wedge \cdots \wedge \frac{dt_n}{t_n}, \quad 
k>0. 
\]
Let $I_\Delta^{(1)}$ be the ideal of $S_\Delta$ 
spanned by the monomials $t^b t_0^k$ such that 
$b$ is in the interior of $k \Delta$. 
Batyrev showed \cite[Theorem 8.1, 8.2]{Batyrev:mixed} 
that $\rsR(t^b t_0^k)$ with $t^b t_0^k\in I_\Delta^{(1)}$, 
$k\le p$ generate the Hodge filter $F^{n+1-p}W_{n+1}(H^n(\chT\setminus 
\chY_\alpha))$. Because the Poincar\'{e} residue map 
$H^n(\chT\setminus\chY_\alpha) \to H^{n-1}(\chY_\alpha)$ 
is a surjective morphism of mixed Hodge structures  
of the Hodge type $(-1,-1)$ \cite[Section 5]{Batyrev:mixed}, 
we know that $\Res(\rsR(t^b t_0^k))$, $t^b t_0^k\in I_\Delta^{(1)}$, 
$k\le p$ generate $F^{n-p}W_{n-1}(H^n(\chY_\alpha))$. 
(One can also see that $\rsR(t^b t_0^k)$ for  
$t^b t_0^k\in I_\Delta^{(1)}$ extends to a holomorphic 
$n$-form on $\chcX\setminus \chcY_\alpha$. 
See the proof of \cite[Theorem 4.4]{Mavlyutov:semiample}.)

By the homogeneity of $\Upsilon_v(q,z)$, we have 
$(-1)^{\deg/2} \Upsilon_v(q,1) = (-1)^{\age(v)} 
\Upsilon(q,-1)$. 
Using this and Theorem \ref{thm:Aperiod=Bperiod}, 
we have  
\begin{align*}
%\label{eq:periodofY} 
%\begin{split}  
\varsigma^* Q_{\rm A}(
\Upsilon_v|_{z=1}, \frs(\iota^*\cO_\cX)) 
& = (-1)^{\age(v)} \varsigma^* \Pi_\cY(\Upsilon_v, \cO_\cY)
\big|_{z=1}  \\
& =   Q_{\rm B}(
\Res(\rsR(\alpha^v t^v t_0^{\age(v)+1})), 
\epsilon_{n-1} \VC(\chT_\R)) 
%\end{split} 
\end{align*} 
for $v\in \Boxop$ and $q$ sufficiently 
close to $\bzero$. 
Here $\VC(\chT_\R) = \chY_\alpha \cap \chT_\R $ 
is what we denote by $\Gamma_\R(\alpha)$ before. 
% On the other hand, Theorem \ref{thm:Aperiod=Bperiod} 
% applied to $\cX$ itself gives 
% \begin{equation} 
% \label{eq:periodofX} 
% \tau^*\Pi_\cX(\theta_v, \frs(\cO_\cX)) = 
% \int_{\chT_\R} e^{-W_\alpha(t)/z} \frac{dt_1}{t_1} 
% \wedge \cdots \frac{dt_n}{t_n}. 
% \end{equation} 
We consider the monodromy transforms of 
the both-hand sides around $q=\bzero$. 
By the last statement in 
Theorem \ref{thm:Aperiod=Bperiod}, 
$\varsigma^*Q_A(\Upsilon_v, \frs(\iota^* \cE))$ is 
monodromy generated by 
$\varsigma^*Q_A(\Upsilon_v, \frs(\iota^*\cO_\cX))$. 
Moreover, $\Mir_\cX$ and the vertical arrows 
in the diagram \eqref{eq:intstr_XandY} 
is equivariant with respect to the monodromy transformation 
around $q=\bzero$ 
(which is the tensor product of line bundles 
on $K(\cX)$; see Definition \ref{def:intstr}, (ii)). 
Therefore, we have for any 
$\cE\in K(\cX)$ and $q$ sufficiently close to $\bzero$, 
\begin{equation} 
\label{eq:periodmatrix}
\varsigma^* Q_{\rm A}(\Upsilon_v|_{z=1}, \frs(\iota^*\cE)) 
= Q_{\rm B} (\Res (\rsR (\alpha^v t^v t_0^{\age(v)+1} )), 
\epsilon_{n-1} \VC( \Gamma_\cE) ) 
\end{equation} 
where $\Gamma_\cE := \Mir_\cX^\Z(\cE)$.  
The A-model VHS $\varsigma^*\rsH_{\rm A}$ is generated 
by $\Upsilon_v|_{z=1}$, $v\in \Boxop$ and 
their covariant derivatives 
near $q=\bzero$ as an $\cO_{\cM_{\rm reg}}$-module  
(see the discussion before Corollary 
\ref{cor:mirrorthm_Y}). 
Likewise we claim that the B-model VHS $\rsH_{\, \rm B}$ 
is generated by 
$\Res(\rsR(\alpha^v t^v t_0^{\age(v)+1}))$, $v\in \Boxop$ 
and their derivatives.  
Take any element 
$t^b t_0^k \in I_\Delta^{(1)}$. 
Take a cone $\sigma\in \Sigma$ such that $b\in \sigma$.  
Then we can write $b = \sum_{i\in \sigma} l_i b_i + v$ 
for some $v\in \Boxop \cap \sigma$ and $l_i\in \Z_{\ge 0}$. 
The piecewise linear function $\varphi_1$ defined in Section 
\ref{subsec:mirrorconst} satisfies $\varphi_1(b_j) =1$ 
for all $1\le j\le m+s$. So we have $\varphi_1(b) = 
\sum_{i\in \sigma} l_i + \age(v)$. 
Since $t^b t_0^k \in I_\Delta^{(1)}$, 
we have $\varphi_1(b)+1 \le k$. 
By a direct calculation, we find 
\[
% \left(\textstyle \sum_{i=1}^{m+s} \alpha_i 
% \nabla_{\alpha_i}^{\rm B}\right )^{k-\varphi_1(b)-1} 
\prod_{i\in \sigma}   
(\nabla^{\rm B}_{\partial/\partial \alpha_i})^{l_i}
\Res\left(\rsR(t^v t_0^{\age(v)+1}) \right) 
=\Res\left(\rsR(t^bt_0^{\varphi_1(b)+1} )\right).   
\]
Using the Euler vector field 
$\tE= \sum_{i=1}^{m+s} \alpha_i (\partial/\partial \alpha_i)$, 
we find  
\[
(\nabla^{\rm B}_{\tE} + k-1) \cdots 
(\nabla^{\rm B}_{\tE} + \varphi_1(b)+1) 
\Res\left(\rsR(t^b t_0^{\varphi_1(b)+1})\right) 
= \Res\left(\rsR(t^b t_0^k)\right). 
\] 
Now the claim follows. 
By taking the derivatives of \eqref{eq:periodmatrix}, 
we know that the full \emph{period matrices} of $\varsigma^*(\rsH_{\rm A}, 
\nabla^{\rm A})$ and $(\rsH_{\,\rm B},\nabla^{\rm B})$ 
are the same. This shows that we have an isomorphism 
$\Mir_\cY \colon \varsigma^*(\rsH_{\rm A}, \nabla^{\rm A}) 
\cong (\rsH_{\, \rm B}, \nabla^{\rm B})$ 
sending $\Upsilon_v(q,1)$ to $\Res(\rsR(t^vt_0^{\age(v)+1}))$. 
Moreover, from the above calculation, 
it turns out that the generators $\Res(\rsR(t^b t_0^k))$, 
$t^bt_0^k\in I_\Delta^{(1)}$, $k\le p$ 
of $F^{n-p} W_{n-1}(H^n(\chY_\alpha)) 
\cong \rsF_{{\rm B},[\alpha]}^{\, n-p}$ 
correspond via $\Mir_\cY$ to elements in 
$(\varsigma^* \rsF_{\rm A}^{\, n-p})_{[\alpha]}  
\cong H^{\le 2(p-1)}_{\rm amb}(\cY)$. 
Hence $\Mir_\cY(\varsigma^* \rsF^{\, n-p}_{\rm A})
\supset \rsF^{\, n-p}_{\rm B}$. 
The other inclusion 
$\Mir_\cY(\varsigma^* \rsF^{\,n-p}_{\rm A}) 
\subset \rsF^{\, n-p}_{\rm B}$ 
can be easily seen by taking a basis of 
$\rsF^{\, n-p}_{\rm A}$ given by the 
covariant derivatives of $\Upsilon_v|_{z=1}$. 
Therefore 
$\varsigma^* (\rsH_{\rm A}, \nabla^{\rm A}, \rsF^{\,\bullet}_{\rm A}) 
\cong (\rsH_{\, \rm B}, \nabla^{\rm B}, \rsF^{\,\bullet}_{\rm B})$. 

Next we show that $\varsigma^* Q_{\rm A} 
= Q_{\rm B}$ under $\Mir_\cY$. 
We know by \eqref{eq:periodmatrix} that the 
``dual" flat sections $Q_{\rm A}(\cdot, \frs(\iota^*\cE))$ 
and $Q_{\rm B}(\cdot, \epsilon_{n-1} \VC(\Gamma_\cE))$ correspond 
to each other under $\Mir_\cY$. Therefore 
it suffices to show that $Q_{\rm A}(\frs(\iota^*\cE_1), 
\frs(\iota^* \cE_2))$ 
and $Q_{\rm B}(\VC(\Gamma_{\cE_1}), \VC(\Gamma_{\cE_2}))$ 
are equal for $\cE_1,\cE_2 \in K(\cX)$. 
We have 
\[
Q_{\rm A}(\frs(\iota^*\cE_1), \frs(\iota^* \cE_2)) 
= \chi_{\cY} (\iota^* \cE_1, \iota^* \cE_2) 
= \chi(\cE_1, \cE_2) - (-1)^n \chi(\cE_2, \cE_1). 
\]
Let $e^{\pi\iu} \Gamma_{\cE_i}$ denote the 
parallel translate of $\Gamma_{\cE_i} 
\in H_n(\chT, \{\Re(W_\alpha(t)/z) \gg 0\})$ 
along the path $[0,1]\ni \theta \mapsto e^{\pi\iu\theta} z$. 
Then we have 
\begin{align*} 
Q_{\rm B}(\VC(\Gamma_{\cE_1}), \VC(\Gamma_{\cE_2}))  
&= \epsilon_{n-1} 
\VC(\Gamma_{\cE_1}) \cdot \VC(\Gamma_{\cE_2}) \\
&= \epsilon_n   
(\Gamma_{\cE_1} \cdot e^{\pi\iu} \Gamma_{\cE_2}
+ (-1)^{n-1} \Gamma_{\cE_2} \cdot e^{\pi\iu} \Gamma_{\cE_1}). 
\end{align*} 
Since $\Mir_\cX$ preserves the pairing\footnote
{We made a sign error in \cite{Iritani:int} 
in matching the pairings under mirror symmetry. 
In \cite[Appendix A.3]{Iritani:int} we 
showed that the A-model and B-model pairings 
differ only by a constant. The constant is fixed by 
comparing $\Gamma_\R \cdot \Gamma_{\rm c}$ 
with $\chi(\cO_\cX, \cO_{\rm pt})=1$, where 
$\Gamma_\R = \chT_\R$ and $\Gamma_{\rm c} 
\cong (S^1)^n$. 
Taking the orientation into account, we find that 
$\Gamma_\R \cdot \Gamma_{\rm c}= (-1)^{n(n-1)/2}$ 
instead of $1$.   
So the B-model pairing should be multiplied by 
$\epsilon_{n}= (-1)^{n(n-1)/2}$ to have the complete match with 
the A-side.}, we have 
\[
\chi(\cE_1,\cE_2) = \epsilon_n 
\Gamma_{\cE_1} \cdot e^{\pi\iu} \Gamma_{\cE_2} 
\]
and $\varsigma^* Q_{\rm A} = Q_{\rm B}$ follows. 
Therefore, $\frs(\iota^*\cE)$ 
corresponds to $\VC(\Gamma_\cE)$ 
under $\Mir_\cY$ when $q$ is sufficiently close to $\bzero$. 
This shows the commutative diagram \eqref{eq:intstr_XandY}. 
\end{proof} 

The class $[\cO_{\rm pt}]\in K(\cY)$ of 
a skyscraper sheaf at a non-stacky point 
on $\cY$ gives a flat section $\frs(\cO_{\rm pt})$ of 
$\rsH_{\rm A}$. Usually $[\cO_{\rm pt}]$ is not 
contained in $\iota^*K(\cX)$, but we can still find 
an integral cycle on $\chcY_\alpha$ corresponding to it. 
\begin{theorem} 
\label{thm:Opt} 
Under the mirror isomorphism $\Mir_\cY$ 
in Theorem \ref{thm:intstr_Batyrevmirror}, 
the flat section $\frs(\cO_{\rm pt})$ corresponds 
to an integral compact $(n-1)$-cycle $C(\alpha) \subset \chY_\alpha$
i.e.\ 
\begin{equation} 
\label{eq:Opt} 
Q_{\rm A}(\phi, \frs(\cO_{\rm pt})) = 
\int_{C(\alpha)} \Mir_\cY(\phi)  
\end{equation} 
for any section $\phi$ of $\varsigma^*\rsH_{\rm A}$. 
In particular, we have  
$(2\pi\iu)^{n-1} F(q) = \int_{C(\alpha)} 
\Omega_\alpha$.  
\end{theorem} 
\begin{proof} 
It suffices to prove \eqref{eq:Opt} for 
$\phi = \Upsilon_v(q,1)$, $v\in \Boxop$. 
We first show that for sufficiently small 
$\alpha\in (\C^\times)^{m+s}$ and $q = [\alpha]$, 
\[
Q_{\rm A}\left(
\Upsilon_v(q,1), \frs(\cO_{\rm pt})(\varsigma(q))
\right) 
= -\frac{1}{2\pi\iu} 
\int_{(S^1)^n} \frac{
(-1)^{\age(v)} \age(v)! 
\alpha^v t^v}{(W_\alpha(t) -1)^{1+\age(v)}} 
\frac{dt_1}{t_1}\wedge \cdots \wedge \frac{dt_n}{t_n} 
\]
where $(S^1)^n=\{|t_1| = \cdots = |t_n|=1\}$. 
By \eqref{eq:Upsilon} and the formula of $I_\cV^v$, 
we find that the left-hand side equals 
\[
(2\pi\iu)^{n-1} 
\sum_{\substack{\sum_{j=1}^{m+s} d_j b_j + v = 0 \\ 
d_j \in \Z_{\ge 0}}} 
\frac{\left(\sum_{j=1}^{m+s} d_j + \age(v)\right)!}{
\prod_{j=1}^{m+s} d_j ! } q^{d+v}.  
\]
On the other hand, by expanding 
$1/(1-W_\alpha(t))$ in geometric series,  
the right-hand side can be calculated as the 
residue at $t=0$. 
This is possible under the assumption that 
$|W_\alpha(t)|<1$ for all $t\in (S^1)^n$ 
(this holds when all $\alpha_j$ are sufficiently small). 
It is easy to see that the residue calculation 
gives the same answer as above. 
We now use the following argument by 
Przyjalkowski \cite[Section 2.5]{Przyjalkowski:weak}. 
For a fixed such $\alpha$, 
we consider the  family of compact tori $(S^1_\epsilon)^{n} 
= \{|t_1| = \cdots = |t_n|=\epsilon\}$ for 
$0<\epsilon\le 1$. 
At $\epsilon =1$, we have $(S^1_1)^n \cap \chY_\alpha = \emptyset$. 
While $\epsilon$ decreases from 1, 
this family of tori slices the hypersurface $\chY_\alpha$. 
For sufficiently small $\epsilon$, $(S^1_\epsilon)^n$ 
does not intersect $\chY_\alpha$ again. Let $0<\delta<1$ be 
such that $(S^1_\epsilon)^n \cap \chY_\alpha =\emptyset$ 
for $\epsilon \le \delta$. 
Then one can use $(S^1_\delta)^n - (S^1_1)^n$ as a 
tube cycle of the slice $C(\alpha) := 
\bigcup_{\delta<\epsilon<1} \chY_\alpha \cap (S^1_\epsilon)^n$.  
We can see that the integral over $(S^1_\delta)^n$ of the 
same integrand tends to zero as $\delta \to 0$ 
since the denominator grows faster than the numerator. 
From this it follows that the integral over $(S^1_\delta)^n$ 
is in fact zero and the right-hand side equals 
$\int_{C(\alpha)} \Mir_\cY(\Upsilon_v(q,1))$. 
The last statement follows from the case $v=0$. 
\end{proof}

\subsection{Multi-GKZ System} 
Batyrev \cite{Batyrev:mixed} 
showed that a rational period of affine hypersurfaces in $\chT$ 
satisfies the Gelfand-Kapranov-Zelevinsky (GKZ)  
hypergeometric differential system \cite{GKZ:hypergeom}. 
Borisov-Horja \cite{Borisov-Horja:better} 
showed that a certain collection of periods 
satisfies a multi-generator version of GKZ system, 
which they called \emph{better behaved GKZ system}. 
Here we see that the residual B-model VHS can 
be realized as a sub $D$-module of the $D$-module 
defined by the multi-GKZ system. 
This is related to the multi-generation phenomenon 
explained in the Introduction and Remark \ref{rem:multigen}. 
In joint work \cite{CCIT:toric} with Coates, Corti and Tseng, 
we also found that the multi-GKZ system arises for 
the quantum $D$-module of a toric orbifold $\cX$ 
\emph{itself}\footnote
{The multi-generation occurs typically for non-compact $\cX$. 
It can also occur for compact $\cX$ which does not 
satisfy the assumption in Section \ref{subsec:toricorb}.}.

Set $\hbN := \bN \oplus \Z$ and 
$\hb_i = (b_i,1)\in \hbN$. 
We still assume $b_i \neq 0$ for all $i$ 
and set $\hb_0 = (0,1)$. 
For simplicity we set $N:=m+s$.  
Let $\hDelta$ be the cone in the vector space 
$\hbN_\R = \hbN\otimes \R$ generated by $\Delta\times \{1\}$. 
Let $\bm_1,\dots,\bm_n$ be a basis of $\bM = \Hom(\bN,\Z)$. 
Let $\bm_0\in \Hom(\hbN,\Z)$ be the projection to the 
second factor. We can then regard $\bm_0,\bm_1,\dots,\bm_n$ 
as a basis of $\hbM :=\Hom(\hbN,\Z)$. 
% In the following we do not need the assumption that 
% $\Delta$ is reflexive. 

\begin{definition}[{\cite{Borisov-Horja:better, CCIT:toric}}] 
The \emph{multi-GKZ system} associated to 
$\{\hb_0,\hb_1,\dots, \hb_N = \hb_{m+s}\}$ 
is the system of differential 
equations for a family 
$\{ \varpi_{\be}(\alpha_0,\dots,\alpha_N)\,|\, 
\be \in \hbN \cap \hDelta\}$ 
of functions on $(\C^\times)^{N+1}$ 
given by $D_{\nu;\be,\be'} = Z_{i,\be}=0$, where 
\begin{align*} 
D_{\nu; \be,\be'} & := \prod_{i=0}^{N} 
\left( \parfrac{}{\alpha_i}\right)^{\nu_{+,i}} 
\varpi_\be  - 
\prod_{i=0}^{N} 
\left( \parfrac{}{\alpha_i} \right)^{\nu_{-,i}} 
\varpi_{\be'} \\
Z_{i,\be} & := \sum_{j=0}^{N} \pair{\bm_i}{\hb_j}
\alpha_j \parfrac{\varpi_\be}{\alpha_j}  
+ (\pair{\bm_i}{\be} -\beta_i)  \varpi_\be, 
\quad 0\le i\le n,  
\end{align*} 
$\nu$ runs through all elements in 
$\Z^{N+1}$ satisfying 
$\be' = \be + \sum_{i=0}^{N} \nu_i \hb_i$ 
and $\nu_+, \nu_-\in (\Z_{\ge 0})^{N+1}$ 
are given by $\nu_{\pm,i} = \max(\pm \nu_i, 0)$
(then $\nu = \nu_+ - \nu_-$). The constants 
$\beta_0,\dots,\beta_n$ are called \emph{exponents}. 
Alternatively, we can regard the multi-GKZ system as a $D$-module on 
$(\C^\times)^{N+1}$ defined by 
\begin{equation} 
\label{eq:multi-GKZ-Dmod}
\bigoplus_{\be\in \hbN \cap \hDelta} 
\cD \varpi_\be 
\biggl/ \sum_{\be,\be',\nu:\, \text{as above}} 
\cD Z_{\nu;\be,\be'} + 
\sum_{1\le i\le N, \be\in \hbN\cap \hDelta} 
\cD Z_{i,\be} 
\end{equation} 
where $\varpi_\be$ here is a formal symbol
and $\cD= \C\left\langle \alpha_0^\pm,\dots,\alpha_N^\pm, 
\partial_{\alpha_0},\dots,\partial_{\alpha_N} \right\rangle$. 
\end{definition} 

It is easy to see that the multi-GKZ system 
is generated by finitely many $\varpi_\be$.  
In this section, for $\alpha=(\alpha_0,\dots,\alpha_N) 
\in (\C^\times)^{N+1}$, we set 
$W_\alpha(t) = \alpha_0 + \sum_{i=1}^N \alpha_i t^{b_i}$  
and $\chY_\alpha := \{t\in \chT\, |\, 
W_\alpha(t)=0\}$. 
Then $W_\alpha(t)$ and $\chY_\alpha$ in the previous 
section are recovered by setting $\alpha_0 = -1$. 
Note that $\chY_{(\alpha_0,\alpha')} = \chY_{-\alpha_0^{-1}\alpha'}$
for $\alpha_0 \in \C^\times$, $\alpha' \in (\C^\times)^N$. 
Take a locally constant section 
$\Gamma(\alpha) \in H_n(\chT,\chY_\alpha;\Z)$ 
of the relative homology bundle over the $\alpha$-space. 
Let $C(\alpha) := \partial \Gamma(\alpha) 
\in H_{n-1}(\chY_\alpha;\Z)$ be its boundary. 
For $\be = (b,k)\in \hDelta \cap \hbN$ we set 
\[
\Pi_\be(\alpha) := 
(-1)^{k-1} (k-1)! \int_{C(\alpha)} 
\Res\left( \frac{t^{b}}{W_\alpha(t)^k} \frac{dt_1}{t_1} 
\wedge \cdots \wedge \frac{dt_n}{t_n} \right) 
\]
for $k>0$ and for $k=0$ (in this case $\be =0$),  
\[
\Pi_0(\alpha) := 
- \int_{\Gamma(\alpha)} \frac{dt_1}{t_1} \wedge 
\cdots \wedge \frac{dt_n}{t_n}. 
\] 
The integrands of these period integrals 
$\Pi_\be(\alpha)$ generate 
the relative cohomology bundle 
$\bigcup_\alpha H^n(\chT, \chY_\alpha)$ 
by the results of Batyrev \cite{Batyrev:mixed} 
and Stienstra \cite{Stienstra:resonant} 
(see also \cite{Konishi-Minabe:local}). 

\begin{proposition} 
The set of functions $\Pi_\be(\alpha)$, $\be \in \hbN \cap \hDelta$ 
satisfies the multi-GKZ system with exponent 
$\beta = (0,0,\dots,0)$.  
\end{proposition} 
\begin{proof} 
Borisov-Horja \cite[Proposition 5.2]{Borisov-Horja:better} 
proved a similar result.  
Stienstra \cite{Stienstra:resonant} proved that  
$\Pi_0$ satisfies the ordinary GKZ system. 
The proposition here can be proved by an easy direct calculation, 
using the formula for $\partial_{\alpha_j} \Pi_0(\alpha)$ of 
Konishi-Minabe \cite[Section 4.3]{Konishi-Minabe:local} 
and the method of Batyrev \cite[Theorem 14.2]{Batyrev:mixed}. 
\end{proof} 

The relative cohomology bundle 
$\bigcup_\alpha H^n(\chT,\chY_\alpha)$ 
has the rank $\Vol(\Delta)$ \eqref{eq:relhom} 
and the multi-GKZ system has the same rank 
by Borisov-Horja \cite[Section 3]{Borisov-Horja:better}. 
Thus they are isomorphic as a local system. 
Because $W_{n-1}(H^{n-1}(\chY_\alpha)) \to 
H^n(\chT,\chY_\alpha)$ is injective (see \eqref{eq:dualincl}), 
the residual B-model VHS is embedded in the multi-GKZ system. 
Note that $\Pi_\be(\alpha)$ is a period of 
the residual B-model VHS if $\be$ is 
in the interior of $\hDelta$ and also that 
the corresponding residue classes on $\chcY_\alpha$ 
generate the B-model VHS. 
Therefore we have the following. 

\begin{theorem} 
\label{thm:CYhyp_multiGKZ} 
Let $\pi\colon (\C^\times)^{N+1} \to \cM$ 
be the map sending $(\alpha_0,\alpha')$ 
to $[-\alpha_0^{-1} \alpha']$ 
and set $(\C^\times)^{N+1}_{\rm reg} 
= \pi^{-1}(\cM_{\rm reg})$. 
Under the pull-back by 
$\pi \colon (\C^\times)^{N+1}_{\rm reg} \to \cM_{\rm reg}$, 
the residual B-model VHS 
is isomorphic to the sub $D$-module of the multi-GKZ system 
\eqref{eq:multi-GKZ-Dmod} 
generated by $\varpi_\be$ such that $\be$ is in the interior 
of the cone $\hDelta$. 
In particular, the ambient A-model VHS is embedded 
in the multi-GKZ system under mirror isomorphism  
of Theorem \ref{thm:intstr_Batyrevmirror}. 
\end{theorem} 

\begin{remark} 
We will see that the multi-GKZ system here 
describes the quantum $D$-module of the total space of $K_\cX$ 
in \cite{CCIT:toric}.  
So $QDM_{\rm amb}(\cY)$ is contained in $QDM(K_\cX)$. 
\end{remark} 

\begin{remark} 
The above identification between the multi-GKZ and 
the relative cohomology 
introduces a mixed Hodge structure on the multi-GKZ system. 
In the context of orbifold mirror symmetry, 
Corti-Golyshev \cite{Corti-Golyshev:wp} 
studied the hypergeometric system 
associated to weighted projective 
Calabi-Yau hypersurfaces 
and its Hodge structure. 
\end{remark} 

\begin{remark}
Mann-Mignon \cite{Mann-Mignon} obtained another (but 
closely related) 
description of the quantum $D$-module of 
a nef complete intersection in a smooth toric variety. 
\end{remark} 

\subsection{Questions and Example} 
In Theorem \ref{thm:intstr_Batyrevmirror}, 
we showed the correspondence between
\emph{vanishing cycles} on $\chcY_\alpha$ 
and \emph{ambient $K$-classes} on $\cY$. 
This match of the integral structures is not 
completely satisfactory. For example, the 
class $[\cO_{\rm pt}]$ on $\cY$ would not 
be contained in $\iota^* K(\cX)$, but the 
corresponding mirror cycle exists (see 
Theorem \ref{thm:Opt}). 
We can also consider different 
integral structures which might be more natural. 
For example, on the A-side, we can take 
the integral structure 
\[
H_{{\rm A},\Z} = \{\frs(\cE)\;|\; \tch(\cE) \in H_{\rm amb}^*(\cY), 
\cE \in K(\cY) \}.  
\]
This is bigger than $H_{{\rm A},\Z}^{\rm amb}$ 
in general. On the B-side we could consider 
the integral structure 
$W_{n-1}(H^{n-1}(\chY_\alpha)) \cap 
H^{n-1}(\chY_\alpha;\Z)$. 
When we can choose a smooth compactification $\chcY_\alpha$, 
we could also take the integral structure 
$H_{\rm res}^{n-1}(\chcY_\alpha)
\cap H^{n-1}(\chcY_\alpha;\Z)$. 
\begin{question} 
What is the integral structure in the B-model 
corresponding to $H_{{\rm A},\Z}$? 
\end{question} 
Yongbin Ruan asked the following question 
to the author. 
\begin{question} 
What is the ``correct" definition of the 
integral middle homology group of the orbifold 
$\chcY_\alpha$ in this context? 
\end{question} 
Mirror symmetry for the orbifold Hodge numbers of 
toric Calabi-Yau hypersurfaces 
was studied by Borisov-Mavlyutov \cite{Borisov-Mavlyutov}. 
\begin{question} 
Can one extend the mirror isomorphism of VHS  
beyond the ambient part/residue part? 
Then what is the integral structure in the B-model VHS 
on the full orbifold cohomology of $\chcY_\alpha$?  
\end{question} 

\begin{example} 
We first consider the simplest example of 
an elliptic curve $\cY$ in $\Proj^2$. 
The mirror is defined by $W_\alpha(t_1,t_2) = 
\alpha_1 t_1 +\alpha_2 t_2 + \alpha_3 (t_1t_2)^{-1}$. 
The $\C^\times$-co-ordinate $q$ on $\cM \cong \C^\times$ 
is given by $q= \alpha_1 \alpha_2 \alpha_3$. 
We choose a section $(\alpha_1,\alpha_2,\alpha_3 ) 
 = (1,1,q)$ and work 
with $W_q(t) = t_1 + t_2 + q (t_1t_2)^{-1}$. 
The mirror hypersurface $\chY_q = \{W_q(t)=1\}$ is 
an elliptic curve minus 3 points.  
The function $W_q(t)$ has the three critical values 
$3q^{1/3}$, $3 \omega q^{1/3}$, $3\omega^2 q^{1/3}$, 
where $\omega = e^{2\pi\iu/3}$. 
Let $q>0$ be small and take vanishing paths from 
the three critical values to 1 
as shown in Figure \ref{fig:P2}. 
\begin{figure}[htbp]  
\begin{center}
\begin{picture}(300,60) 
\put(50,50){\makebox(0,0){$\bullet$}}
\put(50,00){\makebox(0,0){$\bullet$}} 
\put(93,25){\makebox(0,0){$\bullet$}}
\put(250,25){\makebox(0,0){$\bullet$}} 
\path(50,50)(250,25) 
\path(93,25)(250,25)
\path(50,0)(250,25) 
%\qbezier(50,50)(250,50)(250,25) 
%\path(93,25)(250,25)
%\qbezier(50,0)(250,0)(250,25) 
\put(35,50){\makebox(0,0){$3 \omega q^{\frac{1}{3}}$}} 
\put(33,0){\makebox(0,0){$3\omega^2 q^{\frac{1}{3}}$}} 
\put(83,25){\makebox(0,0){$3 q^{\frac{1}{3}}$}} 
\put(257,25){\makebox(0,0){$1$}}
\put(300,25){\makebox(0,0){(smooth fiber)}}
\put(160,45){\makebox(0,0){$\cO(-1)$}}
\put(145,30){\makebox(0,0){$\cO$}}
\put(155,5){\makebox(0,0){$\cO(1)$}}
% \put(220,50){\makebox(0,0){$\cO(-1)$}}
% \put(195,30){\makebox(0,0){$\cO$}}
% \put(220,0){\makebox(0,0){$\cO(1)$}}
\end{picture}
\end{center} 
\caption{Vanishing Paths}
\label{fig:P2} 
\end{figure}
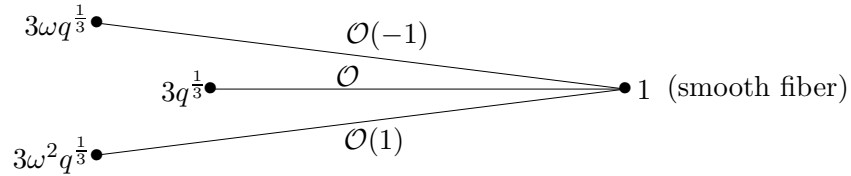 
The Lefschetz thimble $\Gamma_i$ 
from $3\omega^{-i} q^{1/3}$ ($i=-1,0,1$)  
along the given vanishing path 
corresponds to the line bundle $\cO(i)$ on $\Proj^2$. 
The vanishing cycle $C_i = 
\partial \Gamma_i \subset \chY_q$ 
corresponds to $\iota^*\cO(i)$ under 
$\Mir_\cY^\Z$. 
For a suitable symplectic basis $\{A, B\}$ 
of $H_1(\ov{\chY}_\alpha,\Z)$, we have 
\begin{align*} 
C_{\pm 1} = A \pm 3B, \quad C_0 = A. 
\end{align*} 
Similarly for a basis $\{\cO_{\rm pt}, \cO_{\cY}\}$ 
of the topological $K$-group 
$K(\cY)$ of $\cY$, we have 
\[
\iota^*\cO(\pm 1) = \cO_{\cY} \pm 3 \cO_{\rm pt}, \quad 
\iota^* \cO = \cO_{\cY}    
\]
in $K(\cY)$. 
Therefore $\Mir_\cY^\Z$ extends to
an isomorphism of the overlattices in this case. 
\[
H_{{\rm A},\Z} \cong K(\cY) 
\cong H_1(\ov{\chY}_\alpha;\Z).  
\] 
\end{example} 
\begin{example} 
Next we consider a quintic threefold $\cY$ in $\Proj^4$, 
famous example studied in \cite{CdOGP}. 
The mirror of $\cY$ is defined by the function 
$W_q(x_1,\dots, x_4) = x_1 + x_2 + x_3 + x_4 + x_5 + 
q/(x_1x_2 x_3 x_4 x_5)$ 
with one complex parameter $q\in \C^\times$. 
The affine hypersurface $\chY_q = \{z\in \chT\,|\,W_q(x) =1\}$ 
can be compactified to a smooth Calabi-Yau manifold $\chcY_q$. 
In this case, the ambient A-model VHS of $\cY$ 
is a flat bundle of rank 4 with fiber 
$\bigoplus_{p=0}^3 H^{2p}(\cY) = \bigoplus_{p=0}^3 H^{p,p}(\cY)$ 
and the residual B-model VHS of $\chcY_q$ 
is a flat bundle with fiber $H^3(\chcY_q)$. 
We shall show that the isomorphism of 
the $\Z$-structures in Theorem \ref{thm:intstr_Batyrevmirror} 
\[
\Mir_\cY^\Z \colon 
H_{\rm A,\Z}^{\rm amb} 
= \iota^* K(\cX)  \overset{\cong}{\longrightarrow} 
H_{\rm B,\Z}^{\rm vc} \subset 
H_3(\chcY_q;\Z) 
\]
extends to an isomorphism $K(\cY) \cong H_3(\chcY_q;\Z)$
of the overlattices. 
(Here as usual $K(\cY)$ denotes the topological $K$-group.) 
By the Atiyah-Hirzebruch spectral sequence, we know that 
$K(\cY)$ is a free $\Z$-module generated by 
$\cO_{\rm pt}$, $\cO_C$, $\cO_D$, $\cO_\cY$ where 
$C\subset \cY$ is a line and $D = \cY \cap \Proj^3$ 
is a hyperplane section. 
Under the isomorphism $K(\cY) \otimes \Q 
\cong H_3(\chcY_q;\Q)$ of rational vector spaces, 
the dimension filtration 
$K_{\le 0} \subset K_{\le 1} \subset K_{\le 2} \subset 
K_{\le 3} = K(\cY)$ induces the filtration 
$W_0 \subset W_1 \subset W_2 \subset W_3 = H_3(\chcY_q;\Q)$. 
For $\cE \in K(\cY)$, we denote the corresponding 
element in $H_3(\chcY_q;\Q)$ by the same symbol. 
We set $W_i^\Z = W_i \cap H_3(\chcY_q;\Z)$. 
We have $W_0 = \Q [\cO_{\rm pt}]$. 
For $\alpha \in W_3^\Z$, we have 
$\alpha - (\alpha\cdot [\cO_{\rm pt}]) [\cO_\cY] 
\in W_0^\perp = W_2$. Here the intersection 
number $\alpha\cdot [\cO_{\rm pt}]$ is an 
integer since $[\cO_{\rm pt}]\in H_3(\chcY_q;\Z)$ 
by Theorem \ref{thm:Opt}. 
This shows that 
\[
W_3^\Z = W_2^\Z + \Z [\cO_\cY]. 
\]
It is easy to see that the perfect intersection 
pairing on $W_3^\Z$ induces a perfect pairing 
on $W_2^\Z/W_0^\Z$. We know that $[\cO_D] = 
[\cO_\cY -\cO_\cY(-1)]$ 
and $[\cO_D]^2 = 5 ([\cO_C]-2 [\cO_{\rm pt}])$ 
belong to $\iota^*K(\cX)$ 
and thus to $H_3(\chcY_q,\Z)$. 
They also form a rational basis of $W_2/W_0$. 
Take an element 
$a [\cO_D] + b [\cO_D]^2 \in W_2^\Z/W_0^\Z$. 
By taking the tensor product with $\cO_\cY(-1)$ 
which corresponds to the monodromy transformation 
around $q=0$, we have $a [\cO_D]^2 \in W_2^\Z/W_0^\Z$. 
Hence $(a [\cO_D] + b[\cO_D]^2) \cdot (a [\cO_D]^2) 
= 5 a^2 \in \Z$. Thus $a\in \Z$. 
Therefore 
\[
W_2^\Z = W_1^\Z + \Z [\cO_D]. 
\]
Moreover the perfectness of the pairing on 
$W_2^\Z/W_0^\Z$ implies that $[\cO_C] = [\cO_D]^2/5
\in W_2^\Z/W_0^\Z$ and that $W_2^\Z/W_0^\Z 
= \Z [\cO_C] \oplus \Z [\cO_D]$. 
By pairing with $[\cO_\cY]$ again 
we know that $W_0^\Z = \Z [\cO_{\rm pt}]$. 
These show that $K(\cY) \cong W_3^\Z = 
H_3(\chcY_q;\Z)$. 
\end{example} 

\begin{remark} 
Hartmann \cite{Hartmann:period} studied Hodge-theoretic 
mirror symmetry for a quartic K3 surface and identified 
the mirror periods with certain hypergeometric functions. 
\end{remark}

\bibliographystyle{plain}
\bibliography{qcohperiod.bib} 

\end{document}